\documentclass[reqno,11pt,oneside]{amsart}
\usepackage[colorlinks, citecolor=magenta,linkcolor=blue]{hyperref}
\usepackage{amssymb}
\usepackage{geometry}  
\usepackage{physics}
\usepackage{enumerate}
\usepackage{array}
\usepackage{tikz}
\usetikzlibrary{arrows.meta}

\theoremstyle{plain}
\newtheorem{theorem}{Theorem}[section]

\theoremstyle{definition}
\newtheorem{definition}[theorem]{Definition}
\newtheorem{proposition}[theorem]{Proposition}
\newtheorem{remark}[theorem]{Remark}
\newtheorem{corollary}[theorem]{Corollary}
\newtheorem{lemma}[theorem]{Lemma}
\newtheorem{example}[theorem]{Example}

\renewenvironment{proof}{{\noindent \bf  Proof.}}{\qed}

\newcommand{\Rp}{\mathbb R_+}
\newcommand{\DCs}[1]{\partial_0^{#1}}
\newcommand{\DRL}[1]{D_0^{#1}}
\newcommand{\DCp}[1]{D_{\hspace{-2pt}*0}^{#1}}

\makeatletter
\let\orgdescriptionlabel\descriptionlabel
\renewcommand*{\descriptionlabel}[1]{%
  \let\orglabel\label
  \let\label\@gobble
  \phantomsection
  \edef\@currentlabel{#1}%
  \let\label\orglabel
  \orgdescriptionlabel{#1}%
}
\makeatother

\allowdisplaybreaks

\numberwithin{equation}{section}
\begin{document}
\title[Censored stable subordinators and  fractional   derivatives]{Censored stable subordinators and  fractional   derivatives}

\author[Qiang Du]{Qiang Du\hspace{2pt}}
\address{Qiang Du
\newline Department of Applied Physics and Applied Mathematics, and Data Science Institute
\newline Columbia University
\newline New York 10027, USA}
\email {qd2125@columbia.edu}

\author[Lorenzo Toniazzi]{\hspace{2pt}Lorenzo Toniazzi\hspace{2pt}}
\address{Lorenzo Toniazzi
\newline Department of Mathematics and Statistics
\newline Otago University
\newline Dunedin 9054, NZ}
\email {ltoniazzi@maths.otago.ac.nz}

\author[Zirui Xu]{\hspace{2pt}Zirui Xu}
\address{Zirui Xu
\newline Department of Applied Physics and Applied Mathematics
\newline Columbia University
\newline New York 10027, USA}
\email {zx2250@columbia.edu}


\thanks{Research supported in part by  NSF DMS-2012562,  the ARO MURI Grant W911NF-15-1-0562 and the Marsden Fund administered by the Royal Society of New Zealand.\\
\indent The following statement is added upon the request of the publisher: ``This paper is published (in revised form) in
\emph{Fract. Calc. Appl. Anal.} Vol. \textbf{24}, No 4 (2021), pp. 1035--1068, DOI: 10.1515/fca-2021-0045, and is available online at \href{https://www.degruyter.com/journal/key/FCA/html}{https://www.degruyter.com/journal/key/FCA/html}, so please always cite it with the journal's coordinates".}
\subjclass[2010]{26A33, 60G52, 60G40}


\begin{abstract}
Based on the popular Caputo fractional derivative of order $\beta$ in $(0,1)$, we define the censored fractional derivative on 
the positive half-line $\Rp$. This derivative proves to be the Feller generator of the censored (or resurrected) decreasing $\beta$-stable   process in $\Rp$. We provide a series representation for the inverse of this censored fractional derivative,  which we use to study general censored initial value problems. 
We are then able to prove that this censored process  hits the boundary in a finite time $\tau_\infty$, whose expectation is proportional to that of the
first passage time of the $\beta$-stable subordinator. We also show that the censored relaxation equation is solved by the Laplace transform of $\tau_\infty$.
This relaxation solution proves to be a completely monotone series, with algebraic decay 
one order faster than its Caputo counterpart, leading, surprisingly, to a new regime of fractional relaxation models.
   Lastly, we discuss how this work identifies a new  sub-diffusion model. 
\end{abstract}

\keywords{Fractional initial value problem, censored stable subordinator, fractional relaxation equation, Mittag-Leffler function}

\maketitle

\tableofcontents

\section{Introduction}\label{Section: Introduction}
Fractional derivatives, a special class of nonlocal integral and pseudo-differential operators \cite{du19,Grubb17,Uma15,Ja05}, have been successfully employed to model heterogeneities and nonlocal interactions in many applications (see, e.g., \cite{whatis,MK04,pod,mainardi}). They also enjoy an interesting 
mathematical theory  with deep connections to L\'evy processes (see, e.g., \cite{Meerschaert2012,Bogdan,kilbas,Kol19,Kwa17}). For example, the Caputo derivative \cite{kai} of order  $\beta\in(0,1)$  on the positive half-line $\mathbb R^+$, plays  important roles in modelling non-exponential relaxation \cite{kai,MK04} and  non-Markovian sub-diffusive  dynamics \cite{MS08,BC11,hai18}. For a smooth function $u$ vanishing  outside $\Rp$, the Caputo derivative equals the Riemann–Liouville (R–L) derivative $\DRL\beta$ \cite{kai} given by
\begin{align}
\DRL\beta u(x)&
=\int_0^x \big(u(x)-u(x-r)\big)\frac{r^{-1-\beta}}{\big|\Gamma(-\beta)\big|}\,{\rm d}r+u(x)\frac{x^{-\beta}}{\Gamma(1-\beta)},\quad x>0.
\label{eq:Cintro}
\end{align}
Probabilistically, $-\DRL\beta$  generates a killed L\'evy process, which is the decreasing $\beta$-stable process $S^1=\{S^1_s\}_{s\ge0}$ killed at time $\tau_1$, the first exit time from $\Rp$ \cite{BDP11,JK02}. Intuitively, the first summand in \eqref{eq:Cintro} describes the decreasing $\beta$-stable   jumps landing inside $\Rp$, while  $x^{-\beta}/\Gamma(1-\beta)=\int_x^\infty r^{-1-\beta}\big/\big|\Gamma(-\beta)\big|\,{\rm d}r$ is the killing coefficient for the jumps landing outside   $\Rp$. In this work, we introduce what we call the \emph{censored fractional derivative} $\DCs\beta$, allowing the representation
\begin{equation}
\label{r^-1-beta representation of D_x^beta}
\DCs\beta u(x)=\int_0^x \big(u(x)-u(x-r)\big)\frac{r^{-1-\beta}}{\big|\Gamma(-\beta)\big|}\,{\rm d}r,\quad x>0.
\end{equation}
It is intuitively clear that $-\DCs\beta$ only allows the decreasing $\beta$-stable jumps to land inside $\Rp$, and suppresses
those landing outside $\Rp$. Indeed we prove that it is the (Feller) generator of $S^c=\{S^c_s\}_{s\ge0}$, the \emph{censored} decreasing $\beta$-stable   process in $\Rp$. 
We will construct  $S^c$ by repeatedly resurrecting in situ the killed decreasing $\beta$-stable process, following the canonical Ikeda–Nagasawa–Watanabe (INW) piecing together procedure \cite{INW66}.
  (Cf. \cite[Remark 3.3]{kyp14} for two other notions of ``censoring" a process.) 

We initiate the study of the censored fractional derivative, and  then apply its theory to derive several new and non-trivial results about the censored stable subordinator, as we now explain. 
We first prove the  well-posedness of the basic initial value problem (IVP)
 \begin{equation}
\left\{
\begin{aligned}
\DCs\beta u(x)&=g(x), &&x\in(0,T],\\
u(x)&=u_0, &&x=0,
\end{aligned}\right.
\label{eq:main_intro}
\end{equation}
for  any $T>0,\;u_0\in\mathbb R$ and certain $g\in C(0,T]$. Our proof is based on constructing the candidate solution $u=u_0+I^\beta_0g$, where $I^\beta_0$ allows a probabilistic series representation and the expected potential representation, namely
\begin{align}\label{eq:serintro}
I^\beta_0 g(x)&= J^\beta_0g(x)+\sum_{j=1}^{\infty}\mathbb E_x\Big[ J^\beta_0 g(X_j)\Big]\\ \label{eq:fkintro}
&=\mathbb E_x\bigg[\int_0^{\tau_\infty } g(S^c_s)\,{\rm d}s\bigg].
\end{align}
Here,  $J^\beta_0$  is the R–L integral, i.e. the inverse of $\DRL\beta$, given by
\begin{equation}\label{eq:Jbetaintro}
    J^\beta_0g(x)=\int_0^xg(y)\frac{(x-y)^{\beta-1}}{\Gamma(\beta)}\,{\rm d}y =\mathbb E_x\bigg[\int_0^{\tau_1}g(S^1_s)\,{\rm d}s\bigg],
\end{equation}
where the second identity is the known potential representation for $J^\beta_0$; the discrete-time process $X\,|\,X_0\!=\!x$ is defined as $X_j:= x\prod_{i=1}^j B_i$, where $\{B_i \}_{i\in\mathbb N}$ is an \textit{i.i.d.} collection of beta-distributed random variables with parameters $1-\beta$ and $\beta$; and $\tau_\infty $ is the lifetime of $S^c$.
The equivalence of \eqref{eq:serintro} and \eqref{eq:fkintro} is due to the equality in law between $X_j$ and $S^c$ at its $j$-th resurrection time, combined with the   second identity in \eqref{eq:Jbetaintro} (see Remark \ref{rmk:sereqfk} for more details).  The way we solve \eqref{eq:main_intro} is to regard it as a linear R–L IVP   $\DRL\beta u=k\hspace{0.6pt}u+g,\;u(0)=0$    
with the  coefficient $k(x)=x^{-\beta}/\Gamma(1\hspace{-0.7pt}-\hspace{-0.7pt}\beta)$. It turns out that the formula given in \cite[Theorem 7.10]{kai} for bounded $k$ still converges for this specific unbounded $k$, allowing us to construct the solution. (As for more general $k$ that  may diverge as $O(x^{-\beta})$, \cite[Example 3.4]{LPV15} gave a non-constructive proof of the existence result.) This explicit solution then allows us to establish the (global) well-posedness of general IVPs $\DCs\beta u = f(x, u)$, $u(0)=u_0$, for certain Lipschitz data $f$.

Using the results above, we are able to solve the linear IVP   $\DCs\beta u = \lambda u$, $u(0)=u_0$, for  any $\lambda \in\mathbb R$. We obtain the Mittag-Leffler-type representation for its solution
\begin{equation}\label{solution to eigenvalue problem}
u(x) = u_0\sum_{N=0}^\infty\lambda ^Nx^{\beta N}\prod_{n=1}^N\bigg(\frac{\Gamma(1+n\beta)}{\Gamma(n\beta\hspace{-0.7pt}+\hspace{-0.7pt}1\hspace{-0.7pt}-\hspace{-0.7pt}\beta)}-\frac1{\Gamma(1\hspace{-0.7pt}-\hspace{-0.7pt}\beta)}\bigg)^{-1},
\end{equation}
where an empty product equals 1 by convention (also, $u(x)=u_0$ if $\lambda=0$) and each factor  of the indexed product is positive by \eqref{inequality for gamma functions}. Surprisingly, for $\lambda<0$, this solution decays at the fast algebraic rate  $x^{-1-\beta}$ (Theorem \ref{thm:decay}), which we believe is a new regime for fractional relaxation models. Indeed the Caputo fractional relaxation solution $u_0E_\beta(\lambda x^\beta)$ decays at the rate $x^{-\beta}$ \cite[Theorem 7.3]{kai}, where   $E_\beta(x )=\sum_{n=0}^\infty x^n/\Gamma(n\beta\hspace{-0.7pt}+\hspace{-0.7pt}1)$ is the Mittag-Leffler function. Moreover, the lagging and leading coupled fractional relaxation equations in \cite{BMS04,WSJMS12} model the decay rate $x^{-\gamma}$ for some $\gamma\in(0,1)$. Our proof (inspired by  \cite[Theorem 3.2]{DYZ17}) is based on maximum principle and turns out to be versatile, albeit elementary.
Indeed the same argument proves the decay rate $x^{-1\hspace{-0.7pt}-\alpha}$ of the solution to $\DCs\beta u = \lambda x^{\alpha-\beta}u\;(\lambda<0 ,\;\alpha>0)$  (see Proposition \ref{thm:decay 1+alpha}), which is again one order faster than its Caputo counterpart (expressed by the Kilbas–Saigo function \cite{de2014fractional}). Moreover, we will show how to adapt this argument to the Caputo setting to give  new and simple proofs of the two-sided uniform bounds of $E_\beta$ and more generally, a class of Kilbas–Saigo functions,   which are the recent results in \cite[Theorem 4]{simon2014comparing} and \cite[Proposition 4.12]{BSV19}, respectively. 
This very argument may have even broader applications, e.g., in general Caputo-type relaxation problems (corresponding to general killed subordinators), as we discuss in Remark \ref{rmk:caprelax}-(iii). 

As a  special case of \eqref{eq:fkintro},  we   have the identity
\begin{equation}
\mathbb E_x[\tau_\infty]=\mathbb E_x[\tau_1] \frac{\beta\pi}{\beta\pi\hspace{-0.7pt}-\hspace{-0.7pt}\sin(\beta\pi)},\;\;\text{where}\;\;\mathbb E_x[\tau_1]=\frac{x^\beta}{\Gamma(\beta\hspace{-0.7pt}+\hspace{-0.7pt}1)},
 \label{eq:ltintro}
\end{equation}
which implies that $S^c$ hits 0 in finite time, a fact that we believe has not been shown before. This is fundamental and not obvious, especially in view of \cite[Theorem 1.1-(1)]{BBC03}, which proves that the censored \emph{symmetric} $\beta$-stable L\'evy process never hits the boundary, whether censored in an interval or $\Rp$. (Also, censored decreasing compound Poisson processes do not hit the barrier in finite time, and our numerical simulations suggest neither do censored  gamma subordinators.) 
We are then able to show several more connections between the analytic and probabilistic aspects of $\DCs\beta$. That is, we will prove that $S^c$ is indeed a Feller process generated by $-\DCs\beta$, and that the exit problem for $\tau_\infty$ is solved by \eqref{solution to eigenvalue problem}, i.e.
\begin{equation}\label{eq:exptauinf_intro}
    u_0\hspace{0.7pt}\mathbb E_x\hspace{-0.6pt}\big[\hspace{-0.8pt}\exp\{\lambda \tau_\infty\}\big]\text{ equals the  series }\eqref{solution to eigenvalue problem},\;\text{for all $x>0\;\text{and }\lambda\in\mathbb R$}.
\end{equation}
As a consequence of \eqref{eq:exptauinf_intro}, we can obtain all the moments of $\tau_\infty$ and confirm the complete monotonicity of \eqref{solution to eigenvalue problem}. We emphasise that \eqref{eq:exptauinf_intro} is significantly harder to prove than Caputo's counterpart $ \mathbb E_x\hspace{-0.6pt}\big[\hspace{-0.8pt}\exp\{\lambda \tau_1\}\big]=E_\beta(\lambda x^\beta)$. This is mainly due to the inapplicability of Laplace transforms to $S^c$ and the complexity of the coefficients in \eqref{solution to eigenvalue problem} (see Remark \ref{rmk:exptau} for more detail).  Nonetheless, we obtain a proof by combining our series solution to the resolvent equation $\DCs\beta u=\lambda u+g$ with a simple semigroup theory argument, following  \cite[Corollary  5.1]{HK16}. We could alternatively try combining our IVP theory with standard potential theory (see, e.g., \cite[Chapter 3]{CZ12}) appplied to the Feynman–Kac semigroup of $-\DRL\beta+(\lambda\hspace{-0.7pt}+\hspace{-0.7pt}k)$, but it would be more involved. (We also remark that
\eqref{eq:exptauinf_intro}   serves as an efficient alternative to numerically compute \eqref{solution to eigenvalue problem} for $\lambda<0$.)

Lastly, we discuss  how this work sets the foundations for the study of a new time-fractional diffusion equation $\DCs\beta  u =\Delta u/2$, 
solved by the process $\{B_{\tau_\infty(t)}\}_{t\ge0}$. (Here $\DCs\beta$ acts on the time variable, $B$ is a Brownian motion independent of  $\tau_\infty(t):=\tau_\infty\,|\,S^c_0\!=\!t$.) This is the censored analogue of the Caputo time-fractional diffusion equation $\DRL\beta \big[u\hspace{-0.7pt}-\hspace{-0.7pt}u(0)\big] =\Delta u/2$, which is solved by the fractional kinetic process $\{B_{\tau_1(t)}\}_{t\ge0}$ (with $B$ independent of the inverse stable subordinator $\tau_1(t):=\inf\{s: t\!<\!-S_s^1\}$), a non-Markovian sub-diffusion process  arising from several central limit theorems \cite{MS08,BC11,hai18}. As we discuss in Remark \ref{rmk:Btau}-(i), although both $B_{\tau_1}$ and $B_{\tau_\infty}$ are sub-diffusion processes (due to \eqref{eq:ltintro}), their respective characteristic functions, $E_\beta(\lambda t^\beta)$ and \eqref{solution to eigenvalue problem} (for some $\lambda<0$), display strikingly different decay rates (due to our results on relaxation solutions).
 
This work is organized as follows: Section \ref{sec:preliminaries} introduces notation, recalls basic results on fractional calculus, defines the censored fractional derivative, and studies the solution kernels; Section \ref{sec:wellposedness} focuses on the well-posedness and series representation of the solution to  \eqref{eq:main_intro}, then addresses linear and non-linear censored IVPs; in Section \ref{sec:hit_0} we construct the censored decreasing $\beta$-stable process and apply our IVP theory to its study.

\section{Preliminary notation and definitions}\label{sec:preliminaries}
Throughout this article, we denote by $\beta$ ($0<\beta<1$) the order of fractional derivatives, and by $[0,T]$ ($0<T<\infty$) the interval of interest. We denote by $\mathbb N$, $\mathbb R$ and $\Rp$ the sets of positive integers, real numbers and positive numbers, respectively. For any interval $\Omega\subseteq\mathbb R$ we denote by $C(\Omega),C^1(\Omega),C^{0,\beta}(\Omega)$ and $L^1 (\Omega)$ the real functions on $\Omega$ that are continuous, continuously differentiable, $\beta$-H\"older continuous and Lebesgue integrable, respectively. We abbreviate $C(\Omega)\cap L^1(\Omega)$ to $C\cap L^1(\Omega)$.  For compact   $\Omega$  we denote by $\Vert\cdot\Vert_{C(\Omega)}$ the sup norm. We denote by $\Gamma$ the gamma function and frequently use without mention the standard identities  $\Gamma(2\hspace{-0.7pt}-\hspace{-0.7pt}\alpha)=(1\hspace{-0.7pt}-\hspace{-0.7pt}\alpha)\Gamma(1\hspace{-0.7pt}-\hspace{-0.7pt}\alpha)$ for all $\alpha\in\mathbb R\backslash\mathbb N$, $\Gamma(\beta\hspace{-0.7pt}+\hspace{-0.7pt}1)\Gamma(1\hspace{-0.7pt}-\hspace{-0.7pt}\beta)=\beta\pi/\sin(\beta\pi)$ and
\begin{equation*}
\int_0^x(x-r)^{\gamma-1}r^{\alpha-1}\,{\rm d}r = x^{\gamma+\alpha-1} \frac{\Gamma(\gamma)\Gamma(\alpha)}{\Gamma(\gamma+\alpha)}\;\text{for all}\;\alpha,\,\gamma,\,x>0.
\end{equation*}
We also rely crucially on the inequality (which we prove in Lemma \ref{thm:kseries})
\begin{equation}
\label{inequality for gamma functions}
\Gamma(\alpha\hspace{-0.7pt}+\hspace{-0.7pt}1\hspace{-0.7pt}-\hspace{-0.7pt}\beta)<\Gamma(1\hspace{-0.7pt}+\hspace{-0.7pt}\alpha)\Gamma(1\hspace{-0.7pt}-\hspace{-0.7pt}\beta)\;\text{for all}\;\alpha>0.
\end{equation}

\subsection{R–L calculus and fractional function spaces}

We present some basic results about the R–L fractional derivative, and proofs are given to make our presentation self-contained. We refer to \cite{kai} for a general study of Caputo/R–L derivatives.

\begin{definition}
\label{define R-L integral} 
For $\beta\in(0,1),\,u\in C\cap L^1(0,T]$, we define the \emph{R–L integral}
\begin{align*}
J^\beta_0 u(x)&=\int_0^x\frac{(x\hspace{-0.7pt}-\hspace{-0.7pt}r)^{\beta-1}}{\Gamma(\beta)} u(r)\,{\rm d}r,\quad x\in(0,T],\\
\intertext{we define the function spaces}
C_\beta(0,T]&=\Big\{u\in C\cap L^1(0,T]:\,J^{1-\beta}_0u\in C^1(0,T]\Big\},\\
 C_\beta[0,T] &= C[0,T]\cap C_\beta(0,T],\\
\intertext{and for $u\in C_\beta(0,T],\,x\in(0,T]$, we define the \emph{R–L derivative}}
\DRL\beta u(x)&=\frac{\rm d}{{\rm d}x}J^{1-\beta}_0u(x)=\frac{\rm d}{{\rm d}x}\int_0^x\frac{(x\hspace{-0.7pt}-\hspace{-0.7pt}r)^{-\beta}}{\Gamma(1\hspace{-0.7pt}-\hspace{-0.7pt}\beta)}u(r)\,{\rm d}r.
\end{align*}

\end{definition}

\begin{remark}$ $ \begin{enumerate}[(i)] \item Note that $C_\beta(0,T]$ is chosen so that the image of $\DRL\beta$ is contained in $C(0,T]$. Moreover, $C_\beta[0,T]$ is chosen to be the solution space, as we will explain in Remark \ref{strong solution space}. 
 
 \item Note that $C_\beta[0,T]$ is not a subspace of $C^{0,\beta}[0,T]$. Indeed $x^\alpha$ is in $C_\beta[0,T]$ for all $\alpha \ge 0$, but  in  $C^{0,\beta}[0,T]$ only  if $\alpha\geq \beta$.
\end{enumerate}
 
\end{remark}

\begin{lemma}\label{lem:prelim}
The following relations between $\DRL\beta$ and $J_0^\beta$ hold (proved in Appendix \ref{app:proofofelemprop}).

\begin{enumerate}[(i)]
	\item If $u\in C\cap L^1(0,T]$, then $J^\beta_0 u\in C\cap L^1(0,T]$.
	
	\item If $u\in C\cap L^1(0,T]$, then $J^\beta_0 u\in C_\beta(0,T]$ and $\DRL\beta J^\beta_0u=u$.
	
	\item  Assume $g\in C\cap L^1(0,T]$. Then
	\begin{equation*}
	u=J^\beta_0g \quad
		 \text{if and only if}
		\quad
	\left\{
	\begin{aligned}
	&u\in C_\beta(0,T],\\
	&\DRL\beta u=g,\\
	&\lim_{x\rightarrow0}J^{1-\beta}_0u(x)=0.
	\end{aligned}
	\right.
	\end{equation*}

	\item If $u\in C_\beta[0,T]$  satisfies $\DRL\beta u=0$, then $u=0$. 

\end{enumerate}
\end{lemma}

\begin{remark}
Note that in Lemma \ref{lem:prelim}-(iv), the condition $u\in C_\beta[0,T]$ cannot be weakened to $u\in C_\beta(0,T]$, since $\DRL\beta x^{\beta-1}$ is also 0.
\end{remark}

\begin{remark}
\label{rmk: Holder regularity of J0beta g}
For $g\in L^1(0,T]$ satisfying $\big\vert g(x)\big\vert\le Mx^{\alpha-\beta}$ for some $\alpha,\,M\ge0$ and all $x\in(0,T]$, the H\"older regularity of $J_0^\beta g$ is summarized as follows. See Lemmata \ref{beta Holder regularity of J0beta singular g} and \ref{alpha Holder regularity of J0beta singular g} for the proofs.
\begin{enumerate}[(i)]
\item For any $T_1\in(0,T)$, $J_0^\beta g \in C^{0,\beta}[T_1,T]$ with a H\"older constant  $2M\max\{T_1^{\alpha-\beta}\!,\,T^{\alpha-\beta}\}/$ $\Gamma(1+\beta)$.
\item If $ 0<\alpha<\beta$, then $J_0^\beta g\in C^{0,\alpha}[0,T]$ with a H\"older constant $2M\Gamma(\alpha\hspace{-0.7pt}+\hspace{-0.7pt}1\hspace{-0.7pt}-\hspace{-0.7pt}\beta)/\Gamma(1\hspace{-0.7pt}+\hspace{-0.7pt}\alpha)$.
\end{enumerate}
\end{remark}

\subsection{Censored fractional derivative}
We now define our censored fractional derivative.
\begin{definition}\label{def:cfd}
Given $\beta\in(0,1)$, we define the \emph{censored fractional derivative} of any $u\in C_\beta(0,T]$ as
\begin{equation*}
\DCs\beta u(x)=\DRL\beta u(x)-\frac{x^{-\beta}}{\Gamma(1\hspace{-0.7pt}-\hspace{-0.7pt}\beta)}u(x), \quad   \text{for all } x\in(0,T].
\end{equation*}
\end{definition}

\begin{remark}\label{rmk:oncfd}$ $
\begin{enumerate}[(i)]
\item Like the Caputo derivative, the censored fractional derivative maps constants to 0, and  satisfies the scaling property
\begin{equation*}
\DCs\beta v(x)=c^{-\beta}\DCs\beta u(x/c),
\label{scaling property}
\end{equation*}
where  $u\in C_\beta(0,T]$, $c$ is a positive constant and $v(x) := u(x/c)\in C_\beta(0,\,cT]$.

\item For functions of the form $x^\alpha\,(\alpha > 0)$, the censored fractional derivative equals the R–L derivative up to a constant multiple: $\DCs\beta x^\alpha =  c_{\alpha,\,\beta} \DRL\beta x^\alpha$, where
\[
 c_{\alpha,\,\beta}=1-\frac{\Gamma(\alpha+1-\beta)}{\Gamma(\alpha\hspace{-0.7pt}+\hspace{-0.7pt}1)\Gamma(1\hspace{-0.7pt}-\hspace{-0.7pt}\beta)}  ,\quad\DRL\beta x^\alpha=x^{\alpha-\beta}\frac{\Gamma(\alpha+1)}{\Gamma(\alpha\hspace{-0.7pt}+\hspace{-0.7pt}1\hspace{-0.7pt}-\hspace{-0.7pt}\beta)}. 
\]
 By \eqref{inequality for gamma functions},   $ c_{\alpha,\,\beta}$  is in $(0,1)$.
In particular, for $\alpha=\beta$, we have  $\DCs\beta x^\alpha = \Gamma(\beta\hspace{-0.7pt}+\hspace{-0.7pt}1)\big(\beta\pi\hspace{-0.7pt}-\hspace{-0.7pt}\sin(\beta\pi)\big)/(\beta\pi)$.  While we can talk about the semigroup property for $\DRL\beta$ and the Caputo derivative \cite[Theorem 2.13 and Lemma 3.13]{kai}, 
we cannot for $\DCs\beta$. For instance,
\begin{equation*}
\begin{aligned}
&\DCs\beta \DCs\gamma x^\alpha=c_{\alpha-\gamma,\,\beta}\,c_{\alpha,\,\gamma}\,\DRL{\beta+\gamma}x^\alpha,\\
&\DCs\gamma \DCs\beta x^\alpha=c_{\alpha-\beta,\,\gamma}\,c_{\alpha,\,\beta}\,\DRL{\beta+\gamma}x^\alpha,
\end{aligned}
\end{equation*}
however $c_{\alpha-\gamma,\,\beta}\,c_{\alpha,\,\gamma}\neq c_{\alpha-\beta,\,\gamma}\,c_{\alpha,\,\beta}$ unless $\beta=\gamma$.

\item If $u\in  C^1(0,T]\cap L^1(0,T]$, then  on $(0,T]$, $\DCs\beta u$ allows the    representation \eqref{r^-1-beta representation of D_x^beta},
 from  which  it is clear that $-\DCs\beta$ satisfies the positive maximum principle \cite{BSW13}, and hence it is dissipative in the sense that $\|\lambda u+\DCs\beta u\|_{C[0,T]}\ge \lambda \|u\|_{C[0,T]}$ for any $\lambda>0$ and $u\in C^1[0,T]$. 

\item The Laplace transform of the censored fractional derivative can be written as
\begin{equation*}
\mathcal L\big[\DCs\beta u\big](k)=k^\beta \bigg(\mathcal L[ u](k)-k^{-1} \mathcal L\bigg[\frac{u(x/k)x^{-\beta}}{\Gamma(1-\beta)}\bigg](1)\bigg),\quad k>0,
\end{equation*}
which differs from $k^\beta\big(\mathcal L[u](k)-k^{-1} u(0)\big)$, the Laplace transform of the Caputo derivative \cite[Chapter 2.4]{Meerschaert2012}. 
One can notice that even in Laplace space, it is unclear if the initial conditions can be imposed on the problem $\DCs\beta u= g$.


\end{enumerate}
\end{remark}

\begin{remark} $ $
\label{strong solution space}
We will spend the next few pages  establishing the well-posedness of    $\DCs\beta u= g$ with $u(0)=u_0$,   for certain $g$. With the initial condition imposed, $C_\beta[0,T]$, which equals $\big\{u\in C[0,T]:J^{1-\beta}_0u\in C^1(0,T]\big\}$, now becomes a natural function space for solutions. A large part of the Caputo literature (e.g., \cite{kai}), however, chose $J^\beta_0\big[C\cap L^1(0,T]\big]$, i.e., the image of $J^\beta_0$ over $C\cap L^1(0,T]$, as the solution space. This difference seems not to matter, at least to our studies. Indeed, the set $U$ consisting of the solutions to \eqref{eq:main_intro} (for those $g$ of interest) is contained in the intersection of those two spaces, as shown in the diagram below (see Appendix \ref{proof of relations between sets} for the proof of the diagram)

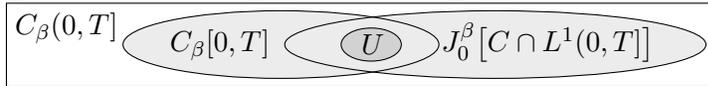
\hspace{50pt}
\begin{tikzpicture}
\fill[gray!15]  (-0.05, -0.12) ellipse    (2.1 and 0.45);

\fill[gray!15]  ( 2.8, -0.12) ellipse    (2.8 and 0.45);

\fill[gray!35]  (1.15, -0.12) ellipse  (0.4 and 0.22);

\draw           (-3.7,-0.68) rectangle +(9.5,1.11) node at (-2.9,0.1) {$C_\beta(0,T]$};

\draw           (-0.05, -0.12) ellipse    (2.1 and 0.45) node at (-0.86,-0.12) {$C_\beta[0,T]$};

\draw           ( 2.8, -0.12) ellipse    (2.8 and 0.45) node at (3.5,-0.09) {$J^\beta_0\big[C\cap L^1(0,T]\big]$};

\draw           (1.15, -0.12) ellipse  (0.4 and 0.22) node at (1.17,-0.12) {$U$};
\end{tikzpicture}\\
where we define $U=\big\{u\in C_\beta[0,T]:  x^{\beta-\alpha}\DCs\beta u \in C[0,T] \text{ for some }\alpha>0 \big\}$. Lastly, let us mention that $J^\beta_0 C[0,T]=\big\{u\in C[0,T]:J^{1-\beta}_0u\in C^1[0,T]\big\}$  \cite[Proposition 4.1]{vainikko2016functions}.

\end{remark}

\subsection{An integral operator and related kernels}
As we can see from \eqref{eq:serintro}, the solution to the IVP \eqref{eq:main_intro} may be seen as 
a variation of the R–L integral. In this subsection we introduce an integral operator and related kernels  for the convergence study of \eqref{eq:serintro}. This leads to Lemma \ref{thm:kseries}, which is a crucial  bound   in this work. The   probabilistic   interpretation  of  the kernels under consideration  will be presented in Section \ref{sec:hit_0}.

\begin{definition}\label{def:solker}
For $0<r<x$, we define the following kernels recursively
\begin{equation}\label{eq:kern}
k_j(x,r)=
\left\{
\begin{aligned}
&\frac{(x\hspace{-0.7pt}-\hspace{-0.7pt}r)^{\beta-1}r^{-\beta}}{\Gamma(\beta)\Gamma(1\hspace{-0.7pt}-\hspace{-0.7pt}\beta)}, & j=1,\\
&\int_r^x k_1(x,s)k_{j-1}(s,r)\,{\rm d}s, &j\ge 2.
\end{aligned}
\right.
\end{equation}
\end{definition}
\begin{remark}\label{rmk:intto1}
Note that for each $x>0$, $k_1(x,\,\cdot\,)$ is a beta distribution on $(0,x)$ with parameters $(1\hspace{-0.7pt}-\hspace{-0.7pt}\beta,\beta)$, and straightforward induction arguments can be used to prove that
$$\int_0^xk_j(x,r)\,{\rm d}r=1\quad (j\ge1,\;x>0)$$
 and 
 $$k_j(x,r)= \int_r^x k_{j-1}(x,s)k_1(s,r)\,{\rm d}s\quad (j\ge 2,\;x>r>0).$$
\end{remark}
\begin{definition}
\label{definition of K}
For $\psi\in C[0,T]$, we define  
\begin{equation*}
\mathcal K\psi(x)=
\left\{
\begin{aligned}
&\int_0^xk_1(x,r)\psi(r)\,{\rm d}r,&x>0,\\
&\psi(0),&x=0,
\end{aligned}
\right.
\end{equation*}
 where the explicit dependence of $\mathcal K$ on $\beta$ is suppressed to ease notation.
\end{definition}

\begin{remark}\label{K = Jbeta x^-beta}
It is easy to see that $\mathcal K\psi(x)=J^\beta_0\big[ x^{-\beta}\psi(x) /\Gamma(1\hspace{-0.7pt}-\hspace{-0.7pt}\beta)\big]$  for $\psi\in C[0,T]$  and $x\in(0,T]$, and that $\mathcal K$ is a linear operator preserving positivity ($\mathcal K\psi\ge 0$ if $\psi\ge 0$).
\end{remark}

\begin{lemma}
\label{continuity and boundness of Kpsi}
 For any $\alpha \geq 0$, we have
\begin{equation*}
\mathcal Kx^\alpha =  x^\alpha\hspace{0.7pt}\Gamma(\alpha\hspace{-0.7pt}+\hspace{-0.7pt}1\hspace{-0.7pt}-\hspace{-0.7pt}\beta)\big/\big(\Gamma(1\hspace{-0.7pt}+\hspace{-0.7pt}\alpha)\Gamma(1\hspace{-0.7pt}-\hspace{-0.7pt}\beta)\big).
\end{equation*}
If $\psi\in C[0,T]$ satisfies $\big\vert\psi(x)\big\vert\le Mx^\alpha$  for some $M>0$ and all $x\in(0,T]$, then $\mathcal K\psi\in C[0,T]$, and $\big\vert\mathcal K\psi(x)\big\vert\le M\mathcal Kx^\alpha$ for all $x\in(0,T]$.
\end{lemma}

\begin{proof}
The first claim is immediate from the definition of $\mathcal K$, and by the assumption on $\psi$, we have $\big\vert\mathcal K\psi(x)\big\vert\le\mathcal K\vert\psi\vert(x)\le M\mathcal Kx^\alpha$.
 We now prove that $\mathcal K\psi$ is continuous on $(0,T]$. For $\varepsilon\in(0,\,x/2)$, define
\begin{equation*}
\mathcal K_\varepsilon\psi(x)=\int_\varepsilon^{x-\varepsilon}k_1(x,r)\psi(r)\,{\rm d}r.
\end{equation*}
Given $T_1\in(0,T],$ for every $x\in[T_1,\,T]$ and $\varepsilon\in(0,T_1/2)$, we have
\begin{equation*}
\begin{aligned}
\big\vert\mathcal K_\varepsilon\psi(x)-\mathcal K\psi(x)\big\vert
&\le\int_0^\varepsilon k_1(x,r)\big\vert\psi(r)\big\vert\,{\rm d}r+\int_{x-\varepsilon}^xk_1(x,r)\big\vert\psi(r)\big\vert\,{\rm d}r\\
&\le\frac{\beta(x/\varepsilon-1)^{\beta-1}+(1\hspace{-0.7pt}-\hspace{-0.7pt}\beta)(x/\varepsilon-1)^{-\beta}}{\beta(1-\beta)\Gamma(\beta)\Gamma(1-\beta)}\|\psi\|_{C[0,T]}\\
&\le\frac{\beta(T_1/\varepsilon-1)^{\beta-1}+(1\hspace{-0.7pt}-\hspace{-0.7pt}\beta)(T_1/\varepsilon-1)^{-\beta}}{\beta(1-\beta)\Gamma(\beta)\Gamma(1-\beta)}\|\psi\|_{C[0,T]},
\end{aligned}
\end{equation*}
therefore, as $\varepsilon\rightarrow0,\;\mathcal K_\varepsilon\psi\to\mathcal K\psi$ uniformly on $[T_1,\,T]$. Because $\mathcal K_\varepsilon\psi$ is continuous on $[T_1,\,T]$,  $\mathcal K\psi$   must be continuous on $[T_1,\,T]$, and thus on $(0,T]$. In addition, by the continuity of $\psi$ at $x=0,\;\mathcal K\psi(x)\rightarrow\psi(0)$ as $x\rightarrow0$, and therefore $\mathcal K\psi\in C[0,T]$.
\end{proof}

We can now obtain the crucial bound that  will help us  adapt  \cite[Theorem 7.10]{kai} to the censored IVP \eqref{eq:main_intro} in order to express the solution as a series.  
\begin{lemma}\label{thm:kseries}
 
For any $\alpha>0$, we have
\begin{equation}\label{eq:xalphaidentity}
\sum_{j=1}^\infty\mathcal K^jx^\alpha=x^\alpha\left(\frac{\Gamma(1\hspace{-0.7pt}+\hspace{-0.7pt}\alpha)\Gamma(1\hspace{-0.7pt}-\hspace{-0.7pt}\beta)}{\Gamma(\alpha+1-\beta)}-1\right)^{-1}.
\end{equation}
If $\psi\in C[0,T]$ satisfies $\big\vert\psi(x)\big\vert\le Mx^\alpha $  for some $M>0$ and all $x\in(0,T]$, then  $$\sum_{j=1}^\infty\mathcal K^j\psi\in C[0,T], \text{ and } \Bigg\vert\sum_{j=1}^\infty\mathcal K^j\psi(x)\Bigg\vert\le M\sum_{j=1}^\infty\mathcal K^jx^\alpha \text{  for all $x\in(0,T]$}.
$$
In addition, $\mathcal K^j\psi(x)=\int_0^xk_j(x,r)\psi(r)\,{\rm d}r$ for all $j\in\mathbb N,\,x\in(0,T]$.
 
\end{lemma}
\begin{proof}   
 We first confirm \eqref{inequality for gamma functions} using the fact that  $t^\alpha$ and $(1\hspace{-0.7pt}-\hspace{-0.7pt}t)^{-\beta}$  strictly increase, so that
\begin{equation*}
\frac1{\alpha\hspace{-0.7pt}-\hspace{-0.7pt}\beta\hspace{-0.7pt}+\hspace{-0.7pt}1}=\int_0^1(1\hspace{-0.7pt}-\hspace{-0.7pt}t)^\alpha(1\hspace{-0.7pt}-\hspace{-0.7pt}t)^{-\beta}\,{\rm d}t<\int_0^1t^\alpha(1\hspace{-0.7pt}-\hspace{-0.7pt}t)^{-\beta}\,{\rm d}t=\frac{\Gamma(1\hspace{-0.7pt}+\hspace{-0.7pt}\alpha)\Gamma(1\hspace{-0.7pt}-\hspace{-0.7pt}\beta)}{\Gamma(1\hspace{-0.7pt}+\hspace{-0.7pt}\alpha\hspace{-0.7pt}+\hspace{-0.7pt}1\hspace{-0.7pt}-\hspace{-0.7pt}\beta)}.
\end{equation*}
Applying Lemma \ref{continuity and boundness of Kpsi} for $j$ times, we get $\mathcal K^jx^\alpha=x^\alpha\big(\Gamma(1\hspace{-0.7pt}+\hspace{-0.7pt}\alpha)\Gamma(1\hspace{-0.7pt}-\hspace{-0.7pt}\beta)/\Gamma(\alpha\hspace{-0.7pt}+\hspace{-0.7pt}1\hspace{-0.7pt}-\hspace{-0.7pt}\beta)\big)^{-j}$. Then, by summing over $j$, we obtain \eqref{eq:xalphaidentity} from \eqref{inequality for gamma functions}. Meanwhile, we have $\big\vert\mathcal K^j\psi(x)\big\vert\le M\mathcal K^jx^\alpha$ and $\mathcal K^j\psi\in C[0,T]$, so $\sum_{j=1}^\infty\mathcal K^j\psi$ converges uniformly to a limit in $C[0,T]$, whose absolute value is pointwise bounded by $M\sum_{j=1}^\infty\mathcal K^jx^\alpha$. Finally, by induction,
\begin{equation*}
\begin{aligned}
\mathcal K^j\psi(x)=\mathcal K\mathcal K^{j-1}\psi(x)&=\int_0^xk_1(x,r)\int_0^rk_{j-1}(r,s)\psi(s)\,{\rm d}s\,{\rm d}r\\
&=\int_0^x\int_s^xk_1(x,r)k_{j-1}(r,s)\psi(s)\,{\rm d}r\,{\rm d}s\\
&=\int_0^xk_j(x,s)\psi(s)\,{\rm d}s.
\end{aligned}
\end{equation*}

\end{proof}

\begin{remark}\label{rem:sumk_js}
 
In Lemma \ref{thm:kseries}, we require $\alpha>0$ (though the last statement there holds for all $\alpha\ge0$), in fact, if $\alpha=0$, let $\psi=1$, then \begin{equation*}
\sum_{j=1}^\infty\mathcal K^j\psi(x)=\sum_{j=1}^\infty\int_0^xk_j(x,r)\,{\rm d}r=\infty.
\end{equation*}
\end{remark}

\section{Well-posedness of the censored IVPs}\label{sec:wellposedness}
\subsection{Inverse of \texorpdfstring{$\DCs\beta$}{censored derivative}}

We begin with the basic censored IVP \eqref{eq:main_intro} with $g\in C(0,T]$ and $u_0\in\mathbb R$. Our strategy is to consider the equivalent Caputo/R–L problem for $\bar u =u\hspace{-0.7pt}-\hspace{-0.7pt}u_0$ with the unbounded coefficient  $x^{-\beta}/\Gamma(1\hspace{-0.7pt}-\hspace{-0.7pt}\beta)$,
\begin{equation}\label{eq:RLIVP}
\DRL\beta\bar u(x)=\frac{x^{-\beta}}{\Gamma(1\hspace{-0.7pt}-\hspace{-0.7pt}\beta)}\bar u(x)+g(x),\quad x>0,\quad \bar u(0)=0,
\end{equation}
and then show that for certain forcing terms $g$, the formula \cite[Theorem 7.10]{kai} for  bounded coefficients still yields a solution to \eqref{eq:RLIVP}, and thus to  \eqref{eq:main_intro}.  

\begin{remark}$ $
\label{remark about the equivalent Caputo/R-L problem}
\begin{enumerate}[(i)]
	\item  We can solve \eqref{eq:RLIVP} using Picard iteration, i.e., $\bar u_{m+1}(x)=J_0^\beta\big[x^{-\beta}\bar u_m(x)/\Gamma(1\hspace{-0.7pt}-\hspace{-0.7pt}\beta)+g(x)\big]$ ($m=1,2,\cdots$) with $\bar u_1=0$. By Remark \ref{K = Jbeta x^-beta},  the limit equals $I^\beta_0g$ defined in  \eqref{eq:sol} if the iteration converges.
	
	 \item For $g\in C[0,T]$ and $u_0=0$,   \cite[Example 3.4]{LPV15} guarantees (after change of variables) that there is a unique solution in $J^\beta_0 C[0,T]$ to \eqref{eq:RLIVP} and thus to \eqref{eq:main_intro} (and also to   \eqref{linear variable in equation}, a linear IVP  we will study  later). 	However, \cite{LPV15} does not cover nonlinear IVPs \eqref{eq:ivpODE}. Moreover, it provides neither explicit expressions  nor stochastic interpretations for the solutions,  and seemingly cannot obtain the continuous dependence. Lastly, no singularity of $g$ at $x=0$ is allowed there either. Let us also  mention that, as  stated beneath \cite[Equation (5)]{LPV15}, for singular fractional differential equations one should not expect to impose initial conditions without losing regularity. But \eqref{eq:RLIVP} proves to be an exception. In fact, constant functions solve its homogeneous version because of the specific coefficient $ x^{-\beta}/\Gamma(1\hspace{-0.7pt}-\hspace{-0.7pt}\beta)$.
	
	\item If one replaces the coefficient   in the R–L problem \eqref{eq:RLIVP}  by  $C x^{-\beta}/\Gamma(1\hspace{-0.7pt}-\hspace{-0.7pt}\beta)$,   then 
 the series representation for the solution would be $\bar u=\sum_{j=0}^\infty C^j\mathcal K^jJ^\beta_0g$, which does not converge for important data (like $g=1$) if $\vert C\vert \ge\Gamma(1\hspace{-0.7pt}+\hspace{-0.7pt}\beta)\Gamma(1\hspace{-0.7pt}-\hspace{-0.7pt}\beta)$.
In other words, the Picard iteration will not converge for such $C$.  The above threshold can be obtained from the proof of Lemma \ref{thm:kseries}, and is consistent with the condition ``$b(0)<\Gamma(\alpha\hspace{-0.7pt}+\hspace{-0.7pt}1)$" in \cite[Example 3.4]{LPV15}.
\end{enumerate}
\end{remark}


We now present a key result concerning IVP \eqref{eq:main_intro}, which serves as the fundamental theorem of calculus for $\DCs\beta$. Or simply put, $I_0^\beta$ is to $\DCs\beta$ as $J_0^\beta$ is to $\DRL\beta$.

\begin{theorem}\label{thm:existence}
Let $u_0\in\mathbb R$ and $g\in C(0,T]$ such that  $\big\vert g(x)\big\vert\le M x^{\alpha-\beta}$ for some  $M,\alpha>0$ and all $x\in(0,T]$. Then there exists a unique function $u\in C_\beta[0,T]$    satisfying \eqref{eq:main_intro}, and it has the series representation
\begin{equation}
u(x)-u_0=I^\beta_0 g(x):=\sum_{j=0}^\infty\mathcal K^jJ^\beta_0g(x),
\label{eq:sol}
\end{equation}
where $\mathcal K^0$ is the identity operator by convention.  Moreover, $u$ depends on $u_0$ and $g$ continuously in the sense of Remark \ref{meaning of continuous dependence}.
\end{theorem}
Theorem \ref{thm:existence} is an immediate consequence of Lemmata \ref{lem:uniq} and \ref{lem:sol_cts}.
 
\begin{remark}
\label{equivalent representations}
For $g$ satisfying the conditions in Theorem \ref{thm:existence}, $I^\beta_0 g$ can  be equivalently represented as  
\begin{align}
I^\beta_0 g(x)&=J^\beta_0g(x)+\sum_{j=1}^\infty\int_0^xk_j(x,r)J^\beta_0g(r)\,{\rm d}r,\label{alter representation 1}\\
&=\sum_{j=1}^\infty\mathcal K^j\big[\Gamma(1\!-\!\beta)x^\beta g(x)\big]\label{alter representation 2}\\
&=\sum_{j=0}^\infty J_0^\beta\bigg[\frac{\mathcal K^j\big[x^\beta g(x)\big]}{x^\beta}\bigg],\label{alter representation 3}
\end{align}
where \eqref{alter representation 1} is due to Lemma \ref{thm:kseries}, while \eqref{alter representation 2} and \eqref{alter representation 3} are due to Remark \ref{K = Jbeta x^-beta}. From any representation, we can see that $I^\beta_0$ is a linear operator preserving positivity ($I^\beta_0 g\ge 0$ if $g\ge 0$).

\end{remark}
 
\begin{remark}
\label{meaning of continuous dependence}
 For $g\in C[0,T]$, we can prove the continuous dependence by showing that $\Vert u\hspace{-0.7pt}-\hspace{-0.7pt}u_0\Vert_{C[0,T]}\le C\Vert g\Vert_{C[0,T]}$ for some $C$ dependent only on $\beta$ and $T$. For a more general $g$ which may diverge at $x=0$, $C$ will depend also on $\alpha$, and $\Vert g\Vert_{C[0,T]}$   needs to be replaced by $\Vert g\Vert_{G^{\text{\fontsize{8pt}{0pt}$\alpha\hspace{-2pt}-\hspace{-3pt}\beta$}}(0,T]}$, where we define for any $\gamma\in\mathbb R$ a Banach space $G^{\gamma}(0,T]= \big\{h\in C(0,T]:\Vert h\Vert_{G^{\text{\fontsize{8pt}{0pt}$\gamma$}}(0,T]}<\infty\big\}$, with the norm $\Vert h\Vert_{G^{\text{\fontsize{8pt}{0pt}$\gamma$}}(0,T]}:=\sup\big\{\vert x^{-\gamma}h(x) \vert:x\in(0,T]\big\}$.    In particular, if $g\in C[0,T]$ and $\alpha=\beta$, then $\Vert g\Vert_{G^{\text{\fontsize{8pt}{0pt}$\alpha\hspace{-2pt}-\hspace{-3pt}\beta$}}(0,T]}=\Vert g\Vert_{C[0,T]}$. (Note that $G^{\gamma}$ is the same as $\hat B$ defined in \cite[Proof of Lemma 5.3]{kai}.)
 
\end{remark}

\begin{lemma}\label{lem:uniq}
Solutions to problem \eqref{eq:main_intro} are unique in $C_\beta[0,T]$.
\end{lemma}
\begin{proof}
Let $u_1,u_2\in C_\beta[0,T]$ be two solutions to problem \eqref{eq:main_intro}. By linearity of $\DCs\beta $, $u:= u_1\hspace{-0.7pt}-\hspace{-0.7pt}u_2\in C_\beta[0,T]$ satisfies $\DCs\beta u=0$ on $(0,T]$. Therefore for every $x\in(0,T]$, $\DRL\beta u(x)=\Gamma(1\hspace{-0.7pt}-\hspace{-0.7pt}\beta)^{-1}x^{-\beta}u(x)$,
where the right-hand side is in $C\cap L^1(0,T]$.  Using Lemma \ref{lem:prelim}-(ii) as well as Remark \ref{K = Jbeta x^-beta}, we obtain
\begin{equation*}
\DRL\beta u(x)=\frac{x^{-\beta}}{\Gamma(1\hspace{-0.7pt}-\hspace{-0.7pt}\beta)}u(x)=\DRL\beta J^\beta_0\bigg[\frac{x^{-\beta}}{\Gamma(1\hspace{-0.7pt}-\hspace{-0.7pt}\beta)}u(x)\bigg]=\DRL\beta\mathcal Ku(x),
\end{equation*}
where $\mathcal Ku\in C_\beta(0,T]$. By Lemma \ref{continuity and boundness of Kpsi}, $\mathcal Ku$ is in $C[0,T]$, and so is $u\hspace{-0.7pt}-\hspace{-0.7pt}\mathcal Ku$.  Consequently, $u\hspace{-0.7pt}-\hspace{-0.7pt}\mathcal Ku\in C_\beta[0,T]$. By the linearity of $\DRL\beta$, we know $\DRL\beta\big[u\hspace{-0.7pt}-\hspace{-0.7pt}\mathcal Ku\big]$=0. According to Lemma \ref{lem:prelim}-(iv), we obtain $u=\mathcal Ku$.

Let $\xi\in\arg\max_{\,r\in[0,T]}\big\vert u(r)\big\vert$. If $\xi=0$, then $u=0$ on $[0,T]$ because $u(0)=0$. If $\xi>0$, using the fact that $u(\xi)=\mathcal Ku(\xi)$, we have $
\int_0^\xi k_1(\xi,r)\big(u(\xi)\hspace{-0.7pt}-\hspace{-0.7pt}u(r)\big)\,{\rm d}r=0$, where $u(\xi)\hspace{-0.7pt}-\hspace{-0.7pt}u(r)$ never changes sign for all $r\in[0,\,\xi]$, according to the definition of $\xi$. So $u(\xi)=u(r)$ for all $r\in[0,\,\xi]$, therefore $u(\xi)=u(0)=0$, and we still obtain $u=0$ on $[0,T]$. This proves $u_1= u_2$, and we are done.
\end{proof}  

\begin{lemma}\label{lem:sol_cts}
For $g$ satisfying the conditions in Theorem \ref{thm:existence}, $I^\beta_0 g$ is in $C_\beta[0,T]$  with $I^\beta_0 g(0)=0$ and $\DCs\beta I^\beta_0 g=g$. In addition, $I^\beta_0 g$ depends on $g$ continuously in the sense of Remark \ref{meaning of continuous dependence}.
\end{lemma}

\begin{proof}  Using representation \eqref{alter representation 2}, we can see $I_0^\beta g(0)=0$ from the assumptions on $g$ and Definition \ref{definition of K}, then from Lemma \ref{thm:kseries} we obtain  for all $x\in(0,T]$,
\begin{equation}
\label{bound for I^beta_0 g}
\big\vert I_0^\beta g(x)\big\vert\le I_0^\beta\vert g\vert(x)\le Mx^\alpha\left(\frac{\Gamma(1+\alpha)}{\Gamma(\alpha\hspace{-0.7pt}+\hspace{-0.7pt}1\hspace{-0.7pt}-\hspace{-0.7pt}\beta)}-\frac1{\Gamma(1\hspace{-0.7pt}-\hspace{-0.7pt}\beta)}\right)^{-1},\;\text{and}\;I_0^\beta g\in C[0,T].
\end{equation}

Note that in \eqref{alter representation 3}, the summation commutes with $J_0^\beta$, by Fubini's Theorem and the above bound. So
\begin{equation}
\begin{aligned}
I^\beta_0 g(x)=J_0^\beta\sum_{j=0}^\infty\frac{\mathcal K^j\big[x^\beta g(x)\big]}{x^\beta}&=J_0^\beta\Big[g(x)+x^{-\beta}\sum_{j=1}^\infty\mathcal K^j\big[x^\beta g(x)\big]\Big]\\
&=J_0^\beta\Big[g(x)+\frac{x^{-\beta}I^\beta_0 g(x)}{\Gamma(1-\beta)}\Big],
\end{aligned}
\end{equation}
with the last equality due to \eqref{alter representation 2}. Therefore, $I^\beta_0 g=J_0^\beta\psi$ for a $\psi\in C(0,T]$ satisfying
\begin{equation}
\label{I^beta_0 g=J_0^beta psi}
\big\vert\psi(x)\big\vert\le Mx^{\alpha-\beta}\bigg(1-\frac{\Gamma(\alpha+1-\beta)}{\Gamma(1\hspace{-0.7pt}+\hspace{-0.7pt}\alpha)\Gamma(1\hspace{-0.7pt}-\hspace{-0.7pt}\beta)}\bigg)^{-1},\quad\text{for all }\;x\in(0,T],
\end{equation}
so Lemma \ref{lem:prelim}-(ii) proves that $I^\beta_0 g$ is in $C_\beta(0,T]$ and thus $C_\beta[0,T]$. Lemma \ref{lem:prelim}-(ii) also proves that
\begin{equation*}
\DRL\beta I^\beta_0g(x) =\psi(x)= g(x)+\frac{x^{-\beta}I^\beta_0g(x)}{\Gamma(1-\beta)},\quad\;\text{for all}\;x\in(0,T],
\end{equation*}
which rewrites as $\DCs\beta I^\beta_0 g=g$ by Definition \ref{def:cfd}.

To see the continuity of $I^\beta_0$, let the $M$ in \eqref{bound for I^beta_0 g} be $\Vert g\Vert_{G^{\text{\fontsize{8pt}{0pt}$\alpha\hspace{-2pt}-\hspace{-3pt}\beta$}}(0,T]}$ ($G^\gamma(0,T]$ is defined in Remark \ref{meaning of continuous dependence}), then we obtain
\begin{equation*}
\Vert I^\beta_0 g\Vert_{G^{\text{\fontsize{8pt}{0pt}$\alpha$}}(0,T]}\le\left(\frac{\Gamma(1+\alpha)}{\Gamma(\alpha\hspace{-0.7pt}+\hspace{-0.7pt}1\hspace{-0.7pt}-\hspace{-0.7pt}\beta)}-\frac1{\Gamma(1\hspace{-0.7pt}-\hspace{-0.7pt}\beta)}\right)^{-1}\Vert g\Vert_{G^{\text{\fontsize{8pt}{0pt}$\alpha\hspace{-2pt}-\hspace{-3pt}\beta$}}(0,T]}.
\end{equation*}
Since $\alpha>0$, we have $\Vert I^\beta_0 g\Vert_{C[0,T]}\le T^\alpha\Vert I^\beta_0 g\Vert_{G^{\text{\fontsize{8pt}{0pt}$\alpha$}}(0,T]}\le  C\Vert g\Vert_{G^{\text{\fontsize{8pt}{0pt}$\alpha\hspace{-2pt}-\hspace{-3pt}\beta$}}(0,T]}$ for some $C$ dependent only on $\alpha,\,\beta$ and $T$.

\end{proof}  

\begin{example} \label{ex:harm}
Recall that for the Caputo IVP $\DRL\beta \big[u\hspace{-0.7pt}-\hspace{-0.7pt}u(0)\big] = x^\alpha\;(\alpha>-1)$ with $u(0)=u_0$, the solution  is $u_0+J^\beta_0x^\alpha$ \cite{kai}.
By \eqref{alter representation 2} and Lemma \ref{thm:kseries}, the solution  to \eqref{eq:main_intro} for $g(x)=x^\alpha \;(\alpha>-\beta)$  is
\begin{equation}\label{eq:seriesforalpha}
u(x)-u_0=I_0^\beta x^\alpha=\bigg(\frac{\Gamma(\alpha\hspace{-0.7pt}+\hspace{-0.7pt}\beta\hspace{-0.7pt}+\hspace{-0.7pt}1)}{\Gamma(\alpha+1)}-\frac1{\Gamma(1\hspace{-0.7pt}-\hspace{-0.7pt}\beta)}\bigg)^{-1}x^{\alpha+\beta}=c_{\alpha+\beta,\,\beta}^{-1}\, J^\beta_0x^\alpha,
\end{equation}
where $c_{\alpha+\beta,\,\beta}$ is defined in Remark \ref{rmk:oncfd}-(ii). In particular, when $\alpha=0$, $c_{\alpha+\beta,\,\beta}^{-1}=\beta\pi\big/\big(\beta\pi\hspace{-0.7pt} -\hspace{-0.7pt}\sin(\beta\pi)\big)$. If $\alpha\in(-1,-\beta]$, we may not be able to impose the initial 
condition in \eqref{eq:main_intro}, since the solution may explode at $0$. For example, when $\alpha=-\beta$, one can verify that a particular solution is $-\Gamma(1\hspace{-0.7pt}-\hspace{-0.7pt}\beta)\ln( x)/H_{-\beta}$,
where $H_{-\beta}$ is the Harmonic number.

\end{example}

\begin{remark}
\label{holder cts of D^-1}
Assume $g$ satisfies the conditions in Theorem \ref{thm:existence}. By \eqref{I^beta_0 g=J_0^beta psi} and Remark \ref{rmk: Holder regularity of J0beta g}, $I^\beta_0 g$ is H\"older continuous. More specifically,
\begin{enumerate}[(i)]
\item for all $T_1\in(0,T)$,  $I^\beta_0 g\in C^{0,\beta}[T_1,T]$ with a H\"older constant being
\begin{equation*}
\frac{2M\max\{T_1^{\alpha-\beta}\!,\,T^{\alpha-\beta}\} \Gamma(1+\alpha)\Gamma(1-\beta) }{\big[\Gamma(1\hspace{-0.7pt}+\hspace{-0.7pt}\alpha)\Gamma(1\hspace{-0.7pt}-\hspace{-0.7pt}\beta)-\Gamma(\alpha\hspace{-0.7pt}+\hspace{-0.7pt}1\hspace{-0.7pt}-\hspace{-0.7pt}\beta)\big]\Gamma(1\hspace{-0.7pt}+\hspace{-0.7pt}\beta)};
\end{equation*}
\item additionally assume $\alpha<\beta$, then  $I^\beta_0 g\in C^{0,\alpha}[0,T]$ with a H\"older constant being
\begin{equation*}
\frac{2M\Gamma(1-\beta)\Gamma(\alpha+1-\beta)}{\Gamma(1\hspace{-0.7pt}+\hspace{-0.7pt}\alpha)\Gamma(1\hspace{-0.7pt}-\hspace{-0.7pt}\beta)-\Gamma(\alpha\hspace{-0.7pt}+\hspace{-0.7pt}1\hspace{-0.7pt}-\hspace{-0.7pt}\beta)}.
\end{equation*}
\end{enumerate}
\end{remark}

\subsection{General censored IVPs}\label{sec:eigen}

We now  study the censored IVP with more general Lipschitz data $f$:
\begin{equation}
\left\{
\begin{aligned}
\DCs\beta u(x)&=f\big(x,u(x)\big), &&x\in(0,T],\\
u(x)&=u_0, &&x=0.
\end{aligned}\right.
\label{eq:ivpODE}
\end{equation}
 Analogously to Caputo IVPs \cite[Chapter 6]{kai} and classical ODEs, the censored IVP \eqref{eq:ivpODE} can be solved by Picard iteration, with Theorem \ref{thm:existence} acting as the fundamental theorem of calculus.

\begin{proposition}
\label{global existence of ivpODE for locally Lipschitz f}
For $f:(0,T]\times[u_0\hspace{-0.7pt}-\hspace{-0.7pt}Y,\,u_0\hspace{-0.7pt}+\hspace{-0.7pt}Y]\rightarrow\mathbb R$, where $u_0\in\mathbb R$, $Y>0$, assume that   there exist $L,\,\alpha,\,M>0$ such that for all $x\in(0,T]$ and all $y, \tilde y\in[u_0\hspace{-0.7pt}-\hspace{-0.7pt}Y,\,u_0\hspace{-0.7pt}+\hspace{-0.7pt}Y]$,
\begin{enumerate}[(i)]
\item $\;f(\,\cdot\,,y)\in C(0,T]$,
\item $\big\vert f(x,y)\big\vert\le Mx^{\alpha-\beta}$,
\item $\big\vert f(x,y) - f(x, \tilde y )\big\vert\le Lx^{\alpha-\beta}\vert y-\tilde y \vert$.
\end{enumerate}
Then there exists a unique $u\in C_\beta[0, \widetilde T ]$ solving \eqref{eq:ivpODE} on $[0, \widetilde T ]$, where either $ \widetilde T =T$, or $ \widetilde T\in(0,T)$ with $u( \widetilde T )\in\{u_0\hspace{-0.7pt}-\hspace{-0.7pt}Y,\,u_0\hspace{-0.7pt}+\hspace{-0.7pt}Y\}$.
Furthermore, $u$ depends on $u_0$ and $f$ continuously, in the sense of Lemma \ref{continuous dependence}.
\end{proposition}


 We leave the proof of Proposition \ref{global existence of ivpODE for locally Lipschitz f} at the end of this subsection. Now we are going
to prove the following lemmata. Lemma \ref{bound of the product} gives us a bound which is essential to the convergence of our Picard iteration. Lemmata \ref{uniqueness of ivpODE},  \ref{continuous dependence} and \ref{local existence of ivpODE for locally Lipschitz f} guarantee the uniqueness, continuous dependence and the local existence of the solutions, respectively.    Finally we will prove the global existence   by extending the local solution.

\begin{lemma}
\label{bound of the product}
For $x, \alpha>0$ and $N\in\mathbb N$, we have
\begin{equation*}
\big(I_0^\beta [x^{\alpha-\beta}\,\cdot\,]\big)^N1(x)\le  C  \frac{2^Nx^{N\alpha}}{(N!\alpha^N)^\beta},
\end{equation*}
where $1(x)$ is the constant function 1, $C $ is a positive constant dependent only on $\alpha$ and $\beta$, and we denote
$$\big(I_0^\beta [x^{\alpha-\beta}\,\cdot\,]\big)^Ng(x)=\underbrace{I_0^\beta\Big[x^{\alpha-\beta}\cdots I_0^\beta\big[x^{\alpha-\beta}}_{N\;\text{times}}g(x)\big]\cdots\Big].$$
\end{lemma}

\begin{proof} 
From Example \ref{ex:harm} we know that
\begin{equation}
\label{The formula for I0beta^n}
\big(I_0^\beta [x^{\alpha-\beta}\,\cdot\,]\big)^N1(x)=\prod_{n=1}^N\bigg(\frac{\Gamma(1+n\alpha)}{\Gamma(n\alpha\hspace{-0.7pt}+\hspace{-0.7pt}1\hspace{-0.7pt}-\hspace{-0.7pt}\beta)}-\frac1{\Gamma(1\hspace{-0.7pt}-\hspace{-0.7pt}\beta)}\bigg)^{-1}x^{N\alpha},
\end{equation}
where each factor is positive. Using Stirling's formula for the gamma function, i.e.
\begin{equation*}
\Gamma(z)=\sqrt{\frac{2\pi}z}\Big(\frac ze\Big)^z\bigg(1+O\Big(\frac1z\Big)\bigg),
\end{equation*}
we  have the following approximation
\begin{equation*}
\frac{\Gamma(1+n\alpha)}{\Gamma(n\alpha\hspace{-0.7pt}+\hspace{-0.7pt}1\hspace{-0.7pt}-\hspace{-0.7pt}\beta)}=(n\alpha)^\beta\bigg(1+O\Big(\frac1n\Big)\bigg),
\end{equation*}
which indicates that there exists $ \tilde n\in\mathbb N $ such that for all $n> \tilde n $,
\begin{equation*}
\frac{\Gamma(1+n\alpha)}{\Gamma(n\alpha\hspace{-0.7pt}+\hspace{-0.7pt}1\hspace{-0.7pt}-\hspace{-0.7pt}\beta)}-\frac1{\Gamma(1\hspace{-0.7pt}-\hspace{-0.7pt}\beta)}\ge\frac{(n\alpha)^\beta}2,
\end{equation*}
so there exists $C>0$ such that Lemma \ref{bound of the product} holds for all $N\in\mathbb N$.
\end{proof}

\begin{remark}
In Proposition \ref{global existence of ivpODE for locally Lipschitz f}, if $\alpha=\beta$, then we only need $(I_0^\beta)^N1(x)$ for its proof. The multiplier $x^{\alpha-\beta}$ is to accommodate more general $f$ that may diverge at $x=0$.
\end{remark}

\begin{lemma}
\label{uniqueness of ivpODE}
If $f$ satisfies the condition (iii) in Proposition \ref{global existence of ivpODE for locally Lipschitz f}, and both $u_1,\,u_2\in C_\beta[0,T]$ solve IVP \eqref{eq:ivpODE}, then $u_1=u_2$.
\end{lemma}

\begin{proof}
By the linearity of $\DCs\beta$, the difference $u:= u_1\hspace{-0.7pt}-\hspace{-0.7pt}u_2\in C_\beta[0,T]$ satisfies
\begin{equation*}
\left\{
\begin{aligned}
\DCs\beta u(x)&=f\big(x,u_1(x)\big)-f\big(x,u_2(x)\big),&&x\in(0,T],\\
\quad u(x)&=0,&&x=0.
\end{aligned}
\right.
\end{equation*}
Since $u\in C_\beta[0,T]$, we know $\DCs\beta u(x)\in C(0,T]$, so $f\big(x,u_1(x)\big)-f\big(x,u_2(x)\big)\in C(0,T]$.\\
By assumption, for all $x\in(0,T]$,
\begin{equation*}
\Big\vert f\big(x,u_1(x)\big)-f\big(x,u_2(x)\big)\Big\vert\le Lx^{\alpha-\beta}\big\vert u_1(x)-u_2(x)\big\vert=Lx^{\alpha-\beta}\big\vert u(x)\big\vert.
\end{equation*}
 Then by Theorem \ref{thm:existence}  and the positivity preserving property of $I^\beta_0$, for all $x\in(0,T]$,
\begin{equation*}
 \big\vert u(x) \big\vert \le I_0^\beta\Big\vert f\big(x,u_1(x)\big)-f\big(x,u_2(x)\big)\Big\vert\le  LI_0^\beta\Big[x^{\alpha-\beta}\big\vert u(x)\big\vert\Big].
\end{equation*}
Iterating the above inequality, we obtain for all $N\in\mathbb N,\;x\in(0,T]$,
\begin{equation*}
\big\vert u(x)\big\vert\le  L^N\big(I_0^\beta [x^{\alpha-\beta}\,\cdot\,]\big)^N\vert u\vert(x)\le  \Vert u\Vert_{C[0,T]}L^N\big(I_0^\beta [x^{\alpha-\beta}\,\cdot\,]\big)^N1(x).
\end{equation*}
By Lemma \ref{bound of the product}, as $N\rightarrow\infty$, we obtain $u(x)=0$ for all $x\in(0,T]$.
\end{proof}

We prove below the continuous dependence of $u$ on $f$. Then the continuous dependence on $u_0$ is a simple corollary if we take $\tilde f(x,y)=f(x,y+\delta)$, where $\delta$ is the  tiny   change in $u_0$. 
\begin{lemma}
\label{continuous dependence}
If $f$ and $\tilde f$ satisfy conditions (ii) and (iii) in Proposition \ref{global existence of ivpODE for locally Lipschitz f}, and $u,\,\tilde u\in C_\beta[0,T]$ satisfy $\DCs\beta u=f(x,u)$ and $\DCs\beta\tilde u=\tilde f(x,\tilde u)$ with $u(0)=\tilde u(0)=u_0$, then $\Vert u-\tilde u\Vert_{C[0,T]}\le  C\varepsilon$, where $C$ depends only on $\alpha,\,\beta,\,L,\,T$, and $\varepsilon=\sup\big\{x^{\beta-\alpha}\vert f(x,y)-\tilde f(x,y)\vert:x\in(0,T],\,y\in[u_0\hspace{-0.7pt}-\hspace{-0.7pt}Y,\,u_0\hspace{-0.7pt}+\hspace{-0.7pt}Y]\big\}$   (the definition of $\varepsilon$ is analogous to  $\Vert g\Vert_{G^{\text{\fontsize{8pt}{0pt}$\alpha\hspace{-2pt}-\hspace{-3pt}\beta$}}(0,T]}$ in Remark \ref{meaning of continuous dependence}).
\end{lemma}
\begin{proof}
By assumption, $\DCs\beta u,\DCs\beta\tilde u\in C(0,T]$, so $f\big(x,u(x)\big),\tilde f\big(x,\tilde u(x)\big)$ are in $C(0,T]$ and bounded by $Mx^{\alpha-\beta}$. According to Theorem \ref{thm:existence} and the positivity preserving property of $I^\beta_0$, for all $x\in(0,T]$, we have
\begin{equation*}
\begin{aligned}
\big\vert u(x)-\tilde u(x)\big\vert&=\Big\vert I_0^\beta f\big(x,u(x)\big)-I_0^\beta\tilde f\big(x,\tilde u(x)\big)\Big\vert\\
&\le \Big\vert I_0^\beta f\big(x,u(x)\big)-I_0^\beta f\big(x,\tilde u(x)\big)\Big\vert+\Big\vert I_0^\beta f\big(x,\tilde u(x)\big)-I_0^\beta\tilde f\big(x,\tilde u(x)\big)\Big\vert\\
&\le  I_0^\beta\Big\vert f\big(x,u(x)\big)-f\big(x,\tilde u(x)\big)\Big\vert+I_0^\beta \Big\vert f\big(x,\tilde u(x)\big)-\tilde f\big(x,\tilde u(x)\big)\Big\vert\\
&\le  LI_0^\beta\Big[x^{\alpha-\beta}\big\vert u(x)\hspace{-0.7pt}-\hspace{-0.7pt}\tilde u(x)\big\vert\Big]+\varepsilon I_0^\beta x^{\alpha-\beta}.
\end{aligned}
\end{equation*}
Like the proof of Lemma \ref{uniqueness of ivpODE}, by iterating the above inequality, we obtain for all $N\in\mathbb N,\;x\in (0,T]$,
\begin{equation*}
\big\vert u(x)\hspace{-0.7pt}-\hspace{-0.7pt}\tilde u(x)\big\vert\le \Vert u\hspace{-0.7pt}-\hspace{-0.7pt}\tilde u\Vert_{C[0,T]}L^N\big(I_0^\beta [x^{\alpha-\beta}\,\cdot\,]\big)^N1(x)+\varepsilon\sum_{n=0}^{N-1} L^n\big(I_0^\beta [x^{\alpha-\beta}\,\cdot\,]\big)^{n+1}1(x).
\end{equation*}
By Lemma \ref{bound of the product}, as $N\rightarrow\infty$, the first summand goes to 0 uniformly in $x$, and the second summand can be bounded by $\varepsilon\hspace{0.8pt}C$ for some finite $C$ dependent only on $\alpha,\,\beta,\,L,\,T$.
\end{proof}

\begin{lemma}
\label{local existence of ivpODE for locally Lipschitz f}
If $f$ satisfies conditions (i), (ii) and (iii) in Proposition \ref{global existence of ivpODE for locally Lipschitz f}, then there exists $u\in C_\beta[0,h]$ solving the IVP \eqref{eq:ivpODE} on $[0,h]$, as long as $h\in(0,T]$ satisfies
\begin{equation*}
h^\alpha\le\frac YM\bigg(\frac{\Gamma(1+\alpha)}{\Gamma(\alpha\hspace{-0.7pt}+\hspace{-0.7pt}1\hspace{-0.7pt}-\hspace{-0.7pt}\beta)}-\frac1{\Gamma(1\hspace{-0.7pt}-\hspace{-0.7pt}\beta)}\bigg).
\end{equation*}

\end{lemma}

\begin{proof}
Define the function space $U = \big\{\varphi\in C[0,h]:\,\Vert\varphi\hspace{-0.7pt}-\hspace{-0.7pt}u_0\Vert_{C[0,h]}\le Y\big\}$.
For $\varphi\in U$, we know that $f\big(x,\varphi(x)\big)\in C(0,h]$ and $\big\vert f\big(x,\varphi(x)\big)\big\vert\le Mx^{\alpha-\beta}$, so   {by Theorem \ref{thm:existence}} we can define  the following Picard iteration operator
\begin{equation*}
\mathcal P\varphi=u_0+ I_0^\beta f\big(x,\varphi(x)\big).
\end{equation*}
From \eqref{bound for I^beta_0 g} we know that $\mathcal P\varphi\in C[0,h]$, and
\begin{equation*}
\Vert\mathcal P\varphi-u_0\Vert_{C[0,h]}=\Big\Vert I_0^\beta f\big(x,\varphi(x)\big)\Big\Vert_{C[0,h]}\le Mh^\alpha\bigg(\frac{\Gamma(1+\alpha)}{\Gamma(\alpha\hspace{-0.7pt}+\hspace{-0.7pt}1\hspace{-0.7pt}-\hspace{-0.7pt}\beta)}-\frac1{\Gamma(1\hspace{-0.7pt}-\hspace{-0.7pt}\beta)}\bigg)^{-1}\le Y.
\end{equation*}
Therefore, $\mathcal P[U]\subseteq U$. In addition, for any $\varphi,\,\psi\in U$ and any $x\in(0,h]$,
\begin{equation*}
\big\vert\mathcal P\varphi(x)-\mathcal P\psi(x)\big\vert\le I_0^\beta\Big\vert f\big(x,\varphi(x)\big)-f\big(x,\psi(x)\big)\Big\vert\le I_0^\beta\Big[Lx^{\alpha-\beta}\big\vert\varphi(x)\hspace{-0.7pt}-\hspace{-0.7pt}\psi(x)\big\vert\Big].
\end{equation*}
Like the proof of Lemma \ref{uniqueness of ivpODE}, by iterating the above inequality, we obtain for all $N\in\mathbb N,\;x\in(0,h]$,
\begin{equation*}
\big\vert\mathcal P^N\varphi(x)\hspace{-0.7pt}-\hspace{-0.7pt}\mathcal P^N\psi(x)\big\vert\le L^N\big(I_0^\beta [x^{\alpha-\beta}\,\cdot\,]\big)^N\vert\varphi\hspace{-0.7pt}-\hspace{-0.7pt}\psi\vert(x)\le  \Vert\varphi\hspace{-0.7pt}-\hspace{-0.7pt}\psi\Vert_{C[0,h]}L^N\big(I_0^\beta [x^{\alpha-\beta}\,\cdot\,]\big)^N1(x).
\end{equation*}
By Lemma \ref{bound of the product}, 
there exists $N$ large enough, such that $\mathcal P^N$ is a contraction on $U$, which is a complete metric space under the metric induced by $\|\cdot\|_{C[0,h]}$. By a corollary of the Banach fixed point theorem, $\mathcal P$ has a unique fixed point $u_*\in U$. Then, by Theorem \ref{thm:existence}, $u_*$ is a solution to \eqref{eq:ivpODE} on $[0,h]$.
\end{proof}\\

With the preceding lemmata, we are ready to    establish the global  well-posedness of IVP \eqref{eq:ivpODE}.

\begin{proof} \emph{[of Proposition \ref{global existence of ivpODE for locally Lipschitz f}]}
The uniqueness and continuous dependence of the solution are already shown in Lemmata \ref{uniqueness of ivpODE} and  \ref{continuous dependence}, respectively.

By Lemma \ref{local existence of ivpODE for locally Lipschitz f}, for some $h \in(0,T]$, there exists $u_*\in C_\beta[0,h ]$ solving \eqref{eq:ivpODE} on $[0,h ]$. If $h <T$ and $\big\vert u_*(h )\hspace{-0.7pt}-\hspace{-0.7pt}u_0\big\vert< Y$, we are going to extend $u_*$ 
beyond $h$, in a manner similar to the proof of \cite[Theorem 6.8]{kai}. Let us choose $\tilde h \in(h ,T]$ such that
\begin{equation*}
\vert\tilde h-h \vert^\beta\le\frac{Y-\big\vert u_*(h )\hspace{-0.7pt}-\hspace{-0.7pt}u_0\big\vert}{2M\max\{h ^{\alpha-\beta}\!,\,T^{\alpha-\beta}\}}\left(1-\frac{\Gamma(\alpha+1-\beta)}{\Gamma(1\hspace{-0.7pt}+\hspace{-0.7pt}\alpha)\Gamma(1\hspace{-0.7pt}-\hspace{-0.7pt}\beta)}\right) \Gamma(1\hspace{-0.7pt}+\hspace{-0.7pt}\beta).
\end{equation*}
Define the function space 
\begin{equation*}
V=\Big\{\varphi\in C[0,\tilde h]:\,\Vert\varphi\hspace{-0.7pt}-\hspace{-0.7pt}u_*\Vert_{C[0,h ]}=0\text{ and }\big\Vert\varphi\hspace{-0.7pt}-\hspace{-0.7pt}u_*(h )\big\Vert_{C[h ,\tilde h]}\le Y-\big\vert u_*(h )\hspace{-0.7pt}-\hspace{-0.7pt}u_0\big\vert\Big\},
\end{equation*}
which is a complete metric space under the metric induced by $\|\cdot\|_{C[0,\tilde h]}$.

For $\varphi\in V$, we know that $f\big(x,\varphi(x)\big)\in C(0,\tilde h]$ and $\big\vert f\big(x,\varphi(x)\big)\big\vert\le Mx^{\alpha-\beta}$, so $\mathcal P\varphi\in C[0, \tilde h]$, where $\mathcal P$ is introduced in the proof of Lemma \ref{local existence of ivpODE for locally Lipschitz f}.  By Remark \ref{holder cts of D^-1}-(i), $\mathcal P\varphi\in C^{0,\beta}[h , \tilde h]$ with a H\"older constant  
\begin{equation*}
C=\frac{2M\max\{h^{\alpha-\beta}\!,\,T^{\alpha-\beta}\}}{\Gamma(1+\beta)}\left(1-\frac{\Gamma(\alpha+1-\beta)}{\Gamma(1\hspace{-0.7pt}+\hspace{-0.7pt}\alpha)\Gamma(1\hspace{-0.7pt}-\hspace{-0.7pt}\beta)}\right)^{-1}.
\end{equation*}
For all $x\in[h, \tilde h]$, $\big\vert\mathcal P\varphi(x)-\mathcal P\varphi(h )\big\vert\le C\vert x-h \vert^\beta\le C\vert \tilde h-h \vert^\beta\le Y-\big\vert u_*(h )\hspace{-0.7pt}-\hspace{-0.7pt}u_0\big\vert$. Recall that $u_*$ is already a fixed point of $\mathcal P$ on $[0,h ]$, we have $\mathcal P\varphi(h )=\mathcal Pu_*(h )=u_*(h )$. Therefore $\mathcal P\varphi\in V$ and $\mathcal P[V]\subseteq V$. Similar to the proof of Lemma \ref{local existence of ivpODE for locally Lipschitz f}, we can show that $\mathcal P^N$ is a contraction on $V$ for some large $N$, so $\mathcal P$ has a unique fixed point $v_*\in V$, which is a solution to \eqref{eq:ivpODE} on $[0, \tilde h]$. Choose $ \tilde h$ as large as possible and repeat the procedure,
then we can extend the solution all the way to the boundary of the domain. That is, the solution  must exist  on the entire interval $[0,T]$, or attains $u_0\hspace{-0.7pt}+\hspace{-0.7pt}Y$ or $u_0\hspace{-0.7pt}-\hspace{-0.7pt}Y$.
\end{proof}

\subsection{A linear  censored IVP}
 We now consider the IVP $\DCs\beta u=\lambda u$ for a constant $\lambda$. Like its counterpart in classical ODEs, such IVP can play important roles in more general equations.\label{Subsection: A linear censored IVPs}
 
\begin{lemma}\label{thm:inho}
For any $\lambda,u_0\in\mathbb R$, the linear IVP
\begin{equation}
\label{linear equation}
\left\{
\begin{aligned}
\DCs\beta u(x)&=\lambda u(x), &&x\in (0,T],\\
u(x)&=u_0, &&x=0,
\end{aligned}\right.
\end{equation}
has a unique solution in $C_\beta[0,T]$ given by the series $u(x)=u_0\sum_{N=0}^\infty\lambda^N(I^\beta_0)^N1(x)$, which is equivalent to \eqref{solution to eigenvalue problem} by letting $\alpha=\beta$ in \eqref{The formula for I0beta^n}. 
\end{lemma}
\begin{proof}  
By Lemma \ref{bound of the product}, we know that $u=u_0\sum_{N=0}^\infty\lambda^N(I^\beta_0)^N1$ converges uniformly on $[0,T]$, thus $u\in C[0,T]$  with $u(0)=u_0$.
By the continuous dependence in Theorem \ref{thm:existence},
\begin{equation*}
I^\beta_0u=u_0\,I^\beta_0\bigg[\sum_{N=0}^\infty\lambda^N(I^\beta_0)^N1\bigg]=u_0\sum_{N=0}^\infty I^\beta_0\Big[\lambda^N(I^\beta_0)^N1\Big]=u_0\sum_{N=0}^\infty\lambda^N(I^\beta_0)^{N+1}1,
\end{equation*}
therefore $u_0+\lambda\,I^\beta_0u=u$. By Theorem \ref{thm:existence}, $u\in C_\beta[0,T]$ and solves \eqref{linear equation}.
Uniqueness is a consequence of  Proposition \ref{global existence of ivpODE for locally Lipschitz f}.
\end{proof}

\begin{remark}\label{rmk:relax}
 We can obtain the solution \eqref{solution to eigenvalue problem} by Picard iteration, i.e. recursively solving the IVPs: $\DCs\beta u_{m+1}=\lambda u_m$ with $u_m(0)=u_0$ ($m=1,2,\cdots$), where $u_1(x)=u_0$ for all $x$.
\end{remark}

Although the series \eqref{solution to eigenvalue problem} looks cumbersome, it surprisingly decays at the simple algebraic rate $x^{-1-\beta}$ for $\lambda<0$. 

\begin{theorem}\label{thm:decay}
For $\lambda<0$ and $u_0>0$, the solution $u$ to \eqref{linear equation} is completely monotone  (i.e., $(-1)^nu^{(n)}\ge 0$ on $\mathbb R_+$ for $n=0,1,2,...$) and there exists a constant $C>1$ such that
\[
\frac{C^{-1}}{x^{1+\beta}}\le u(x)\le \frac { C}{ x^{1+\beta}},\quad\text{for all }x\ge 1.
\]
\end{theorem}

\begin{proof} 
The complete monotonicity will be proved in Corollary \ref{Cu}, using a probabilistic argument. The upper and lower bounds are proved in Lemmata \ref{upper bound for solutions to linear equation} and \ref{lower bound for solutions to linear equation} below, using a maximum principle argument.
\end{proof}

\begin{remark}\label{rmk:caprelax} $ $
\begin{enumerate}[(i)]
\label{comparison between censored and Caputo relaxation solution}
\item  For Caputo's counterpart of IVP \eqref{linear equation}, i.e., $\DRL\beta\big[u\hspace{-0.7pt}-\hspace{-0.7pt}u(0)\big]=\lambda u$  with $u(0)=u_0$, the solution  can be expressed in terms of the  Mittag-Leffler function
\begin{equation}
\label{Caputo relaxation solution}
u(x)=u_0\sum_{N=0}^\infty\frac{(\lambda x^\beta)^N}{\Gamma(N\beta\hspace{-0.7pt}+\hspace{-0.7pt}1)}.
\end{equation}
For $\lambda<0$ and $u_0>0$, it is completely monotone and decays  at the rate $ x^{-\beta} $ \cite[Theorem 7.3]{kai}. By contrast, the censored relaxation equation \eqref{linear equation} models a new decay regime $ x^{-1-\beta}$. (See also \cite[Page 1623]{WSJMS12} for related fractional relaxation equations, where the decay  rate is  $ x^{-\gamma}$ for some $\gamma\in(0,1)$.)

As a side note, for $\lambda ,u_0>0$, obviously both \eqref{solution to eigenvalue problem} and \eqref{Caputo relaxation solution} increase in $x$ faster than any polynomial. Indeed, for $\lambda=1$, the latter grows at the rate $ e^x $ \cite[Proposition 3.5]{gorenflo2014mittag}, and our numerical results suggest  $ \exp\{x\hspace{-0.7pt}+\hspace{-0.7pt}c\hspace{0.5pt}x^{1-\beta}\}$  for the former, where $c$ is positive and depends only on $\beta$.

\item For \eqref{Caputo relaxation solution} with $\lambda =-1,\;u_0=1$, \cite[Theorem 4]{simon2014comparing} gave the uniform estimates with optimal constants: $\big(1 +\Gamma(1\hspace{-0.7pt}-\hspace{-0.7pt}\beta)x^{\beta}\big)^{-1}\le u(x)\le\big(1+\Gamma(1\hspace{-0.7pt}+\hspace{-0.7pt}\beta)^{-1}x^{\beta}\big)^{-1}$. In Proposition \ref{CaputoRelaxation} we give what we believe to be a new and simple proof  of those bounds, using the same strategy used for the uniform bounds of \eqref{solution to eigenvalue problem}. Recently \cite[Proposition 4.12]{BSV19} gave another new proof   by showing the generalized results for a class of Kilbas–Saigo functions. Our simple proof can also be applied with few modifications to prove those generalized results (see Proposition \ref{CaputoRelaxation Kilbas Saigo}).
In Section \ref{Subsection: Other linear censored IVPs} we will use it again, to prove the uniform bounds of \eqref{solution to linear nonconstant equation}, the solution to $\DCs\beta u=\lambda x^{\alpha-\beta}u$.

\item The reason why our proof is both simple and versatile is that it involves only  maximum principle (mentioned in Remark \ref{rmk:oncfd}-(iii)) and some suitable candidate bounds (e.g. $(1\hspace{-0.7pt}+\hspace{-0.7pt}cx)^{-1-\beta}$), but no specific representation of the solution (e.g. \eqref{solution to eigenvalue problem} or \eqref{eq:exptauinf_intro}). In fact, we expect this strategy to have broader applications. As an example, consider a general Caputo-type derivative $\DCp\psi$ (so $-\DCp\psi$ generates a non-increasing pure jump L\'evy process killed upon leaving $\mathbb R^+$ \cite{Kol19}) for a L\'evy measure $\psi$ with $\int_0^\infty \min \{r,1\}\psi(\dd r)<\infty$,
\begin{equation*}
\DCp\psi u(x)=\int_0^x \big(u(x)-u(x\hspace{-0.7pt}-\hspace{-0.7pt}r)\big)\psi(\dd r)+\big(u(x)-u(0)\big)\psi\big((x,\infty)\big),
\end{equation*}
and its relaxation equation $\DCp\psi u=\lambda u\;(\lambda<0)$. The solution is given as an expectation or a series under mild assumptions \cite[Lemma 3.4]{JK16}. It is possible for our strategy to prove two-sided bounds of this solution without  those representations of it. Indeed, this has already been done for certain absolutely continuous $\psi$ (so $\psi(\dd r)=\psi(r)\,\dd r$). For instance, for compactly supported $\psi$, the  solution is given an upper bound of the decay rate $x^{-1}$ \cite[Remark 3.5]{DYZ17}.  A special case is the truncated fractional kernel $\psi(r)=\mathbf 1_{\{r\in(0,\,\delta]\}}\hspace{0.8pt}r^{-1-\beta}$ with $\delta>0$ (see \cite[Theorem 3.2]{DYZ17}, which inspired our proof). Even if  $\psi$ is not compactly supported, as long as $\int_0^\infty r\psi(r)\dd{r}<\infty$, the same argument applies. Another  instance is when $r^{1+\beta}\psi(r)$ is continuous on $\Rp$ and bounded within $[C^{-1},\,C]$ for some $C>1$,  our strategy (in Proposition \ref{CaputoRelaxation}) can still prove the two-sided bounds, both of $x^{-\beta}$ decay.

\end{enumerate}
\end{remark}


\begin{lemma}
\label{upper solutions to linear equation}
If $\lambda<0$ and  $v\in C^1(0,T]\cap C[0,T]$  satisfies $\DCs\beta v   \ge \lambda v$, then $v$ is nonnegative if $v(0)\geq 0$, 
and positive if $v(0)>0$.
\end{lemma}

\begin{proof}
 If $v(0)\geq 0$ but $v$ is not nonnegative, let $x_0$ be a minimum point of $v$ on $[0,T]$, 
then $x_0>0$ and $ v(x_0)<0$. So we have  $0<\lambda v(x_0) \leq  \DCs\beta v(x_0)$. 
However, since $v\in C^1(0,T]$,  by Remark \ref{rmk:oncfd}-(iii)  we know that
\begin{equation*}
\DCs\beta v(x_0)=\int_0^{x_0}\big(v(x_0)-v(x_0\hspace{-0.7pt}-\hspace{-0.7pt}r)\big)\frac{ r^{-1-\beta}}{\big|\Gamma(-\beta)\big|}\,{\rm d}r\leq 0,
\end{equation*}
which is a contradiction.  Similarly, we can prove that $v$ is positive if $v(0)>0$.
\end{proof}

With the above lemma, we can get  the desired bounds for the solution to \eqref{linear equation}.

\begin{lemma}
\label{upper bound for solutions to linear equation}
For $\lambda<0$ and $u_0=1$, the solution $u$ to \eqref{linear equation} is positive and can be bounded from above by   $v(x)=(1+c \vert\lambda\vert^{1/\beta}x)^{-1-\beta}$, where
\begin{equation}\label{eq:cbeta}
c =\frac{\big\vert\Gamma(-\beta)\big\vert^{1/\beta}}2\Big(\frac{2^{1+\beta}-1}{1-\beta}+\frac2\beta\Big)^{-1/\beta}.
\end{equation}
\end{lemma}

\begin{proof}  We know from  Lemma \ref{thm:inho} that $u\in C[0,T]$.  We also know that $u\in C^1(0,T]$ from the uniform convergence of the series representation of its derivative on any closed interval not containing 0. Thus   $u$ remains positive by Lemma \ref{upper solutions to linear equation}.

Let $v(x)=(1+x/c)^{-1-\beta}$ and first assume that there is a constant $c>0$  such that $v$   satisfies the condition in Lemma \ref{upper solutions to linear equation}.  Under this assumption, we get $\DCs\beta(v\hspace{-0.7pt}-\hspace{-0.7pt}u)-\lambda(v\hspace{-0.7pt}-\hspace{-0.7pt}u)\ge0$ with $v(0)\hspace{-0.7pt}-\hspace{-0.7pt}u(0)=0$. By Lemma \ref{upper solutions to linear equation}, we have    $v\geq u$ on $[0,T]$.

Now, given $\beta\in(0,1)$ and $\lambda<0$,   up to a constant multiple, 
it remains to find a constant $c>0$ such that $v(x)=(x+c)^{-1-\beta}$ satisfies  $\DCs\beta v\ge \lambda v$, i.e., for all $x>0$,
\begin{equation*}
\int_0^x\big(v(x)-v(x\hspace{-0.7pt}-\hspace{-0.7pt}r)\big)\frac{ r^{-1-\beta}}{\big|\Gamma(-\beta)\big|}\,{\rm d}r  \ge \lambda v(x),
\end{equation*}
  or equivalently, for all $x>0$,
\begin{equation}
\label{equivalent inequality}
\big\vert\lambda\Gamma(-\beta)\big\vert\ge\int_0^x\bigg(\frac{(x+c)^{1+\beta}}{(x\hspace{-0.7pt}+\hspace{-0.7pt}c\hspace{-0.7pt}-\hspace{-0.7pt}r)^{1+\beta}}-1\bigg)\frac{{\rm d}r}{r^{1+\beta}} .
\end{equation}
Let $y=x+c$, then the   right-hand side  of \eqref{equivalent inequality} equals
\begin{equation}
\label{inequality on the first subinterval}
\int_0^{y-c}\bigg(\frac{y^{1+\beta}}{(y\hspace{-0.7pt}-\hspace{-0.7pt}r)^{1+\beta}}-1\bigg)\frac{{\rm d}r}{r^{1+\beta}}=y^{-\beta}\int_0^{1-c/y}\bigg(\frac1{(1\hspace{-0.7pt}-\hspace{-0.7pt}s)^{1+\beta}}-1\bigg)\frac{{\rm d}s}{s^{1+\beta}}.
\end{equation}
If $y\le2c$, then  the right-hand side of \eqref{inequality on the first subinterval} can be bounded from above by
\begin{equation*}
y^{-\beta}\int_0^{1/2}\bigg(\frac1{(1\hspace{-0.7pt}-\hspace{-0.7pt}s)^{1+\beta}}-1\bigg)\frac{{\rm d}s}{s^{1+\beta}}\le y^{-\beta}\int_0^{1/2}2s(2^{1+\beta}-1)\frac{{\rm d}s}{s^{1+\beta}}
=\frac{2^\beta}{y^\beta}\frac{2^{1+\beta}-1}{1-\beta}.
\end{equation*}
If $y>2c$, we split the interval $[0,1\hspace{-0.7pt}-\hspace{-0.7pt}c/y]$ into two parts $[0,1/2]$ and $[1/2,1\hspace{-0.7pt}-\hspace{-0.7pt}c/y]$. For the second subinterval,
\begin{equation*}
\int_{1/2}^{1-c/y}\bigg(\frac1{(1\hspace{-0.7pt}-\hspace{-0.7pt}s)^{1+\beta}}-1\bigg)\frac{{\rm d}s}{s^{1+\beta}}\le2^{1+\beta}\int_{1/2}^{1-c/y}\frac{{\rm d}s}{(1\hspace{-0.7pt}-\hspace{-0.7pt}s)^{1+\beta}}\le\frac{2^{1+\beta}}\beta\Big(\frac yc\Big)^\beta.
\end{equation*}
Therefore the right-hand side of \eqref{equivalent inequality} can be bounded from above by
\begin{equation*}
\frac{2^\beta}{y^\beta}\frac{2^{1+\beta}-1}{1-\beta}+\frac{2^{1+\beta}}{y^\beta\beta}\Big(\frac yc\Big)^\beta\le\frac{2^\beta}{c^\beta}\frac{2^{1+\beta}-1}{1-\beta}+\frac{2^{1+\beta}}{c^\beta\beta}.
\end{equation*}
Let
\begin{equation*}
c\ge\frac{2}{\big\vert\lambda\Gamma(-\beta)\big\vert^{1/\beta}}\Big(\frac{2^{1+\beta}-1}{1-\beta}+\frac2\beta\Big)^{1/\beta},
\end{equation*}
then \eqref{equivalent inequality} will be satisfied and we are done.
\end{proof}

\begin{lemma}
\label{lower bound for solutions to linear equation}
If $\lambda<0$ and $u_0=1$, then the solution to \eqref{linear equation} is bounded from below by $w(x)=(1+d \vert\lambda\vert^{1/\beta}x)^{-1}(1+d^\beta\vert\lambda\vert x^\beta)^{-1}$, where $d = \big\vert\Gamma(-\beta)\big\vert^{1/\beta}\max\big\{4,\,(1\hspace{-0.7pt}-\hspace{-0.7pt}\beta)(1\hspace{-0.7pt}+\hspace{-0.7pt}2^\beta)/\beta\big\}^{1/\beta}$.
\end{lemma}
 
The proof is similar to that of Lemma \ref{upper bound for solutions to linear equation}, so we put it in Appendix \ref{maximum principle proofs}.

\subsection{Other linear censored IVPs}\label{Subsection: Other linear censored IVPs}

We conclude this section by generalizing IVP \eqref{linear equation} to the inhomogeneous version with a variable coefficient. 

\begin{proposition}\label{lem:inhomlin}
Let $\lambda,u_0\in\mathbb R$ and $g\in C(0,T]$ such that $\big\vert g(x)\big\vert\le M x^{\alpha-\beta}$ for some $M,\alpha>0$ and all $x\in(0,T]$. Then the inhomogeneous linear IVP
\begin{equation}
\label{linear in equation}
\left\{
\begin{aligned}
\DCs\beta u(x)&=\lambda u(x)+g(x), &x&\in (0,T],\\
u(x)&=u_0, &x&=0,
\end{aligned}
\right.
\end{equation}
has a unique solution in $C_\beta[0,T]$ given by the following series
\begin{equation}
\label{solution to inh eigenvalue problem}
u(x)=u_0 \sum_{N=0}^\infty \lambda^N(I^\beta_0)^N1(x) +\sum_{N=0}^\infty \lambda^N(I^\beta_0)^{N+1}g(x),
\end{equation}
and from Proposition \ref{global existence of ivpODE for locally Lipschitz f}, $u$ inherits the continuous dependence on $u_0$ and $g$.  
\end{proposition}
 
\begin{proof}
From \eqref{bound for I^beta_0 g} we know $I_0^\beta g\in C[0,T]$ and thus $(I_0^\beta)^{N+1} g\in C[0,T]$ for $N=0,1,2,\cdots$. Then, by the positivity preserving property of $I^\beta_0$, we have $$\big\vert(I^\beta_0)^{N+1}g\big\vert \le (I^\beta_0)^N\vert I^\beta_0g\vert \le (I^\beta_0)^N1\cdot\Vert I^\beta_0g\Vert_{C[0,T]}.$$
By Lemma \ref{bound of the product},  the series $\sum_{N=0}^\infty \lambda^N(I^\beta_0)^N1$ converges uniformly on $[0,T]$, and  so does $\sum_{N=0}^\infty \lambda^N(I^\beta_0)^{N+1}g$. So the function $u$ given by \eqref{solution to inh eigenvalue problem} is in $C[0,T]$ and $I^\beta_0u$ is well-defined. Therefore,
\begin{equation*}
\begin{aligned}
\lambda\,I^\beta_0u+I^\beta_0g&=\lambda u_0\,I^\beta_0\sum_{N=0}^\infty \lambda^N(I^\beta_0)^N1+\lambda I^\beta_0\sum_{N=0}^\infty \lambda^N(I^\beta_0)^{N+1}g+I^\beta_0g\\
&=\lambda u_0\sum_{N=0}^\infty \lambda^N(I^\beta_0)^{N+1}1+\lambda\sum_{N=0}^\infty \lambda^N(I^\beta_0)^{N+2}g+I^\beta_0g\\
&=u-u_0,
\end{aligned}
\end{equation*}
where the second equality is due to the continuous dependence in Theorem \ref{thm:existence}. Using Theorem \ref{thm:existence} again, we know that $u$ solves \eqref{linear in equation}. By Proposition \ref{global existence of ivpODE for locally Lipschitz f}, $u$ is actually the unique solution in $C_\beta[0,T]$, and depends on $(u_0,g)$ continuously.  
\end{proof}

\begin{lemma}
For $\lambda,u_0\in\mathbb R$ and $\alpha>0$, the linear IVP
\begin{equation}
\label{linear nonconstant equation}
\left\{
\begin{aligned}
\DCs\beta u(x)&=\lambda x^{\alpha-\beta}u(x), &&x\in (0,T],\\
u(x)&=u_0, &&x=0.
\end{aligned}
\right.
\end{equation}
has a unique solution in $C_\beta[0,T]$ given by the following series
\begin{equation}
\label{solution to linear nonconstant equation}
\begin{aligned}
u(x)& =u_0\sum_{N=0}^\infty\lambda^N\big(I_0^\beta [x^{\alpha-\beta}\,\cdot\,]\big)^N1(x) \\
& =u_0\sum_{N=0}^\infty(\lambda x^{\alpha})^N\prod_{n=1}^N\bigg(\frac{\Gamma(1+n\alpha)}{\Gamma(n\alpha\hspace{-0.7pt}+\hspace{-0.7pt}1\hspace{-0.7pt}-\hspace{-0.7pt}\beta)}-\frac1{\Gamma(1\hspace{-0.7pt}-\hspace{-0.7pt}\beta)}\bigg)^{-1}.
\end{aligned}
\end{equation}
\end{lemma}

We omit the proof since it is a special case of Proposition \ref{lem:inhomvarlin}. Surprisingly, the solution \eqref{solution to linear nonconstant equation} has a decay property analogous to what we see in Section \ref{Subsection: A linear censored IVPs}.

\begin{proposition}\label{thm:decay 1+alpha}
For $\lambda<0$ and $u_0>0$, there exists a constant $C>1$ such that the solution $u$ to \eqref{linear nonconstant equation} satisfies
\[
\frac{C^{-1}}{x^{1+\alpha}}\le u(x)\le \frac { C}{ x^{1+\alpha}},\quad\text{for all }x\ge 1.
\]
\end{proposition}

\begin{proof}
See Lemmas \ref{upper bound for solutions to linear nonconstant equation} and \ref{lower bound for solutions to linear nonconstant equation}, which can be shown by maximum principle arguments.
\end{proof}

\begin{remark}$ $
\begin{enumerate}[(i)]
    \item For Caputo's counterpart of IVP \eqref{linear nonconstant equation}, the solution can be expressed in terms of the Kilbas–Saigo function
\begin{equation}
\label{solution to Caputo linear nonconstant equation}
u(x)=u_0\sum_{N=0}^\infty(\lambda x^{\alpha})^N\prod_{n=1}^N\left(\frac{\Gamma(1+n\alpha)}{\Gamma(n\alpha\hspace{-0.7pt}+\hspace{-0.7pt}1\hspace{-0.7pt}-\hspace{-0.7pt}\beta)}\right)^{-1}.
\end{equation}

For $\lambda<0$ and $u_0>0$, the  solution  \eqref{solution to Caputo linear nonconstant equation} decays   at the rate $ x^{-\alpha}$ \cite[Remark 4.6 (c)]{BSV19} (and is completely monotone \cite[Remark 3.1 (d)]{BSV19} if $\alpha\le1$). On the other hand, the censored IVP \eqref{linear nonconstant equation} once again models a new decay regime $ x^{-1-\alpha} $.

As a side note, for $\lambda,u_0>0$, both \eqref{solution to linear nonconstant equation} and \eqref{solution to Caputo linear nonconstant equation} increase in $x$ faster than any polynomial. Indeed, for $\lambda =1$, the latter can be bounded by $\exp\big\{(\frac\beta\alpha\hspace{-0.5pt}+\hspace{-0.5pt}\varepsilon)\hspace{0.5pt} x^{\alpha/\beta}\big\}$ for any $\varepsilon$ positive and $x$ large enough \cite[Theorem 5.9]{gorenflo2014mittag}, and our numerical results suggest the same for the former. 

\item For \eqref{solution to Caputo linear nonconstant equation} with $\lambda =-1,\;u_0=1$,  \cite[Proposition 4.12]{BSV19} proved the uniform bounds $$\big(1+\Gamma(1\hspace{-0.7pt}-\hspace{-0.7pt}\beta)x^\alpha\big)^{-1}\le u(x)\le\big(1+\Gamma(1\hspace{-0.7pt}+\hspace{-0.7pt}\alpha\hspace{-0.7pt}-\hspace{-0.7pt}\beta)\Gamma(1\hspace{-0.7pt}+\hspace{-0.7pt}\alpha)^{-1}x^\alpha\big)^{-1}.$$   As mentioned in Remark \ref{rmk:caprelax}-(ii), our maximum principle argument can  give a new and simple proof of those bounds.

\item For $\alpha=1$, \eqref{linear nonconstant equation} can be seen as a linear equation $\sigma \DCs\beta u=\lambda u$, where we let $\sigma(x)=x^{\beta-1}$ so that the rescaled fractional derivative $\sigma \DCs\beta$ acts like the classical first order derivative on linear functions. This kind of rescaling naturally extends to more general nonlocal derivatives, and we refer to \cite{du19} for a discusison of nonlocal calculus and rescaling.

\end{enumerate}
\end{remark}

\begin{lemma}
\label{upper bound for solutions to linear nonconstant equation}
For $\lambda<0$ and $u_0=1$, the solution to \eqref{linear nonconstant equation} is positive and can be bounded from above by $v(x)=\big(1+(c\vert\lambda\vert^{1/\alpha}x)^{1+\alpha}\big)^{-1}$, where
\begin{equation*}
c=\frac{\big\vert\Gamma(-\beta)\big\vert^{1/\alpha}}{2^{\beta/\alpha}}\bigg(\frac{2^{1+\alpha}}{\alpha}+\frac{\alpha^{\alpha/(1+\alpha)}}{1-\beta}\bigg)^{-1/\alpha}.
\end{equation*}
\end{lemma}

The proof is similar to that of Lemma \ref{upper bound for solutions to linear equation}, so we put it in Appendix \ref{maximum principle proofs}.

\begin{lemma}
\label{lower bound for solutions to linear nonconstant equation}
If $\lambda<0$ and $u_0=1$, then the solution to \eqref{linear nonconstant equation} is bounded from below by $w(x)=(1+d\vert\lambda\vert^{1/\alpha}x)^{-1}(1+d^\alpha\vert\lambda\vert x^\alpha)^{-1}$, where $d=\big\vert\Gamma(-\beta)\big\vert^{1/\alpha}\max\big\{4,\,(1\hspace{-0.7pt}+\hspace{-0.7pt}2^\alpha)(1\hspace{-0.7pt}-\hspace{-0.7pt}\beta)/\alpha\big\}^{1/\alpha}$.
\end{lemma}

The proof is parallel to that of Lemma \ref{lower bound for solutions to linear equation}, so we omit it.

\begin{proposition}\label{lem:inhomvarlin}
Let $\lambda,u_0\in\mathbb R,\;\alpha>0$ and $g\in C(0,T]$ such that $\big\vert g(x)\big\vert\le M x^{\gamma-\beta}$ for some $M,\gamma>0$ and all $x\in(0,T]$. Then the inhomogeneous linear IVP
\begin{equation}
\label{linear variable in equation}
\left\{
\begin{aligned}
\DCs\beta u(x)&=\lambda x^{\alpha-\beta}u(x)+g(x), &x&\in (0,T],\\
u(x)&=u_0, &x&=0,
\end{aligned}
\right.
\end{equation}
has a unique solution in $C_\beta[0,T]$ given by the following series
\begin{equation}
\label{solution to variable inh eigenvalue problem}
u(x)=u_0 \sum_{N=0}^\infty \lambda^N\big(I_0^\beta [x^{\alpha-\beta}\,\cdot\,]\big)^N1(x) +\sum_{N=0}^\infty \lambda^N\big(I_0^\beta [x^{\alpha-\beta}\,\cdot\,]\big)^NI^\beta_0g(x).
\end{equation}
From Proposition \ref{global existence of ivpODE for locally Lipschitz f}, $u$ inherits the continuous dependence on $u_0$ and $g$.  
\end{proposition}

The proof is parallel to that of Proposition \ref{lem:inhomlin}, so we omit it. A quick check can be done by Picard iteration.

\section{Censored decreasing \texorpdfstring{$\beta$}{beta}-stable process }\label{sec:hit_0}

In this section we first  prove that the hitting time of 0 (or lifetime) for the censored decreasing $\beta$-stable process is finite and   that $I_0^\beta$ has probabilistic representations    \eqref{eq:serintro} and \eqref{eq:fkintro}. We then use these results to prove that our censored process is Feller with generator $-\DCs\beta$, which in turn leads us to   show that the Laplace transform of the lifetime equals the series  \eqref{solution to eigenvalue problem}, and  thus they are completely monotone. We denote by $\mathbf 1_A$ the indicator function of a set $A$. All our stochastic processes are real-valued right-continuous with left limits (c\`adl\`ag), hence we always assume the   canonical underlying filtered probability space as in \cite[Chapter O]{bertoin}. For a stochastic process $Y=\{Y_s\}_{s\ge0}$ and a real-valued integrable function $f$ on the probability space of $Y$, we use the  notation $\mathbb E_y\big[f(Y)\big]=\mathbb E\big[f(Y)\,\big|\,Y_0\!=\!y\big]$, $\mathbb E\big[f(Y)\big]=\mathbb E_0\big[f(Y)\big]$, and correspondingly $\mathbb P_y[A]$, $\mathbb P[A]$ when $f=\mathbf 1_A$. We write $Y_{t-}=\lim_{s\uparrow t}Y_s$. By a \emph{$\beta$-stable subordinator} ($\beta\in(0,1)$) we mean the L\'evy process $-S^1=\{-S_s^1\}_{s\ge0}$ characterised by the Laplace transform $\mathbb E\big[\!\exp\{kS^1_s\}\big]=\exp\{-sk^\beta\}$, $k,s>0$ \cite[Chapter III]{bertoin}.  
We denote by $B[0,T]$ the set of real-valued bounded Borel measurable functions on $[0,T]$ and define $C_\infty(0,T]=\big\{u\in C[0,T]: u(0)=0\big\}$, both understood as Banach spaces  with the sup norm. We extend the domain of any $f\in B[0,T]$  to a cemetery state $\partial$ imposing $f(\partial )=0$.  As discussed in Section \ref{Section: Introduction}, we treat the censored decreasing stable process in $\Rp$ because it is generated by $-\DCs\beta$, where $\DCs\beta$ is the ``left'' censored derivative  (at $0$).
However, it should be clear that all the results in this section translate immediately to the censored stable subordinator in $(-\infty, b)$ when  paired with the  ``right'' censored derivative at $b\in\mathbb R$.
 
\subsection{Construction and finite lifetime}\label{sec:hit_0_1}
The starting point of the censored process is always assumed to be fixed to some $x>0$.
We define the \emph{censored decreasing $\beta$-stable process} $S^c$ by the INW piecing together construction, then \cite[Theorem 1.1 and Section 5.i]{INW66} guarantees us a c\`adl\`ag strong (sub-)Markov process. The construction is:  run $x\hspace{-0.7pt}+\hspace{-0.7pt}S_t^1$   until $\tau_1$, the  time when it first exits $(0,T]$, where $-S^1$ is a $\beta$-stable subordinator (started at 0); then kill the process if $x\hspace{-0.7pt}+\hspace{-0.7pt}S^1_{\tau_1-}\le0$; otherwise piece together an independent copy of $S^1$ started at $x\hspace{-0.7pt}+\hspace{-0.7pt}S^1_{\tau_1-}$  and repeat the same procedure for at most countably many times. 

With Lemma \ref{lem:1} we prove that we can directly define the censored decreasing $\beta$-stable process $S^{c}\,|\,S^c_0\!=\!x$  as
\begin{equation}\label{eq:cen}
S^{c}_{t} :=\left\{\begin{aligned}&  \widetilde S^{j}_{t} , &&  \tau_{j\hspace{-0.8pt}-\!1}\le t< \tau_{j},\;j\in\mathbb N,\\
&\partial, && t\ge\tau_\infty,
\end{aligned}
\right.
\end{equation}
with
\begin{equation*}
\widetilde S^{j}_{t} :=\left\{
\begin{aligned}
&x+ S^j_t, &&j=1,\\
&\widetilde S^{j-1}_{\tau_{j\hspace{-0.8pt}-\!1}-}+S^{j}_{t-\tau_{j\hspace{-0.8pt}-\!1}}, &&j\ge 2,
\end{aligned}\right.
\quad\text{and}\quad
\tau_{j}:=\left\{
\begin{aligned}
&\;0, &&j=0,\\
&\,\inf\big\{s>\tau_{j\hspace{-0.8pt}-\!1}: \widetilde S^{j}_{s}\le 0\big\}, &&j\in\mathbb N,\\
&\lim_{j\to\infty}\tau_{j}, &&j=\infty,
\end{aligned}\right.
\end{equation*}
where $\{-S^{j}\}_{j\in\mathbb N}$ is an \textit{i.i.d.} collection of $\beta$-stable subordinators. Recall \cite[Chapter III]{bertoin} the expectation of the inverse stable subordinator
\begin{equation}
\label{eq:id}
\mathbb E\big[E_1(y)\big]=y^\beta/\Gamma(\beta\hspace{-0.7pt}+\hspace{-0.7pt}1),\;\text{where}\;E_j(y):=\inf\big\{s>0:\, y<-S^j_s\big\}, \;j\in\mathbb N,\;y>0.
\end{equation}

\begin{lemma}\label{lem:1} For any $x>0$ and $j\in\mathbb N$, assuming $S^c_0=x$, we have
\begin{enumerate}[(i)]
    \item $\mathbb E_x[\tau_{j}]<\infty$, $\mathbb P_x\big[S^c_{\tau_{j}}\in(0,x)\big]=1$ and $S^c_{\tau_{j}}$ has the density $k_j(x,\,\cdot\,)$, as defined in \eqref{eq:kern};

\item $ S^c_{\tau_{j}}>0$, and \eqref{eq:cen} 
equals the INW construction of the censored decreasing $\beta$-stable process;
\item $ $\\[-28pt]
\begin{align}
\mathbb E_x[\tau_{j+1}\hspace{-0.7pt}-\hspace{-0.7pt}\tau_j ]&=\mathbb E_x\big[E_{j+1}(S^c_{\tau_{j}})\big]=\int_0^x \frac{y^\beta}{\Gamma(\beta\hspace{-0.7pt}+\hspace{-0.7pt}1)}k_{j}(x,y)\,{\rm d}y;
\label{eq:expE}
\end{align}

\item $\mathbb P_x[\tau_\infty<\infty]=1$ and $\mathbb P_x\big[S^c_{\tau_\infty-}=0\big]=1$.
\end{enumerate}
\end{lemma}

\begin{proof} The statement (ii) follows immediately from (i). We now prove (i) by induction.
For $j=1$, $\mathbb E_x[\tau_1]=\mathbb E\big[E_1(x)\big]=x^{\beta}/\Gamma(\beta\hspace{-0.7pt}+\hspace{-0.7pt}1)<\infty$, and it is known that $S^c_{\tau_1}=x+S^1_{\tau_1-}$ is beta-distributed on $(0,x)$ with density $ k_1(x,\,\cdot\,)$ \cite[Chapter III, Proposition 2]{bertoin}. Then we perform induction for each $j\ge1$: since $\tau_{j}<\infty$, $S^{c}_{\tau_{j}}>0$ and $S^{j+1}$ is independent of $(S^{c}_{\tau_{j}},\tau_{j})$, we have
\begin{align}
\tau_{j+1}-\tau_{j}&=\inf\big\{s>\tau_{j}: S^{c}_{\tau_{j}}<-S^{j+1}_{s-\tau_{j}}\big\}-\tau_{j}\nonumber\\
&=\inf\big\{r>0: S^{c}_{\tau_{j}}<-S^{j+1}_{r}\big\}\nonumber\\
&=E_{j+1}(S^{c}_{\tau_{j}}).\label{eq:com}
\end{align}
Combining \eqref{eq:com}  with $S^{c}_{\tau_{j}}<x$  and \eqref{eq:id}, we obtain
\[
\mathbb E_x[\tau_{j+1}]=\mathbb E_x\big[E_{j+1}(S^{c}_{\tau_{j}})\big]+\mathbb E_x[\tau_{j}]\le \mathbb E\big[E_{j+1}(x)\big]+\mathbb E_x[\tau_{j}]<\infty.
\]
By definition and \eqref{eq:com}, we have
\begin{equation*}
S^{c}_{\tau_{j\hspace{-0.5pt}+\hspace{-0.5pt}1}}=S^{c}_{\tau_{j}}+S^{j+1}_{(\tau_{j\hspace{-0.5pt}+\hspace{-0.5pt}1}-\tau_{j})-}=S^{c}_{\tau_{j}}+S^{j+1}_{E_{j\hspace{-0.5pt}+\hspace{-0.5pt}1}(S^{c}_{\tau_{j}})-}\in(0,S^{c}_{\tau_{j}})\subseteq(0,x).
\end{equation*}
Therefore for any bounded measurable $f$, we have
\begin{align*}
\mathbb E_x\Big[f\big( S^{c}_{\tau_{j+1}}\big)\Big]&=\mathbb E_x\bigg[f\Big(S^{c}_{\tau_{j}}+ S^{j+1}_{E_{j+1}(S^{c}_{\tau_{j}})-}\Big)\bigg]\\
&=\int_0^x\left(\int_0^y f(z) k_1(y,z)\,{\rm d}z\right)k_{j}(x,y)\,{\rm d}y\\
&=\int_0^xf(z)\left(\int_z^x  k_{j}(x,y)\,k_1(y,z)\,{\rm d}y\right) {\rm d}z,
\end{align*}
where the second equality holds because $S^{c}_{\tau_{j}}$ is independent of $S^{j+1}$ and has the density $ k_j(x,\,\cdot\,)$; the last equality is due to Fubini's theorem. By Remark \ref{rmk:intto1} we know that $S^c_{\tau_{j\hspace{-0.5pt}+\hspace{-0.5pt}1}}$ has the density $ k_{j\hspace{-0.5pt}+\hspace{-0.5pt}1}(x,\,\cdot\,)$. The induction step is now complete.

For part (iii), by \eqref{eq:com} we have $\mathbb E_x[\tau_{j+1}\hspace{-0.7pt}-\hspace{-0.7pt}\tau_j]=\mathbb E_x\big[E_{j+1}(S^c_{\tau_{j}})\big]$, meanwhile, since $S^{j+1}$ is independent of $S^c_{\tau_{j}}$, by \eqref{eq:id} we have
\begin{align*}
\mathbb E_x\Big[E_{j+1}\big(S^c_{\tau_{j}}\big) \Big]&=\int_0^x \mathbb E\big[E_{j+1}(y)\big]k_{j}(x,y)\,{\rm d}y=\int_0^x \frac{y^\beta}{\Gamma(\beta+1)}k_{j}(x,y)\,{\rm d}y.
\end{align*}
We now prove part (iv). The results obtained so far are enough to derive Theorem \ref{thm:main7} below, which immediately implies that $\mathbb P_x[\tau_\infty<\infty]=1$. To prove $\mathbb P_x[S^c_{\tau_\infty-}>0]=0$, first, observe that 
\[
\mathbb P_x\big[S^c_{\tau_\infty-}>0\big]\le \sum_{n=1}^\infty \mathbb P_x\big[S^c_{\tau_\infty-}\ge n^{-1}\big],
\] 
and for each $n\in\mathbb N$
\begin{align*}
\mathbb P_x\big[S^c_{\tau_\infty-}\ge n^{-1}\big] &=\mathbb P_x\bigg[\bigcap_{j=1}^\infty\big\{  S^c_{\tau_{j}}\ge n^{-1}\big\}\bigg]=\lim_{j\to\infty } \mathbb P_x\big[ S^c_{\tau_{j}}\ge n^{-1}\big],
\end{align*}
where we used $\{  S^c_{\tau_{j}}\ge n^{-1}\}\supseteq \{  S^c_{\tau_{j+1}}\ge n^{-1}\}$ for each $j\in\mathbb N$ and convergence  from above of finite measures. Then, Chebyshev's inequality and the above results guarantee that
\begin{align*}
\frac{1}{n}\mathbb P_x\big[ S^c_{\tau_{j}}\ge n^{-1}\big] \le  \mathbb E_x\big[ S^c_{\tau_{j}}\big]=\int_0^x k_j(x,y)\hspace{0.5pt}y\,{\rm d}y,
\end{align*}
 and the right-hand side goes to $0$ as $j\to\infty$ by Lemma \ref{thm:kseries}.
\end{proof}

We can now prove our main result of this subsection, which gives \eqref{eq:ltintro}.

\begin{theorem}\label{thm:main7} The hitting time of $0$ of the censored $\beta$-stable L\'evy process \eqref{eq:cen} is finite in expectation, with $\mathbb E_x[\tau_\infty]=\mathbb E_x[\tau_1]\big(1-\sin(\beta\pi)/(\beta\pi)\big)^{-1},\, x>0.$
\end{theorem}

\begin{remark}\label{lem:ind}
Our key ingredient for proving Theorem \ref{thm:main7} is the following closed formula for \eqref{eq:expE} (obtained in the proof of Lemma \ref{thm:kseries}) 
\begin{equation*}
\int_0^xy^\beta k_j(x,y)\,{\rm d}y=x^\beta \big(\Gamma(\beta\hspace{-0.7pt}+\hspace{-0.7pt}1)\hspace{0.5pt}\Gamma(1\hspace{-0.7pt}-\hspace{-0.7pt}\beta)\big)^{-j} ,\quad\text{for all $j\in\mathbb N$ and $x>0$}.
\end{equation*}
\end{remark}

\begin{proof} \emph{[of Theorem \ref{thm:main7}]} On the one hand, by Monotone Convergence Theorem, $\mathbb E_x[\tau_\infty]=\lim\limits_{j\to\infty}\mathbb E_x[\tau_{j+1}]$.  
On the other hand, by   \eqref{eq:id}, \eqref{eq:expE} and Remark \ref{lem:ind},  for each $j\in\mathbb N$, 
\begin{equation*}
\mathbb E_x[\tau_{j+1}]=\mathbb E_x[\tau_{1}]+\sum_{i=1}^{j} \mathbb E_x[\tau_{ i+1}\hspace{-0.7pt}-\hspace{-0.7pt}\tau_{ i}] = \frac{x^\beta}{\Gamma(\beta\hspace{-0.7pt}+\hspace{-0.7pt}1)} \sum_{i=0}^{j}\big(\Gamma(\beta\hspace{-0.7pt}+\hspace{-0.7pt}1)\hspace{0.5pt}\Gamma(1\hspace{-0.7pt}-\hspace{-0.7pt}\beta)\big)^{-i},
\label{eq:eq}
\end{equation*}
and as $\Gamma(\beta\hspace{-0.7pt}+\hspace{-0.7pt}1)\hspace{0.5pt}\Gamma(1\hspace{-0.7pt}-\hspace{-0.7pt}\beta)= \beta\pi/\sin(\beta\pi)>1$, the result follows letting $j\to \infty$.
\end{proof}

\begin{remark}\label{rmk:tau} $ $
\begin{enumerate}[(i)]
 \item Theorem \ref{thm:main7} is not obvious. For instance, the censored \emph{symmetric} $\beta$-stable   process for $\beta\in(0,1)$ never hits the boundary, whether the  censoring is performed in an interval or $\Rp$ \cite[Theorem 1.1-(1)]{BBC03}.
 \item Any compound Poisson process in $\mathbb R^d$ censored upon exiting an open set must have infinite lifetime, and so does a non-increasing compound Poisson process censored in $(0,T]$. This is because the lifetime can be bounded from below by $\sum_{n=1}^\infty e_n=\infty$, where $\{e_n\}_{n\in\mathbb N}$ is an infinite subset of the \textit{i.i.d.} exponential waiting times of the process.
 \item The censored gamma subordinator with L\'evy measure $\psi(r)=e^{-r}/r$ \cite[Example 5.10]{Bogdan} seems to have infinite lifetime, because our numerical simulations indicate pathwise that $\tau_j\approx2\sqrt{j/3}$ and $S^c_{\tau_j}\approx\exp{-\sqrt{3j}}$ for $x=1$ and $j\gg1$. We do not know whether other censored (driftless) subordinators hit the barrier in finite time. If they do, it is not clear if our proof strategy can be extended to such cases, as it relies on the closed formula for the potential kernel, which is only available for the stable case.
\end{enumerate}
\end{remark}

\subsection{\texorpdfstring{Probabilistic representations of \texorpdfstring{$I^\beta_0$}{I0beta} }{Probabilistic representations}}

Firstly, we  prove that $I^\beta_0$  is equal to the  potential of the semigroup of the censored process $S^c$. Secondly, we give a representation of $I^\beta_0$ in terms of  products of \textit{i.i.d.} beta-distributed random variables. 

\begin{proposition}\label{thm:FKform}
If $g$ satisfies the assumption in Theorem \ref{thm:existence} or if $g\in B[0,T]$,  it holds that $I^\beta_0 g\in C_\infty(0,T]$, and for all $x\in(0,T]$ we have the identity
\begin{equation}
I^\beta_0 g(x) =  \mathbb E_x\bigg[\int_0^{\tau_\infty} g(S^c_s)\,{\rm d}s\bigg].
\label{eq:thisthm}
\end{equation}
\end{proposition}
\begin{proof}  For $g\ge 0$ we justify the following  equalities
\begin{align}
\mathbb E_x\bigg[\int_0^{\tau_\infty} \hspace{-8pt}g(S^c_s)\,{\rm d}s\bigg]
&=\sum_{j=0}^\infty\mathbb E_x\bigg[\int_{0}^{\tau_{j\hspace{-0.5pt}+\hspace{-0.5pt}1}-\hspace{0.5pt}\tau_j}\hspace{-10pt}g(S^c_{\tau_j+s})\,{\rm d}s\bigg]\nonumber\\
&=\sum_{j=0}^\infty\mathbb E_x\bigg[\int_{0}^{E_{j\hspace{-0.5pt}+\hspace{-0.5pt}1}(S^c_{\tau_j})} \hspace{-7pt}g(S^c_{\tau_j}+ S^{j+1}_s)\,{\rm d}s\bigg]\nonumber\\
&=\sum_{j=0}^\infty\mathbb E_x\bigg[\mathbb E\bigg[\int_{0}^{E_{j\hspace{-0.5pt}+\hspace{-0.5pt}1}(S^c_{\tau_j})} \hspace{-7pt}g(S^c_{\tau_j}\hspace{-0.7pt}+\hspace{-0.7pt}S^{j+1}_s)\,{\rm d}s\,\Big|\,S^c_{\tau_j}\bigg]\bigg]\nonumber\\
&=\sum_{j=0}^\infty\mathbb E_x \bigg[ J^\beta_0 g(S^c_{\tau_j})\bigg]=\sum_{j=0}^\infty\mathcal K^j J^\beta_0 g(x)=I^\beta_0g(x).\label{sum of expectation of S^c_{tau_j} equals I^beta_0}
\end{align}

The first equality is an application of 
 Tonelli's Theorem and a simple change of variables; the second follows from \eqref{eq:com}; the third is due to the law of total expectation; the fourth is due to  the independence of  $S^{j+1}$ and  $S^c_{\tau_j}$ along with the known identity \eqref{eq:Jbetaintro}   (which is a straightforward consequence of \cite[Eq. (1.38)]{Bogdan}); the fifth follows from  Lemmata \ref{lem:1}-(i) and \ref{thm:kseries}; the last follows the definition of $I^\beta_0$.   If $g\in B[0,T]$, recalling that $J^\beta_0 |g|(x)\le \sup\big\{|g(y)|:y\in[0,x]\big\}\hspace{0.8pt}x^\beta/\hspace{0.5pt}\Gamma(\beta\hspace{-0.7pt}+\hspace{-0.7pt}1)$  and $J^\beta_0 g\in C[0,T]$, by Lemma \ref{thm:kseries} we know that $\sum_{j=0}^\infty\mathcal K^j J^\beta_0 g\in C_\infty(0,T]$, and that we can apply Fubini's Theorem to the above equalities. If $g$ satisfies the condition in Theorem \ref{thm:existence}, then Theorem \ref{thm:existence} proves  $I^\beta_0 g\in C_\infty(0,T]$ and  justifies the application of Fubini's Theorem.
\end{proof} 
 
\begin{remark}\label{rmk:sereqfk}$ $
\begin{enumerate}[(i)]
\item The above proof provides the following intuition for  how  $\DCs\beta$ extends the memory effect of $\DRL\beta$. Rewrite the right-hand side of \eqref{eq:thisthm} as 
\begin{equation}
\mathbb E\bigg[\int_0^{E_1(x)} g(x\hspace{-0.7pt}+\hspace{-0.7pt}S^1_s)\,{\rm d}s\bigg]+\mathbb E_x\bigg[\int_{\tau_1}^{\tau_\infty} g(S^c_s)\,{\rm d}s\bigg].
\label{eq:memoryint}
\end{equation}
Then the first term in \eqref{eq:memoryint} weights the \emph{past} values of $g$ on the interval $(x+S^1_{E_1(x)-},x]$, just like \eqref{eq:Jbetaintro} in the Caputo case (note that \eqref{eq:Jbetaintro} takes a slightly different form, just because there we assume $S^1$ starts from $x$ instead of $0$). Meanwhile, the second term proceeds on the interval $(0,x+S^1_{E_1(x)-}]$ according to the censored process. This second term can be simplified further using the the distribution of $S^c_{\tau_j}$ and written  in terms of products of \textit{i.i.d.} beta-distributed random variables, as we will see in Proposition \ref{thm:SR}.
\item  Proposition \ref{thm:FKform} proves that $\mathbb E_x\big[\int_0^{\tau_{\infty}} (S^c_s)^{\alpha}\,\dd s\big]$
 equals the right-hand side of \eqref{eq:seriesforalpha}. If $\alpha>-\beta$,  it  is finite and yields Theorem \ref{thm:main7}    (by letting $\alpha=0$). If $\alpha\le -\beta$, then  it is infinite by Remark \ref{rem:sumk_js}. In contrast, $\mathbb E\big[\int_0^{E_1(x)} (x\hspace{-0.7pt}+\hspace{-0.7pt}S^1_s)^\alpha\,{\rm d}s\big]<\infty$ for all $\alpha>-1$. 
\end{enumerate}
\end{remark}


\begin{definition}\label{def:stoc_rep} 
For any $x>0,$ we define the  $(0,x]$-valued discrete time Markov process  $X_j =x\prod_{i=1}^jB_i$, $j\in \mathbb N$, with $X_0=x$,  and $\{B_i\}_{i\in\mathbb N}$  being an \textit{i.i.d.}  collection of beta-distributed random variables on $(0,1)$ with parameters $(1\hspace{-0.7pt}-\hspace{-0.7pt}\beta,\beta)$.
\end{definition}

\begin{proposition}\label{thm:SR}
Under the assumption of Proposition \ref{thm:FKform},  $S^c_{\tau_{j}}$ equals $X_j$ in law for each $j\in\mathbb N$, and  $I^\beta_0$ allows the probabilistic series representation
\begin{equation}
 I^\beta_0 g(x) =\sum_{j=0}^\infty\mathbb E_x\big[J^\beta_0g(X_j)\big],\quad x\in(0,T]. 
\label{eq:solSR}
\end{equation} 
\end{proposition}
\begin{proof}  
We use induction to prove that $k_j(x,\,\cdot\,)$ is the density   of  $X_j\,\big|\,X_0\!=\!x$, for each $j\in\mathbb N$. The case when $j=1$ is clear. By the independence of $B_{j+1}$ and $X_j$, and the induction hypothesis
 \begin{equation*}
\mathbb P_x [X_{j+1}\le r]
=\mathbb P_x [X_jB_{j+1}\le r]
=\mathbb P_x\Big[B_{j+1}\le \frac r{X_j}\Big]
=\int_0^xk_j(x,s)\,\mathbb P \Big[B_{j+1}\le\frac rs\Big]{\rm d}s.
\end{equation*}
Then, recalling that $k_1(1,\,\cdot\,)$ is the density of $B_{j+1}$ and that it is supported on $(0,1)$,
\begin{align*}
\frac{\rm d}{{\rm d}r}\mathbb P_x [X_{j+1}\le r]
&=\int_0^xk_j(x,s)\,\frac{\rm d}{{\rm d}r}\mathbb P \Big[B_{j+1}\le\frac rs\Big]{\rm d}s\\
&=\int_r^xk_j(x,s)\,\frac1sk_1\Big(1,\frac rs\Big){\rm d}s\\
&=\int_r^xk_j(x,s)\,k_1(s,r){\rm d}s.
\end{align*}
Now apply Remark \ref{rmk:intto1}  and we know that $k_{j+1}(x,\,\cdot\,)$ is the density of $X_{j+1}\,\big|\, X_0\!=\!x$.  The induction step is now complete. By Lemma \ref{lem:1}-(i), we know that $S^c_{\tau_{j}}$ equals $X_j$ in law for each $j\in\mathbb N$, therefore the left-hand side of \eqref{sum of expectation of S^c_{tau_j} equals I^beta_0} equals the right-hand side of \eqref{eq:solSR}.
\end{proof}

\begin{remark}
Clearly Proposition \ref{thm:SR} can be strengthened into that $\{S^c_{\tau_{j}}\}_{j\in\mathbb N}$ equals $\{X_j\}_{j\in\mathbb N}$ in law. It is also clear that the series in \eqref{eq:solSR} equals $ \Gamma(1\hspace{-0.7pt}-\hspace{-0.7pt}\beta)\sum_{j=1}^{\infty}\mathbb E_x\big[(X_j)^\beta g(X_j)\big]$, since we know $J^\beta_0g (x)=\Gamma(1\hspace{-0.7pt}-\hspace{-0.7pt}\beta)\mathbb E_x\big[(X_1)^\beta g(X_1)\big]$ from Remark \ref{K = Jbeta x^-beta}.
\end{remark}

\subsection{Laplace transform of \texorpdfstring{$\tau_\infty$}{tauinf}}  We recall some definitions adapted to our setting that relate to Feller semigroups \cite{BSW13}. A collection of operators $P=\{P_s\}_{s\ge 0}$ is said to be a \emph{semigroup}  on a Banach space $X$ if $P_s:X\to X$ is   bounded and linear for any $s>0$,  $P_sP_t=P_{s+t}$ for all $t,s>0$, and $P_0$ is the identity operator. We say that $P$ is \emph{strongly continuous on} $\mathcal L\subseteq X$ if for   any $f\in \mathcal L$,   $P_sf\to f$ in $X$ as $s\to 0$, and that $P$ is \emph{strongly continuous} if $P$ is strongly continuous on $X$. We define the \textit{generator of} $P$ to be the pair $(\mathcal G,\mathcal D)$, where $\mathcal D:=\{f \in X:  \mathcal Gf \;\text{converges in}\; X\}$ with $\mathcal Gf:=\lim_{s\rightarrow 0}(P_sf\hspace{-0.7pt}-\hspace{-0.7pt}f)/s$, and we call $\mathcal D$ the \textit{domain of the generator of $P$}. Moreover, $P$ is said to be a \emph{positivity preserving contraction} on $X \subseteq B[0,T]$ if $0\le P_sf\le 1$ for any $s>0$ and $f\in X$ such that $0\le f\le 1$. Finally, a semigroup $P$ on $X=C_\infty(0,T]$ is said to be a \emph{Feller semigroup} if it is a strongly continuous positivity preserving contraction  on $X$. We recall \cite[Page 15]{BSW13} that
there exists a one-to-one correspondence between Feller semigroups and  Markov processes $\{Y_s\}_{s\ge0}$ such that  $f(\,\cdot\,)\mapsto P_sf(\,\cdot\,):=\mathbb E\big[f(Y_s)\,|\,Y_0\!=\cdot\,\big] $, $s\ge0$, is a Feller semigroup \cite[Page 15]{BSW13}.
 
\begin{proposition}\label{thm:gen}
For any $T>0$, the censored decreasing $\beta$-stable process $S^c$ induces a Feller semigroup on $C_\infty(0,T]$,  whose generator is
\[
\left(-\DCs\beta,  I^\beta_0 C_\infty(0,T]\right).
\] 
\end{proposition}
\begin{proof}
Let $P^c_tf(x)=\mathbb E_x\big[f(S^c_t)\big]$   for $t\ge0$ and $f\in B[0,T]$ (defining $S^c_t=\partial$ for all $t>0$ if $S^c_0=0$). Then $P^c=\{P^c_t\}_{t\ge0}$ is a positivity preserving contraction semigroup on $B[0,T]$, due to $S^c$ being a Markov process. We denote by  $\mathcal L^c$   the largest subset of $B[0,T]$    on which $ P^c$   is strongly continuous and by $\mathcal D^c$ the domain of the generator of $P^c$. First we prove   $\mathcal L^c\supseteq  C_\infty(0,T]$. For $f\in C_\infty(0,T]$, let $\tilde f(x):=f\big(\!\max\{x,0\}\big)$ for any $x\in(-\infty,T]$, and compute
\begin{align*}
\big|P^c_tf(x)-f(x)\big|&\le \Big|\mathbb E_x\big[\mathbf 1_{\{t<\tau_1\}}(f(S^c_t)\hspace{-0.7pt}-\hspace{-1.3pt}f(x))\big]\Big| +\mathbb E_x\big[\mathbf 1_{\{t\ge\tau_1\}}|f(S^c_t)\hspace{-0.7pt}-\hspace{-1.3pt}f(x)|\big]\\
&\le \Big|\mathbb E\big[\mathbf 1_{\{t<E_1(x)\}}(\tilde f(x\hspace{-0.7pt}+\hspace{-0.7pt}S^1_t)-\tilde f(x))\big]\Big| +2\|f\|_{C[0,x]}\mathbb P\big[t\ge E_1(x)\big]\\
&\le \Big|\mathbb E\big[\tilde f(x\hspace{-0.7pt}+\hspace{-0.7pt}S^1_t)-\tilde f(x)\big]\Big|+3\|f\|_{C[0,x]}\mathbb P\big[t\ge E_1(x)\big],
\end{align*}
where the first summand vanishes uniformly in $x$ as $t\to 0$  because $S^1$ is a Feller process on $\big\{g\in C(-\infty,T]: \lim_{x\to -\infty}g(x)=0\big\} $ \cite{BSW13}. Meanwhile for the second summand, for any $\varepsilon>0$ we can choose $\delta>0$ so that $\|f\|_{C[0,\hspace{0.6pt}x]}\le \varepsilon$ for all $x\in(0,\delta]$ and then we choose $\tilde t$ small so that
$$\mathbb P\big[t\ge E_1(x)\big]=\mathbb P[x+S^1_t \le 0]\le\mathbb P[\delta\le -S^1_t]\le \varepsilon,\quad \text{for all }x\ge\delta  \text{ and } t\le \tilde t,$$
so for all  $t\le\tilde t$
\begin{equation}
3\|f\|_{C[0,x]}\mathbb P\big[t\ge E_1(x)\big]\le\left\{ \begin{aligned}
& 3\varepsilon, &0\le x\le \delta,\\
& 3\varepsilon\|f\|_{C[0,T]}, &\delta <x\le T.
\end{aligned}\right.
\end{equation}
Therefore we have proved the strong continuity of $P^c$ on $C_\infty(0,T]$ and thus $C_\infty(0,T]\subseteq\mathcal L^c$. We now prove that  $C_\infty(0,T]$ is invariant under $P^c$.   The key ingredients are Theorem \ref{thm:main7} and Proposition \ref{thm:FKform}, which prove that  $ I^\beta_0 $ equals the potential $ \int_0^\infty P^c_s\dd{s}$ and is a bounded operator from $B[0,T]$ to $C_\infty(0,T]$. Then  \cite[Theorem 1.1']{dynkin} implies $\mathcal D^c = I^\beta_0 \mathcal L^c$, and we have
 \[ 
  I^\beta_0 C_\infty(0,T]  \subseteq   I^\beta_0 \mathcal L^c\subseteq    I^\beta_0 B[0,T] \subseteq C_\infty(0,T].
 \]
  Because Stone–Weierstrass Theorem and Example  \ref{ex:harm} prove  that $  I^\beta_0 C_\infty(0,T]$ is dense in $C_\infty(0,T]$,   by \cite[Property 1.3.B]{dynkin} and the above inclusions  we obtain  $\mathcal L^c= C_\infty(0,T]$. Because $P^c\mathcal L^c\subseteq \mathcal L^c$ \cite[Property 1.3.A]{dynkin}, we have proved that  $P^c$ is a Feller semigroup on $C_\infty(0,T]$. Since its  potential is $ I^\beta_0 $ and a bounded potential determines the generator \cite[Theorem 1.1']{dynkin}, Theorem \ref{thm:existence} implies that the  generator of $P^c$ is $\big(\hspace{-1.2pt}-\hspace{-0.7pt}\DCs\beta,  I^\beta_0 C_\infty(0,T]\big)$.
\end{proof}
\begin{remark}$ $
\begin{enumerate}[(i)]
\item As a corollary of Proposition \ref{thm:gen}, for any $b>0$ and $f\in I^\beta_0C_\infty(0,b]$, $u(t,x)=\mathbb E_x\big[f(S^c_t)\big]$ is the unique solution  satisfying the conditions \cite[a, b and c of Theorem 1.3]{dynkin} to the evolution equation (with $\DCs\beta$ acting on the spatial variable)
\begin{equation*}
\frac{{\rm d}}{{\rm d}t} u=-\DCs\beta u,\hspace{2pt}(t,x)\in \Rp \times (0,b];\hspace{7pt} u(t,0)=0,\hspace{2pt} t\in\Rp;\hspace{7pt}u(0,x)=f(x),\hspace{2pt} x\in (0,b].
\end{equation*}
\item  Note that Proposition \ref{thm:gen} relies  crucially on  Theorem \ref{thm:existence}. Also, it is not a special case of \cite[Theorem 3.3]{INW66EE}, primarily because the latter needs the assumption that $S^c$ is a Feller process, which is not clear from the INW construction.
\end{enumerate}
\end{remark}

We are now ready to prove a Mittag-Leffler-type representation for $\mathbb E_x\big[e^{\lambda \tau_\infty}\big]$,  whose analogue in the Caputo setting is the probabilistic identity 
\begin{equation}\label{eq:capmittag}
    \mathbb E_x\big[e^{\lambda  \tau_1}\big]=\sum_{j=0}^\infty \frac{\lambda^j x^{\beta j}}{\Gamma(j\beta+1)},
\end{equation}   first proved in \cite{B71}. Our proof follows the approach of \cite[Corollary  5.1]{HK16} to \eqref{eq:capmittag}. This approach allows one to solve exit problems by computing the Laplace transform of the lifetime of a killed Markov process  when one knows the analytical solution to the   resolvent equation   $-\mathcal G u = \lambda u+g$ ($\mathcal G$ being the generator of the process). In our case, the analytical solution is given by Proposition \ref{lem:inhomlin}.

\begin{theorem}\label{thm:exptau}
For every $\lambda<0,\,T>0$ and $g\in C[0,T]$,
\begin{equation}\label{eq:resleq}
\mathbb E_x\bigg[\int_0^{\tau_\infty} e^{\lambda s}g(S^c_s)\dd{s}\bigg]=\sum_{j=0}^\infty\lambda^{j}(I^\beta_0)^{j+1}g(x), \quad x\in(0,T].
\end{equation}
Moreover, the Mittag-Leffler-type series \eqref{solution to eigenvalue problem} ($u_0=1$) equals $ \mathbb E_x\big[e^{\lambda \tau_\infty}\big]$ for all $\lambda\in\mathbb R$ and $x>0$.
\end{theorem}

\begin{proof}
For the first claim, if $g\in C_\infty(0,T]$, recalling \cite[Theorem 1.1]{dynkin}, the equality  \eqref{eq:resleq} holds by Propositions \ref{thm:gen} and  \ref{lem:inhomlin}, as both sides of it are the unique  solution in $C_\beta[0,T]$  to the resolvent equation 
\begin{equation*}\label{eq:resleq2}
\DCs\beta u =\lambda u +g,\quad  u(0)=0,\quad g\in C_\infty(0,T],
\end{equation*}
where we used $I_0^\beta C_\infty(0,T]\subseteq  C_\beta[0,T]$ given by Lemma \ref{lem:sol_cts}.
Now, for any $g\in C[0,T]$, take $g_n\in C_\infty(0,T] $ so that $g_n\to g$ uniformly  on every compact subset of $(0,T]$ and $\sup\limits_n\|g_n\|_{C[0,T]}<\infty$. Fix $x\in(0,T]$, then for any $s>0$,  
\[
\mathbb E_x\big[g_n(S^c_s)\big]\to \mathbb E_x\big[g(S^c_s) \big],\;\text{as}\;n\to\infty,
\]
by Dominated Convergence Theorem. Then another application of Dominated Convergence Theorem, with the dominating function $\sup\limits_n\|g_n\|_{C[0,T]}e^{\lambda s}$, yields  
\[
\int_0^\infty  e^{\lambda s}\mathbb E_x\big[g_n(S^c_s)\big] \,{\rm d}s \to\int_0^\infty e^{\lambda s}  \mathbb E_x\big[g(S^c_s) \big] \,{\rm d}s,\;\text{as}\;n\to\infty.
\]
On the other hand, by the continuous dependence in Proposition \ref{lem:inhomlin}  (let the $\alpha$ there be e.g. $\beta/2$), 
\[
\sum_{j=0}^\infty\lambda^{j}(I^\beta_0)^{j+1}g_n(x)\to \sum_{j=0}^\infty\lambda^{j}(I^\beta_0)^{j+1}g(x),\;\text{as}\;n\to\infty.
\]
Therefore we have proved  \eqref{eq:resleq} for all $g\in C[0,T]$.

To prove the second claim for $\lambda<0$, in \eqref{eq:resleq} let $g=\lambda$, so that on the left-hand side
\begin{align*}
 \mathbb E_x\bigg[\int_0^{\tau_\infty }e^{\lambda s}\lambda\dd{s}\bigg]=\lambda \frac{\mathbb E_x[ e^{\lambda \tau_\infty }]-1}{\lambda} =
\mathbb E_x\big[ e^{\lambda \tau_\infty }\big]-1,
\end{align*}
and on the right-hand side
\begin{align*}
\sum_{j=0}^\infty \lambda^{j+1}(I^\beta_0)^{j+1}1(x) =\sum_{j=0}^\infty \lambda^{j}(I^\beta_0)^j1(x)-1.
\end{align*}
Hence by \eqref{The formula for I0beta^n} (with $\alpha=\beta$) we have proved the second claim for $\lambda\le0$ (with $\lambda=0$ being a trivial case), which combined with Lemma \ref{bound of the product} allows us to compute the moments $\mathbb E_x\big[(\tau_\infty)^j\big]\;(j\in \mathbb N)$ by differentiating $\mathbb E_x\big[ e^{\lambda \tau_\infty }\big]$ in $\lambda$ ($\lambda\le0$) for $j$ times. Those moments are displayed in \eqref{eq:momtauinf}, and in turn they allow us to prove the second claim also for $\lambda>0$, since we have \begin{equation*}
\mathbb E_x\big[ e^{\lambda \tau_\infty }\big]=\sum_{j=0}^\infty \frac{\lambda^j}{j!}\mathbb E_x\big[(\tau_\infty)^j\big],\quad\lambda,x>0,
\end{equation*}
where the series in the right-hand side converges to \eqref{solution to eigenvalue problem} by Lemma \ref{bound of the product}.
\end{proof}

\begin{corollary}\label{Cu}
 For any $\lambda<0$, the Mittag-Leffler-type series \eqref{solution to eigenvalue problem} is completely monotone.   More generally, for any Bernstein function $f$  the  series \eqref{solution to eigenvalue problem} composed with $f^{1/\beta}$ is completely monotone.  
\end{corollary}
\begin{proof} 
 Denote by $\mu_1$ the law of $\tau_\infty$ for $S^c_0=1$ and by $M_\lambda(x)$  the series \eqref{solution to eigenvalue problem} ($u_0=1$). Then 
\[  
 M_\lambda(x)=M_{\lambda x^\beta}(1)=\mathbb E_1\big[e^{(\lambda x^{\beta}) \tau_\infty}\big]=\int_{[0,\infty)}e^{\lambda x^\beta y}\mu_1({\rm d}y),
\]
where the second equality is due to Theorem \ref{thm:exptau}. The second claim now follows from \cite[Theorem 3.7]{SV12} because $M_\lambda$  composed with $f^{1/\beta}$ equals 
\[  
x\mapsto \int_{[0,\infty)}e^{\lambda f(x) y}\mu_1({\rm d}y), 
\] the composition of $x\mapsto\hspace{-1pt}\int_{[0,\infty)}e^{\lambda x y}\mu_1({\rm d}y)$ (which is completely monotone \cite[Theorem 1.4]{SV12}) with the Bernstein function $f$. The first claim corresponds to the Bernstein function $f(x)=x^\beta$.
\end{proof}

\begin{remark}\label{rmk:exptau}  The proof of \eqref{eq:exptauinf_intro} in Theorem \ref{thm:exptau} suits well our IVP theory, but is rather indirect, especially when compared to the standard proofs of \eqref{eq:capmittag}. Below we discuss the issues encountered with adapting those standard proofs to our censored setting, suggesting that our strategy is quite efficient.
\begin{enumerate}[(i)]
   
    \item  A simple proof of \eqref{eq:capmittag} follows the known \textit{direct} evaluation of the moments $\mathbb E_x\big[(\tau_1)^n\big]$ $= x^{\beta n}n! /\Gamma(\beta n\hspace{-0.7pt}+\hspace{-0.7pt}1)$, as then one can write  $\mathbb E_x\big[e^{\lambda\tau_1}\big]=\sum_{n=0}^\infty\lambda^n\mathbb E_x\big[(\tau_1)^n\big]/n!$ for $\lambda>0$ and check the convergence \cite{BKS96}. In Theorem \ref{thm:exptau} after proving that  \eqref{eq:exptauinf_intro} holds for all $\lambda\le0$, we obtain the closed form for the moments  $\mathbb E_x\big[(\tau_\infty)^n\big]= x^{\beta n}n!C_n$ where
    \begin{equation}\label{eq:momtauinf}
   C_n=\prod_{j=1}^n\left(\frac{\Gamma(1+j\beta)}{\Gamma(j\beta\hspace{-0.7pt}+\hspace{-0.7pt}1\hspace{-0.7pt}-\hspace{-0.7pt}\beta)}-\frac1{\Gamma(1\hspace{-0.7pt}-\hspace{-0.7pt}\beta)}\right)^{-1}, \;n\in\mathbb N,
    \end{equation}
    and those moments in turn allow us to prove the case when  $\lambda>0$. But obtaining  those moments by a direct evaluation is not easy.  Indeed, although one can directly compute 
    $ \mathbb E_x\big[(\tau_\infty)^n\big]=\lim\limits_{j\to\infty} \mathbb E_x\big[(\tau_j)^n\big]=x^{n\beta}n! \lim\limits_{j\to\infty} C_{n,j}$ (see Appendix \ref{app:C_j_n} for details), where
    $$ 
   C_{n,j}=\sum_{n_1+...+n_j=n} \frac{ \Gamma(1\hspace{-0.7pt}-\hspace{-0.7pt}\beta)^{1-j} }{ \Gamma( \beta n_j+1)  } \prod_{i=1}^{j-1}\frac{\Gamma\big(1\hspace{-0.7pt}-\hspace{-0.7pt}\beta\hspace{-0.7pt}+\hspace{-0.7pt}\beta\sum_{l=0}^{i-1}n_{j-l}\big)}{\Gamma\big(1+\beta\sum_{l=0}^{i}n_{j-l}\big)}  , \;n\in\mathbb N,$$
a direct proof of   $C_{n,j}\to C_n$ as $j\to\infty$ appears hard (nonetheless, Theorem \ref{thm:exptau} can serve as an indirect proof).
 \item There exist  several proofs of \eqref{eq:capmittag}  whose key steps rely   on the Laplace transform of $-S^1_1$ or infinitely divisible random variables (see \cite[Proposition 1.a]{B71}, \cite[XIII.8 Example (b), p. 453]{FII}, \cite[Theorem 2.10.2]{Z86},  \cite[Lemma 3.4]{JK16} and  \cite[6.6 (ii)]{K06}), and therefore these proofs do not apply to  $S^c$. 
  Also, the recent approach in \cite[Section 4.1]{BSV19}, which characterises  complete monotonicity of Kilbas–Saigo functions, poses many challenges to being adapted to our censored setting (in particular obtaining an appropriate analogue of \cite[Lemma 4.1]{BSV19}).
     

\item  We could not apply the strategy  used by the proof of Proposition \ref{thm:FKform} to prove \eqref{eq:exptauinf_intro}, as one would need a closed form of the complicated expectations $\mathbb E_x\big[\big( E_\beta(\lambda (S^c_{\tau_j})^\beta)-1\big)e^{\lambda \tau_j}\big]$ for $j\in\mathbb N$, where $E_\beta(x)=\sum_{n=0}^\infty y^n/\Gamma(\beta n\hspace{-0.7pt}+\hspace{-0.7pt}1)$  
(not to be confused with the inverse stable subordinator $E_j(y)$ defined in \eqref{eq:id}).

\item The simple observation that the relaxation equation \eqref{linear equation} rewrites as the Dirichlet problem $(-\DRL\beta\hspace{-0.7pt}+\hspace{-0.7pt}q)u=0$, $u(0)=1$  with the unbounded potential $q=\lambda+x^{-\beta}/\Gamma(1\hspace{-0.7pt}-\hspace{-0.7pt}\beta)$, suggests combining our IVP theory with potential theoretic techniques to prove \eqref{eq:exptauinf_intro}. Namely, show that the gauge function $u(x)=\mathbb E\big[\exp\big\{\int_0^{E_1(x)}q(x\hspace{-0.7pt}+\hspace{-0.7pt}S^1_t)\,\dd t\big\}\big]$ solves  \eqref{linear equation}. Although this approach appears feasible (as Proposition \ref{thm:FKform} should prove gaugeablility by \cite[Theorem 2.9.ii]{Bogdan}), we expect it to be more involved than our strategy. The main  reason is that, to show that the gauge function equals $\mathbb E_x\big[\exp\{ \lambda\tau_\infty\}\big]$, one would need to derive a relationship between the INW construction of $S^c$ and the coefficient $x^{-\beta}/\Gamma(1\hspace{-0.7pt}-\hspace{-0.7pt}\beta)$. Such relationship seems difficult to derive without the knowledge of the generator of $S^c$ (cf. \cite[Theorem 2.1]{BBC03}),  on the other hand this knowledge is already enough for our strategy to work. Another   reason is that one would have to employ   general results from the potential theory of Feynman–Kac semigroups (cf. \cite[Chapter 3]{CZ12}), only making the proof more technical. 


\end{enumerate}

\end{remark}

\begin{remark}\label{rmk:Btau}$ $
\begin{enumerate}[(i)]
\item Let $\tau_1(t)$, $ \tau_\infty(t)$ and $B$ denote $E_1(t)$ (i.e. the inverse stable subordinator), $\tau_\infty\,|\,S^c_0\!=\!t$ and an independent Brownian motion, respectively.  It is known (e.g. \cite{MS08}) that the Caputo time-fractional   diffusion equation $\DRL\beta \big[u\hspace{-0.7pt}-\hspace{-0.7pt}u(0)\big]=\Delta u/2$ is solved by the fractional kinetic process $\{B_{\tau_1(t)}\}_{t\ge0}$. This process is   well-known as sub-diffusion since \eqref{eq:id} implies $\mathbb E \big[|B_{\tau_1(t)}|^2\big]=\mathbb E\big[\tau_1(t)\big]=t^\beta /\Gamma(\beta\hspace{-0.7pt}+\hspace{-0.7pt}1)$, which is slower than normal diffusion $\mathbb E\big[|B_t|^2\big]=t$.    Our work suggests that  the censored  counterpart  $\DCs\beta  u=\Delta u/2$ is solved by a new sub-diffusion process $\{B_{\tau_\infty(t)}\}_{t\ge0}$. Indeed, Theorem \ref{thm:main7} shows that $\mathbb E\big[|B_{\tau_\infty(t)}|^2\big]=c t^\beta\;(c>0)$, and we expect the time-fractional evolution equation $\DCs\beta u=\mathcal G u+g$, $u(0)=\phi$ to have a unique (generalised) solution
\[u(t,x)=\mathbb E\bigg[\phi(X_{\tau_\infty})+\!\!\int_0^{\tau_\infty} \!\!\!\!g(S^c_s,X_s)\,{\rm d}s\;\bigg|\; (S^c_0,X_0)\!=\!(t,x)\bigg],\quad(t,x)\in(0,T]\times\mathbb R^d,\]
where $\phi\in\text{Dom}(\mathcal G),\;g\in C\big([0,T]\times\mathbb R^d\big)$, and $\big(\mathcal G,\text{Dom}(\mathcal G)\big)$ is the generator of any Feller process $X$ on $\mathbb R^d$ independent of $S^c$. (We think the last claim can be proved using the techniques from \cite{DTZ19,HKT18}, in the light of Proposition \ref{thm:gen}.) Let us also mention that to find strong solutions to $\DCs\beta u=\Delta u/2$, Theorem \ref{thm:exptau} opens up the possibility of applying the spectral decomposition method of \cite{CMN12}.
\item Although both $B_{\tau_1}$ and $B_{\tau_\infty}$ spread like $t^\beta$,   their respective Fourier modes model entirely different relaxation regimes. Namely, for any $\lambda\in\mathbb R^d$ we have,
\begin{align*}
  \mathbb E \big[\exp\big\{i\lambda\cdot B_{\tau_1(t)}\big\}  \big]&= \mathbb E \big[\exp\big\{\!-\!|\lambda|^2 \tau_1(t)/2\big\}\big]\asymp t^{-\beta}, \\
  \mathbb E \big[\exp\big\{i\lambda\cdot B_{\tau_\infty(t)}\big\}\big]&=\mathbb E \big[\exp\big\{\!-\!|\lambda|^2 \tau_\infty(t)/2\big\}\big]\asymp t^{-1-\beta},
  \end{align*}
  by \eqref{eq:capmittag}, Theorem \ref{thm:exptau} and Remark \ref{rmk:caprelax}-(i). Here $f\asymp g$ means $C^{-1}g\le f\le Cg$ for some constant $C>1$.
  
 \item There are several interesting questions revolving around $B_{\tau_\infty}$, a new example of anomalous diffusion. For instance, it is natural to ask if there  is a continuous-time-random-walk-type framework which scales to $B_{\tau_\infty}$, as is the case for $B_{\tau_1}$ \cite{BC11,MS08} and several other anomalous diffusion processes \cite{BMS04,WSJMS12} related to Caputo derivatives. Moreover, it is challenging and interesting to study the difference  in path regularity  between $B_{\tau_\infty}$ and $B_{\tau_1}$, in particular because the latter can be ``trapped" \cite{MS08}.

\item We mention that sub-diffusion and fractional relaxation equations are widely used to model anomalous (non-Debye) relaxation in dielectrics, see \cite{J02,WJMWT10,KCCT04,WSJMS12} and references therein. Their role is to provide a probabilistic theoretic explanation of the empirical (Havriliak–Negami) formula $\chi(\omega)=\big(1+(i\omega)^\alpha\big)^{-\gamma}$. (Here $\omega$ is the electric field's frequency and $\chi$ is the electric susceptibility. This formula fits well a majority of experimental data.) A typical example (Cole–Cole) is $\alpha=\beta\in(0,1)$ and $\gamma=1$, which is modelled by the sub-diffusion $B_{\tau_1}$ \cite[Page 3]{WJMWT10}. On the other hand, we expect $B_{\tau_\infty}$ to model a new regime with $\alpha=1+\beta\in (1,2)$ and $\gamma = \beta/(1+\beta)$ (by \cite[Eq. (1) and (5)]{WJMWT10}), although in the literature (e.g. \cite{KCCT04,WJMWT10}) we have not seen the parameter range $\alpha>1$.

 \end{enumerate}
\end{remark} 
 
\section*{Acknowledgement}
The authors would like to thank Professors Kai Diethelm,  Thomas Simon,   Zhen-Qing Chen, Krzysztof Bogdan and Zhi Zhou for many helpful discussions.

\appendix
\section{Proofs of elementary properties of \texorpdfstring{$J_0^\beta$}{J0beta} and \texorpdfstring{$\DRL\beta$}{DRL beta}}
\subsection{Proof of Lemma \ref{lem:prelim} } The proof consists of four parts.
\label{app:proofofelemprop}

\begin{enumerate}[(i)]
	\item Since both $u$ and $x^{\beta-1}$ are in $C\cap L^1(0,T]$,   for any $x\in(0,T],\,J^\beta_0 u(x)$ is well-defined and finite.
For $\varepsilon\in(0,\,x)$, define
\begin{equation*}
J^\beta_\varepsilon u(x)=\int_0^{x-\varepsilon}\frac{(x\hspace{-0.7pt}-\hspace{-0.7pt}r)^{\beta-1}}{\Gamma(\beta)} u(r)\,{\rm d}r.
\end{equation*}
Given $T_1\in(0,T],$ for all $x\in[T_1,\,T]$ and $\varepsilon\in(0,\,T_1)$, we have
\begin{align*}
\Big\vert J^\beta_\varepsilon u(x)-J^\beta_0 u(x)\Big\vert \le\int_{x-\varepsilon}^x\bigg\vert\frac{(x\hspace{-0.7pt}-\hspace{-0.7pt}r)^{\beta-1}}{\Gamma(\beta)}u(r)\bigg\vert\,{\rm d}r \le\frac{\varepsilon^\beta}{\beta\Gamma(\beta)}\Vert u\Vert_{C[T_1-\varepsilon,\,T]},
\end{align*}
therefore, as $\varepsilon\rightarrow0,\;J^\beta_\varepsilon u\to J^\beta_0 u$ uniformly on $[T_1,\,T]$. By Dominated Convergence Theorem, $J^\beta_\varepsilon u$ is continuous on $[T_1,\,T]$. So $J^\beta_0 u$ is also continuous on $[T_1,\,T]$, and thus on $(0,T]$. Integrability of $J^\beta_0 u$ follows by
\begin{equation*}
\begin{aligned}
\int_0^T\big\vert J^\beta_0 u(x)\big\vert\,{\rm d}x
&\le\int_0^T\int_0^x\frac{(x\hspace{-0.7pt}-\hspace{-0.7pt}r)^{\beta-1}}{\Gamma(\beta)}\big\vert u(r)\big\vert\,{\rm d}r\,{\rm d}x\\
&=\int_0^T\frac{\big\vert u(r)\big\vert}{\Gamma(\beta)}\int_r^T(x\hspace{-0.7pt}-\hspace{-0.7pt}r)^{\beta-1}\,{\rm d}x\,{\rm d}r\\
&\le\frac{T^\beta}{\Gamma(1\hspace{-0.7pt}+\hspace{-0.7pt}\beta)}\int_0^T\big\vert u(r)\big\vert\,{\rm d}r<\infty.
\end{aligned}
\end{equation*}
\item By Lemma \ref{lem:prelim}-(i), $J^\beta_0u\in C\cap L^1(0,T]$, and $J^{1-\beta}_0J^\beta_0u$ is well-defined on $ (0,T]$.  For $x\in(0,T]$,
\begin{equation}
\begin{aligned}
\label{J1-beta*Jbeta}
J^{1-\beta}_0J^\beta_0u(x)
&=\int_0^x\frac{(x\hspace{-0.7pt}-\hspace{-0.7pt}r)^{-\beta}}{\Gamma(1-\beta)}\int_0^r\frac{(r\hspace{-0.7pt}-\hspace{-0.7pt}s)^{\beta-1}}{\Gamma(\beta)} u(s)\,{\rm d}s\,{\rm d}r\\
&=\int_0^xu(s)\int_s^x\frac{(x\hspace{-0.7pt}-\hspace{-0.7pt}r)^{-\beta}}{\Gamma(1-\beta)}\frac{(r\hspace{-0.7pt}-\hspace{-0.7pt}s)^{\beta-1}}{\Gamma(\beta)}\,{\rm d}r\,{\rm d}s\\
&=\int_0^xu(s)\,{\rm d}s,
\end{aligned}
\end{equation}
where the second identity is due to Fubini's Theorem. Therefore $J^{1-\beta}_0J^\beta_0u\in C^1(0,T]$ and $\DRL\beta J^\beta_0u=u$.

\item  The ``only if\," is due to Lemma \ref{lem:prelim}-(ii) and \eqref{J1-beta*Jbeta} as well as the assumption that $g\in L^1(0,T]$.
For the ``if\,", we use  Lemma \ref{lem:prelim}-(ii) to get $\DRL\beta u=g=\DRL\beta J^\beta_0g$,  so $\DRL\beta [u\hspace{-0.7pt}-\hspace{-0.7pt}J^\beta_0g]=0$. By the definition of $\DRL\beta$, we know   that $J^{1-\beta}_0[u\hspace{-0.7pt}-\hspace{-0.7pt}J^\beta_0g]$ is constant. By \eqref{J1-beta*Jbeta} we know   that $\lim_{x\rightarrow0}J^{1-\beta}_0J^\beta_0g(x)=0$, and  by assumption $\lim_{x\rightarrow0}J^{1-\beta}_0u(x)=0$. 
Therefore $J^{1-\beta}_0[u\hspace{-0.7pt}-\hspace{-0.7pt}J^\beta_0g]$ must be 0. Conclude with Lemma \ref{lem:prelim}-(ii) which proves $u-J^\beta_0g=\DRL{1-\beta}J^{1-\beta}_0[u\hspace{-0.7pt}-\hspace{-0.7pt}J^\beta_0g]=\DRL{1-\beta}0=0$. 

\item  For $u\in C[0,T]$, we  have  $| J^{1-\beta}_0u(x)|\le\|u\|_{C[0,T]}x^{1-\beta}/\Gamma(2\hspace{-0.7pt}-\hspace{-0.7pt}\beta)\to 0$ as $x\to 0$. Then the   ``if\,"   of Lemma \ref{lem:prelim}-(iii) applies, giving $u=J^\beta_00=0$.\hfill$\square$
\end{enumerate}

\subsection{Proof of H\"older regularity of \texorpdfstring{$J_0^\beta g$}{J0beta g}}
\label{sec: Holder regularity of J0beta g}

For any interval $\Omega\subseteq\mathbb R$, we denote by $L^\infty(\Omega)$  the essentially bounded functions in $L^1(\Omega)$.

\begin{lemma}
\label{lem: Holder regularity of J0beta g}
If $g\in L^\infty(0,T]$, then $J_0^\beta g\in C^{0,\beta}[0,T]$ with $J_0^\beta g(0)$ being 0 and a H\"older constant being $2\Vert g\Vert_{L^\infty(0,T]}/\Gamma(1\hspace{-0.7pt}+\hspace{-0.7pt}\beta)$.
\end{lemma}

\begin{proof}
For $0<x<x\hspace{-0.7pt}+\hspace{-0.7pt}h\le T$,  the difference $J_0^\beta g(x\hspace{-0.7pt}+\hspace{-0.7pt}h)-J_0^\beta g(x)$   equals
\begin{equation*}
\begin{aligned}
&\int_0^{x+h}\frac{(x\hspace{-0.7pt}+\hspace{-0.7pt}h\hspace{-0.7pt}-\hspace{-0.7pt}r)^{\beta-1}}{\Gamma(\beta)}g(r)\,{\rm d}r-\int_0^x\frac{(x\hspace{-0.7pt}-\hspace{-0.7pt}r)^{\beta-1}}{\Gamma(\beta)}g(r)\,{\rm d}r\\
=&\int_0^x\frac{(x\hspace{-0.7pt}+\hspace{-0.7pt}h\hspace{-0.7pt}-\hspace{-0.7pt}r)^{\beta-1}-(x\hspace{-0.7pt}-\hspace{-0.7pt}r)^{\beta-1}}{\Gamma(\beta)}g(r)\,{\rm d}r+\int_x^{x+h}\frac{(x\hspace{-0.7pt}+\hspace{-0.7pt}h\hspace{-0.7pt}-\hspace{-0.7pt}r)^{\beta-1}}{\Gamma(\beta)}g(r)\,{\rm d}r.
\end{aligned}
\end{equation*}
The absolute value of the first summand can be bounded from above by
\begin{equation*}
\frac{\Vert g\Vert_{L^\infty(0,T]}}{\beta\Gamma(\beta)}\big\vert(x\hspace{-0.7pt}+\hspace{-0.7pt}h)^\beta-h^\beta-x^\beta\big\vert\le\frac{h^\beta\Vert g\Vert_{L^\infty(0,T]}}{\Gamma(1+\beta)},
\end{equation*}
where the inequality is obtained by observing that $(x\hspace{-0.7pt}+\hspace{-0.7pt}h)^\beta-h^\beta-x^\beta$ is decreasing from 0 to $-h^\beta$ (unattainable) with respect to $x\in[0,\infty)$. The absolute value of the second summand can be bounded from above by $h^\beta\Vert g\Vert_{L^\infty(0,T]}/\Gamma(1\hspace{-0.7pt}+\hspace{-0.7pt}\beta)$. Therefore $$\big\vert J_0^\beta g(x\hspace{-0.7pt}+\hspace{-0.7pt}h)-J_0^\beta g(x)\big\vert\le 2h^\beta\Vert g\Vert_{L^\infty(0,T]}/\Gamma(1\hspace{-0.7pt}+\hspace{-0.7pt}\beta).$$
Note that for all $x\in(0,T]$, 
$$\big\vert J_0^\beta g(x)-J_0^\beta g(0)\big\vert\le x^\beta\Vert g\Vert_{L^\infty(0,T]}/\Gamma(1\hspace{-0.7pt}+\hspace{-0.7pt}\beta),$$
we know $J_0^\beta g\in C^{0,\beta}[0,T]$ with $2\Vert g\Vert_{L^\infty(0,T]}/\Gamma(1\hspace{-0.7pt}+\hspace{-0.7pt}\beta)$ as a H\"older constant.
\end{proof}

\begin{lemma}
\label{beta Holder regularity of J0beta singular g}
If $g\in L^1 (0,T] $ satisfies $\big\vert g(x)\big\vert\le Mx^{\alpha-\beta}$   for some $\alpha,\,M\ge0$ and all $x\in(0,T]$, then for all $T_1\in(0,T),\;J_0^\beta g\in C^{0,\beta}[T_1,T]$ with a H\"older constant being $2M\max\{T_1^{\alpha-\beta}\!,\,T^{\alpha-\beta}\}/\Gamma(1\hspace{-0.7pt}+\hspace{-0.7pt}\beta)$.
\end{lemma}

\begin{proof}
For $0<T_1\le x<x\hspace{-0.7pt}+\hspace{-0.7pt}h\le T$, we have
\begin{equation*}
J_0^\beta g(x\!+\!h)-J_0^\beta g(x)=\int_0^x\frac{(x\!+\!h\!-\!r)^{\beta-1}-(x\!-\!r)^{\beta-1}}{\Gamma(\beta)}g(r)\,{\rm d}r+\int_x^{x+h}\frac{(x\!+\!h\!-\!r)^{\beta-1}}{\Gamma(\beta)}g(r)\,{\rm d}r.
\end{equation*}
The absolute value of the first summand can be bounded from above by
\begin{equation*}
\begin{aligned}
&\, \int_0^x\frac{(x\hspace{-0.7pt}-\hspace{-0.7pt}r)^{\beta-1}-(x\hspace{-0.7pt}+\hspace{-0.7pt}h\hspace{-0.7pt}-\hspace{-0.7pt}r)^{\beta-1}}{\Gamma(\beta)}Mr^{\alpha-\beta}\,{\rm d}r\\
=&\, \frac M{\Gamma(\beta)}\bigg(x^\alpha\hspace{0.6pt}{\rm B}(1;\,\alpha\!+\!1\!-\!\beta,\,\beta)-(x\hspace{-0.7pt}+\hspace{-0.7pt}h)^\alpha\hspace{0.6pt}{\rm B}\Big(\frac x{x\hspace{-0.7pt}+\hspace{-0.7pt}h};\,\alpha\!+\!1\!-\!\beta,\,\beta\Big)\bigg)\\
\le&\, \frac{Mx^\alpha}{\Gamma(\beta)}\bigg({\rm B}(1;\,\alpha\!+\!1\!-\!\beta,\,\beta)-{\rm B}\Big(\frac x{x\hspace{-0.7pt}+\hspace{-0.7pt}h};\,\alpha\!+\!1\!-\!\beta,\,\beta\Big)\bigg)\\
=&\, \frac{Mx^\alpha}{\Gamma(\beta)}\int_{x/(x+h)}^1r^{\alpha-\beta}(1\hspace{-0.7pt}-\hspace{-0.7pt}r)^{\beta-1}\,{\rm d}r\\
\le&\, \frac{Mx^\alpha}{\Gamma(\beta)}\max\bigg\{1,\,\Big(\frac x{x\hspace{-0.7pt}+\hspace{-0.7pt}h}\Big)^{\alpha-\beta}\bigg\}\frac1\beta\Big(\frac h{x\hspace{-0.7pt}+\hspace{-0.7pt}h}\Big)^\beta\\
\le&\, \frac{Mx^{\alpha-\beta}}{\Gamma(1\hspace{-0.7pt}+\hspace{-0.7pt}\beta)}h^\beta
\le \frac{Mh^\beta\max\{T_1^{\alpha-\beta}\!,\,T^{\alpha-\beta}\}}{\Gamma(1\hspace{-0.7pt}+\hspace{-0.7pt}\beta)},
\end{aligned}
\end{equation*}
where B is the incomplete beta function. The absolute value of the second summand can be bounded from above by
\begin{equation*}
\begin{aligned}
\int_x^{x+h}\frac{(x\hspace{-0.7pt}+\hspace{-0.7pt}h\hspace{-0.7pt}-\hspace{-0.7pt}r)^{\beta-1}}{\Gamma(\beta)}Mr^{\alpha-\beta}\,{\rm d}r&=\, \frac M{\Gamma(\beta)}\int_0^h(h\hspace{-0.7pt}-\hspace{-0.7pt}r)^{\beta-1}(x\hspace{-0.7pt}+\hspace{-0.7pt}r)^{\alpha-\beta}\,{\rm d}r\\
&\le\, \frac M{\Gamma(\beta)}\max\{x^{\alpha-\beta},\,(x\hspace{-0.7pt}+\hspace{-0.7pt}h)^{\alpha-\beta}\}\frac{h^\beta}\beta\\
&\le\, \frac{Mh^\beta\max\{T_1^{\alpha-\beta}\!,\,T^{\alpha-\beta}\}}{\Gamma(1\hspace{-0.7pt}+\hspace{-0.7pt}\beta)}.
\end{aligned}
\end{equation*}
Therefore $\big\vert J_0^\beta g(x\hspace{-0.7pt}+\hspace{-0.7pt}h)-J_0^\beta g(x)\big\vert\le2Mh^\beta\max\{T_1^{\alpha-\beta}\!,\,T^{\alpha-\beta}\}/\Gamma(1\hspace{-0.7pt}+\hspace{-0.7pt}\beta)$, we know $J_0^\beta g\in C^{0,\beta}[T_1,T]$ with $2M\max\{T_1^{\alpha-\beta}\!,\,T^{\alpha-\beta}\}/\Gamma(1\hspace{-0.7pt}+\hspace{-0.7pt}\beta)$ as a H\"older constant.
\end{proof}

\begin{remark}
Lemma \ref{beta Holder regularity of J0beta singular g} is a natural generalization of Lemma \ref{lem: Holder regularity of J0beta g}, since $\Vert g\Vert_{L^\infty[T_1,T]}\le M\max\{T_1^{\alpha-\beta}\!,\,T^{\alpha-\beta}\}$.
For $\alpha<\beta$, the $\beta$-H\"older continuity of $J_0^\beta g$ is only local (away from 0), since $T_1^{\alpha-\beta}$ explodes as $T_1\rightarrow0$. Indeed, in such cases, $\big\vert J_0^\beta g(x)-J_0^\beta g(0)\big\vert\le Mx^\alpha\Gamma(\alpha\hspace{-0.7pt}+\hspace{-0.7pt}1\hspace{-0.7pt}-\hspace{-0.7pt}\beta)/\Gamma(1\hspace{-0.7pt}+\hspace{-0.7pt}\alpha)$ for $x\in(0,T]$, with $J_0^\beta g(0)$ defined to be 0. So the uniform H\"older continuity of $J_0^\beta g$ on $[0,T]$ is probably only of exponent $\alpha$. Sure enough, we have the following result.
\end{remark}

\begin{lemma}
For the same $g$ in Lemma \ref{beta Holder regularity of J0beta singular g}, additionally assume $\alpha\in(0,\beta)$, then $J_0^\beta g\in C^{0,\alpha}[0,T]$ with a H\"older constant being $2M\Gamma(\alpha\hspace{-0.7pt}+\hspace{-0.7pt}1\hspace{-0.7pt}-\hspace{-0.7pt}\beta)/\Gamma(1\hspace{-0.7pt}+\hspace{-0.7pt}\alpha)$.
\label{alpha Holder regularity of J0beta singular g}
\end{lemma}

\begin{proof}
By assumptions on $g$, we have
\begin{equation*}
\big\vert J_0^{\beta-\alpha}g(x)\big\vert\le J_0^{\beta-\alpha}\vert g\vert(x)\le MJ_0^{\beta-\alpha}[x^{\alpha-\beta}]=M\Gamma(\alpha\hspace{-0.7pt}+\hspace{-0.7pt}1\hspace{-0.7pt}-\hspace{-0.7pt}\beta),
\end{equation*}
then by Lemma \ref{lem: Holder regularity of J0beta g}, $J_0^\alpha J_0^{\beta-\alpha}g\in C^{0,\alpha}[0,T]$ with a H\"older constant being $2M\Gamma(\alpha\hspace{-0.7pt}+\hspace{-0.7pt}1\hspace{-0.7pt}-\hspace{-0.7pt}\beta)/\Gamma(1\hspace{-0.7pt}+\hspace{-0.7pt}\alpha)$, and consequently $J_0^\alpha J_0^{\beta-\alpha}g\in C[0,T]$. By Lemma \ref{beta Holder regularity of J0beta singular g}, $J_0^\beta g\in C(0,T]$. Recalling \cite[Theorem 2.2]{kai}, we know  $J_0^\alpha J_0^{\beta-\alpha}g=J_0^\beta g$ everywhere, with $J_0^\beta g(0)$ defined to be 0. Therefore, $J_0^\beta g$ is also in $C^{0,\alpha}[0,T]$ with the same H\"older constant as $J_0^\alpha J_0^{\beta-\alpha}g$.
\end{proof}
\subsection{Proof of the diagram in Remark \ref{strong solution space}}
\label{proof of relations between sets}

The proof can be decomposed into several statements.
\begin{enumerate}[(i)]
\item $J^\beta_0\big[C\cap L^1(0,T]\big]\subseteq C_\beta(0,T]$: follows Lemma \ref{lem:prelim}-(ii).
\item $J^\beta_0\big[C\cap L^1(0,T]\big]\nsupseteq C_\beta(0,T]$: this is because $x^{\frac\beta2-1}$ is in $C_\beta(0,T]$ but not in $J^\beta_0\big[C\cap L^1(0,T]\big]$. The latter fact can be proven by Lemma \ref{lem:prelim}-(ii), noticing $\DRL\beta x^{\frac\beta2-1}\notin L^1(0,T]$.
\item $J^\beta_0\big[C\cap L^1(0,T]\big]\nsubseteq C_\beta[0,T]$: for example, $J^\beta_0\big(x^{-\frac{1+\beta}2}\big)$ is not in $C[0,T]$.
\item $J^\beta_0\big[C\cap L^1(0,T]\big]\nsupseteq C_\beta[0,T]$: see Lemma \ref{C_beta[0,T] not subseteq Jbeta0(C cap L1(0,T])}.
\item $J^\beta_0\big[C\cap L^1(0,T]\big]\supseteq U$: follows Lemma \ref{lem:prelim}-(iii).
\item $J^\beta_0\big[C\cap L^1(0,T]\big]\cap C_\beta[0,T]\nsubseteq U$: to see this, in the proof of Lemma \ref{C_beta[0,T] not subseteq Jbeta0(C cap L1(0,T])},  we may take $\alpha\!-\!\beta\!-\!\gamma\beta\!\in(-1,-\beta]$ instead of $\alpha\!-\!\beta\!-\!\gamma\beta\!<\!-1$, then $\big(J^{1-\beta}_0u\big)'=O(x^{\alpha-\beta-\gamma\beta})+O(x^{\alpha-\beta})$ is absolutely integrable. We know such $u$ is in $C_\beta[0,T]$, and by Lemma \ref{lem:prelim}-(iii), $u\in J^\beta_0\big[C\cap L^1(0,T]\big]$. Yet $u \notin U$, because $\DCs\beta u=\big(J^{1-\beta}_0u\big)'-x^{-\beta}u/\Gamma(1\!-\!\beta)=O(x^{\alpha-\beta-\gamma\beta})+O(x^{\alpha-\beta})$, which cannot be controlled by $x^{ \tilde\alpha-\beta}$ for any $\tilde\alpha>0$.
\end{enumerate}

\begin{lemma}
\label{C_beta[0,T] not subseteq Jbeta0(C cap L1(0,T])}
We have $J^\beta_0\big[C\cap L^1(0,T]\big]\nsupseteq C_\beta[0,T]$.
\end{lemma}

\begin{proof}
We only need to find a $u\in C[0,T]$, such that $J^{1-\beta}_0u\in C^1(0,T]$ but $\big(J^{1-\beta}_0u\big)'\notin L^1(0,T]$. Once succeed, we know such $u\in C_\beta[0,T]$ by definition. If $u\in J^\beta_0\big[C\cap L^1(0,T]\big]$, then $\exists\;v\in C\cap L^1(0,T]$ such that $u=J^\beta_0v$. By Lemma \ref{lem:prelim}-(iii), we know $v=\DRL\beta u=\big(J^{1-\beta}_0u\big)'\notin L^1(0,T]$. This contradiction implies $u\notin J^\beta_0\big[C\cap L^1(0,T]\big]$ and thus $C_\beta[0,T]\nsubseteq J^\beta_0\big[C\cap L^1(0,T]\big]$.

In order for $J^{1-\beta}_0u$ to have continuous but  not absolutely integrable derivative on $(0,T]$, and thus have unbounded variation, $u$ had better oscillate quickly near 0. Let $u(x)=x^\alpha\sin(x^{-\gamma})$, where $\alpha,\gamma>0$. Then
\begin{equation*}
J^{1-\beta}_0u(x)=\int_0^x\frac{r^\alpha\sin(r^{-\gamma})\,{\rm d}r}{\Gamma(1-\beta)(x-r)^\beta}=\frac{x^{1-\beta+\alpha}}{\Gamma(1-\beta)}\int_0^1\frac{s^\alpha\sin\big((xs)^{-\gamma}\big)}{(1-s)^\beta}\,{\rm d}s=\frac{x^{1-\beta+\alpha}}{\Gamma(1-\beta)}\phi(x),
\end{equation*}
where $\phi$ is defined in Lemma \ref{phi and estimate for phi'}. So $J^{1-\beta}_0u\in C^1(0,T]$ and
\begin{equation*}
\big(J^{1-\beta}_0u\big)'(x)=\frac{1-\beta+\alpha}{\Gamma(1-\beta)x^{\beta-\alpha}}\phi(x)+\frac{x^{1-\beta+\alpha}}{\Gamma(1-\beta)}\phi'(x)
\end{equation*}
Note that $\Vert\phi\Vert_{L^\infty(0,T]}\le\Gamma(1\!+\!\alpha)\Gamma(1\!-\!\beta)/\Gamma(2\!-\!\beta\!+\!\alpha)$, we know the first summand  is absolutely integrable, so $\big(J^{1-\beta}_0u\big)'$ has the same absolute integrability with the second summand. By Lemma \ref{phi and estimate for phi'}, we know
\begin{equation*}
\frac{x^{1-\beta+\alpha}}{\Gamma(1-\beta)}\phi'(x)=\gamma^\beta x^{\alpha-\beta-\gamma\beta}\sin\Big(\frac1{x^\gamma}-\frac{\pi\beta}2\Big)+O\big(x^{\alpha-\beta}\big).
\end{equation*}
Choose $\gamma$ large enough so that $\alpha\!-\!\beta\!-\!\gamma\beta\!<\!-1$, then $\big(J^{1-\beta}_0u\big)'\notin L^1(0,T]$ and we are done.
Note although the above estimate is enough for us, it may be tightened into $\big(J^{1-\beta}_0u\big)'(x)=\gamma^\beta x^{\alpha-\beta-\gamma\beta}\big[\sin(x^{-\gamma}\!-\!\pi\beta/2) +O(x^\gamma)\big]$.

\end{proof}

\begin{lemma}
\label{phi and estimate for phi'}
For $\alpha,\gamma>0$, define
\begin{equation*}
\phi(x)=\int_0^1\frac{s^\alpha\sin\big((xs)^{-\gamma}\big)}{(1-s)^\beta}\,{\rm d}s,\quad x\in(0,T],
\end{equation*}
then $\phi\in C^1(0,T]$, and $\phi'(x)$ is given by an improper integral which converges absolutely or conditionally,
\begin{equation*}
\phi'(x)=\frac{-\gamma}{x^{1+\gamma}}\int_0^1\frac{\cos\big((xs)^{-\gamma}\big)}{s^{\gamma-\alpha}(1-s)^\beta}\,{\rm d}s=\int^\infty_1\frac{-x^{-1-\gamma}\cos(x^{-\gamma}r)\,{\rm d}r}{r^{(1+\alpha)/\gamma}(1-r^{-1/\gamma})^\beta},\quad x\in(0,T],
\end{equation*}
and we have the following estimate for $\phi'$,
\begin{equation*}
\phi'(x)=\gamma^\beta\frac{\Gamma(1\!-\!\beta)}{x^{1+\gamma\beta}}\sin\Big(\frac1{x^\gamma}-\frac{\pi\beta}2\Big)+O\Big(\frac1x\Big),
\end{equation*}
which means $\phi'$ explodes like $x^{-1-\gamma\beta}$ at 0.
\end{lemma}

\begin{proof}
For $0<\varepsilon<1/2$, define
\begin{equation*}
\phi_\varepsilon(x)=\int_\varepsilon^{1-\varepsilon}\frac{s^\alpha\sin\big((xs)^{-\gamma}\big)}{(1-s)^\beta}\,{\rm d}s,\quad x\in(0,T],
\end{equation*}
then $\phi_\varepsilon\in C^1(0,T]$, and
\begin{equation*}
\phi_\varepsilon'(x)=\frac{-\gamma}{x^{1+\gamma}}\int_\varepsilon^{1-\varepsilon}\frac{\cos\big((xs)^{-\gamma}\big)}{s^{\gamma-\alpha}(1-s)^\beta}\,{\rm d}s = \int^{\varepsilon^{-\gamma}}_{(1-\varepsilon)^{-\gamma}}\frac{-x^{-1-\gamma}\cos(x^{-\gamma}r)\,{\rm d}r}{r^{(1+\alpha)/\gamma}(1-r^{-1/\gamma})^\beta}.
\end{equation*}
As $\varepsilon\rightarrow0$, obviously $\phi_\varepsilon$ converges to $\phi$ uniformly on any closed subinterval of $(0,T]$. Our next step is to verify $\phi_\varepsilon'$ also converges uniformly on the closed subinterval. After that, we can obtain $\phi'$ by exchanging limit and differentiation. To this end, we will prove that $\phi_\varepsilon'$ satisfies the Cauchy criterion for uniform convergence.

For $0<\varepsilon<\delta<1/2$ and $x\in(0,T]$,
\begin{equation*}
\phi_{\varepsilon}'(x)-\phi_{\delta}'(x)=\frac{-1}{x^{1+\gamma}}\bigg(\int^{(1-\delta)^{-\gamma}}_{(1-\varepsilon)^{-\gamma}}+\int^{\varepsilon^{-\gamma}}_{\delta^{-\gamma}}\bigg)\frac{\cos(x^{-\gamma}r)\,{\rm d}r}{r^{(1+\alpha)/\gamma}(1-r^{-1/\gamma})^\beta}.
\end{equation*}
We can bound the above integral over the first subinterval as follows,
\begin{equation*}
\begin{aligned}
\bigg\vert\int^{(1-\delta)^{-\gamma}}_{(1-\varepsilon)^{-\gamma}}\frac{\cos(x^{-\gamma}r)\,{\rm d}r}{r^{(1+\alpha)/\gamma}(1-r^{-1/\gamma})^\beta}\bigg\vert&\le\int^{(1-\delta)^{-\gamma}}_{(1-\varepsilon)^{-\gamma}}r^{-(1+\alpha)/\gamma}(1-r^{-1/\gamma})^{-\beta}\,{\rm d}r\\
&=\int^{(1-\delta)^{-\gamma}}_{(1-\varepsilon)^{-\gamma}}\frac{r^{1-\alpha/\gamma}}{r^{1+1/\gamma}}(1-r^{-1/\gamma})^{-\beta}\,{\rm d}r\\
&\le\int^{(1-\delta)^{-\gamma}}_{(1-\varepsilon)^{-\gamma}}\frac{(1-\delta)^{-\gamma}}{r^{1+1/\gamma}}(1-r^{-1/\gamma})^{-\beta}\,{\rm d}r\\
&=\frac\gamma{1-\beta}\frac{\delta^{1-\beta}-\varepsilon^{1-\beta}}{(1-\delta)^\gamma}.
\end{aligned}
\end{equation*}

For the second subinterval, we have
\begin{equation*}
\begin{aligned}
\bigg\vert\int^{\varepsilon^{-\gamma}}_{\delta^{-\gamma}}\frac{\cos(x^{-\gamma}r)\,{\rm d}r}{r^{(1+\alpha)/\gamma}(1-r^{-1/\gamma})^\beta}\bigg\vert&=\bigg\vert\int^{\varepsilon^{-\gamma}}_{\delta^{-\gamma}}\frac{x^\gamma\,{\rm d}\sin(x^{-\gamma}r)}{r^{(1+\alpha)/\gamma}(1-r^{-1/\gamma})^\beta}\bigg\vert\\
&\le\Bigg\vert\frac{x^\gamma\sin(x^{-\gamma}r)}{r^{(1+\alpha)/\gamma}(1-r^{-1/\gamma})^\beta}\bigg\vert^{\varepsilon^{-\gamma}}_{\delta^{-\gamma}}\Bigg\vert\\
&\quad+\bigg\vert x^\gamma\int^{\varepsilon^{-\gamma}}_{\delta^{-\gamma}}\sin(x^{-\gamma}r)\,{\rm d}\big(r^{-(1+\alpha)/\gamma}(1-r^{-1/\gamma})^{-\beta}\big)\bigg\vert,
\end{aligned}
\end{equation*}
note that $r^{-(1+\alpha)/\gamma}(1-r^{-1/\gamma})^{-\beta}$ is decreasing on $(1,\infty)$, so the second summand can be bounded from above by $-x^\gamma\int^{\varepsilon^{-\gamma}}_{\delta^{-\gamma}}{\rm d}\big[r^{-(1+\alpha)/\gamma}(1-r^{-1/\gamma})^{-\beta}\big]=x^\gamma\big(\delta^{1+\alpha}(1-\delta)^{-\beta}-\varepsilon^{1+\alpha}(1-\varepsilon)^{-\beta}\big)$. The first summand can be bounded from above by $x^\gamma\big(\delta^{1+\alpha}(1-\delta)^{-\beta}+\varepsilon^{1+\alpha}(1-\varepsilon)^{-\beta}\big)$. To sum up, for $0<\varepsilon<\delta<1/2$ and $x\in(0,T]$, we have
\begin{equation*}
\big\vert\phi_{\varepsilon}'(x)-\phi_{\delta}'(x)\big\vert\le\frac1{x^{1+\gamma}}\bigg(\frac\gamma{1-\beta}\frac{\delta^{1-\beta}-\varepsilon^{1-\beta}}{(1-\delta)^\gamma}+\frac{2x^\gamma\delta^{1+\alpha}}{(1-\delta)^\beta}\bigg)\le\frac{2^\gamma\gamma}{1-\beta}\frac{\delta^{1-\beta}}{x^{1+\gamma}}+2^{1+\beta}\frac{\delta^{1+\alpha}}x,
\end{equation*}
and thus Cauchy criterion is satisfied on any closed subinterval of $(0,T]$.

So far we have proved $\phi\in C^1(0,T]$ and
\begin{equation}
\label{expression for phi'}
\phi'(x)=\int^\infty_1\frac{-x^{-1-\gamma}\cos(x^{-\gamma}r)\,{\rm d}r}{r^{(1+\alpha)/\gamma}(1-r^{-1/\gamma})^\beta},\quad x\in(0,T].
\end{equation}

The only job left is to estimate $\phi'$ around $x=0$. The essential observation is that the denominator in \eqref{expression for phi'} behaves like $\gamma^{-\beta}(r-1)^\beta$ as $r\rightarrow1$, and that this singularity at $r=1$ has the dominating effect on the whole integral, so that we need not worry such behavior is only asymptotic at $r=1$. Our strategy is to replace the denominator with $\gamma^{-\beta}(r-1)^\beta$, which will make our life much easier, and then we only need to show the above replacement will only cause higher order effect. To this end, we need the boundedness of the total variation of $\eta$ in Lemma \ref{eta of bounded variation}. Let us rewrite \eqref{expression for phi'} in terms of $\eta$ and $(r-1)^{-\beta}$,
\begin{equation*}
\phi'(x)=\frac{-1}{x^{1+\gamma}}\int^\infty_1\bigg(\eta(r)+\frac{\gamma^\beta}{(r-1)^\beta}\bigg)\cos(x^{-\gamma}r)\,{\rm d}r.
\end{equation*}

The effect of $\eta$ can be bounded as follows,
\begin{equation*}
\begin{aligned}
\bigg\vert\int^\infty_1\eta(r)\cos(x^{-\gamma}r)\,{\rm d}r\bigg\vert&=x^\gamma\bigg\vert\int^\infty_1\eta(r)\,{\rm d}\sin(x^{-\gamma}r)\bigg\vert\\
&=x^\gamma\bigg\vert\eta(r)\sin(x^{-\gamma}r)\Big\vert^\infty_1-\int^\infty_1\sin(x^{-\gamma}r)\,{\rm d}\eta(r)\bigg\vert\\
&=x^\gamma\bigg\vert\int^\infty_1\sin(x^{-\gamma}r)\,{\rm d}\eta(r)\bigg\vert\\
&\le x^\gamma V_1^\infty(\eta),
\end{aligned}
\end{equation*}
where the third identity is due to $\eta(1)=\eta(\infty)=0$ and $\big\vert\sin(x^{-\gamma}r)\big\vert\le1$.

For $(r-1)^{-\beta}$, we can evaluate the integral exactly,
\begin{equation*}
\begin{aligned}
\int^\infty_1\frac{\cos(x^{-\gamma}r)}{(r-1)^\beta}\,{\rm d}r&=\int^\infty_0\frac{\cos(x^{-\gamma}r+x^{-\gamma})}{r^\beta}\,{\rm d}r\\
&=\cos(x^{-\gamma})\int^\infty_0\frac{\cos(x^{-\gamma}r)}{r^\beta}\,{\rm d}r-\sin(x^{-\gamma})\int^\infty_0\frac{\sin(x^{-\gamma}r)}{r^\beta}\,{\rm d}r\\
&=x^{\gamma-\gamma\beta}\cos(x^{-\gamma})\int^\infty_0\frac{\cos(r)}{r^\beta}\,{\rm d}r-x^{\gamma-\gamma\beta}\sin(x^{-\gamma})\int^\infty_0\frac{\sin(r)}{r^\beta}\,{\rm d}r\\
&=x^{\gamma-\gamma\beta}\Gamma(1-\beta)\bigg(\cos(x^{-\gamma})\sin\Big(\frac{\pi\beta}2\Big)-\sin(x^{-\gamma})\cos\Big(\frac{\pi\beta}2\Big)\bigg)\\
&=x^{\gamma-\gamma\beta}\Gamma(1-\beta)\sin\Big(\frac{\pi\beta}2-x^{-\gamma}\Big).
\end{aligned}
\end{equation*}
In conclusion
\begin{equation*}
\phi'(x)=\gamma^\beta\frac{\Gamma(1\!-\!\beta)}{x^{1+\gamma\beta}}\sin\Big(\frac1{x^\gamma}-\frac{\pi\beta}2\Big)+O\Big(\frac1x\Big),
\end{equation*}
where $\big\vert O(1/x)\big\vert\le  V_1^\infty(\eta)/x$.
\end{proof}

\begin{lemma}
\label{eta of bounded variation}
For $\alpha,\gamma>0$, denote $\eta(r)=r^{-(1+\alpha)/\gamma}(1-r^{-1/\gamma})^{-\beta}-\gamma^\beta(r-1)^{-\beta}$, then $\eta\in C^1(1,\infty)$, $\lim\limits_{r\rightarrow1}\eta(r)=\lim\limits_{r\rightarrow\infty}\eta(r)=0$, and $V_1^\infty(\eta)=\int_1^\infty\big\vert\eta'(r)\big\vert\,{\rm d}r<\infty$.
\end{lemma}

\begin{proof}
Obviously $\eta\in C^1(1,\infty)$, and $\eta(r)\rightarrow0$ as $r\rightarrow\infty$. By Taylor expansion at $r=1$, we obtain
\begin{equation*}
\eta(r)=\Big(\frac{1+\gamma}2\beta-1-\alpha\Big)\gamma^{\beta-1}(r-1)^{1-\beta}+O\big((r-1)^{2-\beta}\big).
\end{equation*}
So we know $\eta$ vanishes at both 1 and $\infty$. Now let us substitute $r$ with $s^{-\gamma}$, i.e., for $s\in(0,1)$ define $\tilde\eta(s)=\eta(s^{-\gamma})=s^{1+\alpha}(1-s)^{-\beta}-\gamma^\beta(s^{-\gamma}-1)^{-\beta}$. Then $\tilde\eta\in C^1(0,1)$ and
\begin{equation*}
\tilde\eta'(s)=s^\alpha\frac{1+\alpha}{(1-s)^\beta}+\beta\frac{s^{1+\alpha}}{(1-s)^{1+\beta}}-\beta\frac{\gamma^{1+\beta}s^{-1-\gamma}}{(s^{-\gamma}-1)^{1+\beta}}.
\end{equation*}
On $(0,1/2)$, we can bound $\vert\tilde\eta'\vert$ by an integrable function,
\begin{equation*}
\vert\tilde\eta'(s)\vert\le \frac{1+\alpha+\beta}{2^{\alpha-\beta}}+\frac{\beta\gamma^{1+\beta}s^{\gamma\beta-1}}{(1-2^{-\gamma})^{1+\beta}},\quad s\in\Big(0,\frac12\Big).
\end{equation*}
By Taylor expansion at $s=1$, we obtain
\begin{equation*}
\begin{aligned}
\tilde\eta'(s)=&\frac{2(\alpha\!+\!1)\!-\!\beta(\gamma\!+\!1)}2(1\!-\!\beta)(1-s)^{-\beta}\\
&+\frac{\beta(\gamma\!+\!1)\big(3\beta(\gamma\!+\!1)\!-\!\gamma\!-\!5\big)\!-\!12\alpha(\alpha\!+\!1)}{24}(2\!-\!\beta)(1-s)^{1-\beta}+O\big((1-s)^{2-\beta}\big).
\end{aligned}
\end{equation*}
Note that $\tilde\eta'\in C(0,1)$, we know $\tilde\eta'$ is absolutely integrable on $(0,1)$ according to the above estimates. Therefore $V_0^1(\tilde\eta)=\int_0^1\big\vert\tilde\eta'(s)\big\vert\,{\rm d}s$ is a finite number which only depends on $\alpha,\beta,\gamma$. Since the change of variables from $r$ to $s$ is monotone, $\eta$ and $\tilde\eta$ have the same total variation which is bounded.
\end{proof}

\section{Uniform bounds from maximum principle arguments}
\subsection{Bounds of Caputo relaxation solutions (Mittag-Leffler and Kilbas–Saigo functions)}\label{App:C} 
We define the Caputo derivative as $\DCp\beta u:=\DRL\beta[u-u(0)]$.

\begin{proposition}
\label{CaputoRelaxation}
For $\lambda<0$ and $u_0=1$, the series \eqref{Caputo relaxation solution}, which solves $\DCp\beta u =\lambda u $, has the following bounds for all $x\ge0$,
\[
\frac{1}{1+|\lambda|\Gamma(1-\beta)x^{\beta}}\le u(x)\le \frac{1}{1+|\lambda|\Gamma(1+\beta)^{-1}x^{\beta}}.
\]
\end{proposition}

\begin{proof}
 Obviously the $u$ defined by \eqref{Caputo relaxation solution} is in $C^1(0,\infty)\cap C[0,\infty)$, so the Caputo derivative of $u$  allows the representation 
\begin{equation}
\label{representation}
\DCp\beta u(x)=\int_0^x\big(u(x)-u(x-r)\big)\frac{r^{-1-\beta}}{|\Gamma(-\beta)|}\,{\rm d}r+\frac{u(x)-u(0)}{x^\beta\beta|\Gamma(-\beta)|},\quad x>0.
\end{equation}

Then Proposition \ref{CaputoRelaxation} is immediate from Lemma \ref{lem} and Lemma \ref{lem2} below.
\end{proof}

\begin{lemma}\label{lem}
The solution $u$ in Proposition \ref{CaputoRelaxation} stays positive and is bounded from above by
$v(x)=(1+c \vert\lambda\vert x^\beta)^{-1}$, where  $c =\Gamma(1+\beta)^{-1}$.
 
\end{lemma}

\begin{proof}   The positivity of $u$ can be proven similarly to that in Lemma \ref{upper bound for solutions to linear equation}. As for its upper bound, analogously we let $v(x)=(x^\beta+c)^{-1}$. Now, given $\beta\in(0,1)$ and $\lambda<0$, it remains to find a $c>0$, such that $\DCp\beta v\ge \lambda v$, i.e., for all $x>0$,
\begin{equation*}
\int_0^x\big(v(x)-v(x-r)\big)\frac{ r^{-1-\beta}}{\big|\Gamma(-\beta)\big|}\,{\rm d}r+\frac{v(x)-v(0)}{x^\beta\beta\big|\Gamma(-\beta)\big|}  \ge \lambda v(x) ,
\end{equation*}
or equivalently, for all $x>0$,
\begin{equation}
\label{equivalent inequality Caputo}
\begin{aligned}
\big\vert\lambda\hspace{0.5pt}\Gamma(-\beta)\big\vert&\ge \int_0^x\bigg(\frac{x^\beta+c}{(x-r)^\beta+c}-1\bigg)\frac{{\rm d}r}{r^{1+\beta}}+\beta^{-1}x^{-\beta}\bigg(\frac{x^\beta+c}c-1\bigg),
\end{aligned}
\end{equation}
 where the first summand  can be bounded from above by 
\begin{equation*}
c^{-1}\int_0^x\big(x^\beta-(x-r)^\beta\big)\frac{{\rm d}r}{r^{1+\beta}} 
=c^{-1}\big\vert\Gamma(-\beta)\big\vert\big(\Gamma(1+\beta)-\Gamma(1-\beta)^{-1}\big).
\end{equation*}
In order to fulfil \eqref{equivalent inequality Caputo}, we only need to require $c$ to satisfy $\big\vert\lambda\hspace{0.5pt}\Gamma(-\beta)\big\vert\ge c^{-1}\big\vert\Gamma(-\beta)\big\vert\hspace{0.5pt}\Gamma(1+\beta)$, so we can choose $c=\Gamma(1+\beta)/\vert\lambda\vert$.
\end{proof}

\begin{lemma}\label{lem2}
The solution $u$ in Proposition \ref{CaputoRelaxation} is bounded from below by $w(x)=(1+d\vert\lambda\vert  x^\beta)^{-1}$, where $d =\Gamma(1-\beta)$.
\end{lemma}

\begin{proof}
 Similarly to the proof of Lemma \ref{lem},  given $\beta\in(0,1)$ and $\lambda<0$, it remains to  find a $d>0$ such that $\DCp\beta w \le \lambda w$, where  $w(x)=(x^\beta+d)^{-1}$  , i.e., for all $x>0$, 
\begin{equation*}
\int_0^x\big(w(x)-w(x-r)\big)\frac{ r^{-1-\beta}}{\big|\Gamma(-\beta)\big|}\,{\rm d}r+\frac{w(x)-w(0)}{x^\beta\beta\big|\Gamma(-\beta)\big|}\le \lambda w(x), 
\end{equation*} or equivalently, for all $x>0$,
\begin{equation*}
\big\vert\lambda\hspace{0.5pt}\Gamma(-\beta)\big\vert\le\int_0^x\bigg(\frac{w(x-r)}{w(x)}-1\bigg)\frac{{\rm d}r}{r^{1+\beta}}+\beta^{-1}x^{-\beta}\bigg(\frac{w(0)}{w(x)}-1\bigg).
\end{equation*}
Note that the first  summand  is nonnegative, and the second summand equals $\beta^{-1}x^{-\beta}\big((x^\beta+d)/d-1\big)$ and thus $(d\beta)^{-1}$, so we only need choose such $d$ that $\big\vert\lambda\hspace{0.5pt}\Gamma(-\beta)\big\vert\le(d\beta)^{-1}$, i.e., $d\le\big\vert\lambda\beta\hspace{0.5pt}\Gamma(-\beta)\big\vert^{-1}=\big\vert\lambda\hspace{0.5pt}\Gamma(1\!-\!\beta)\big\vert^{-1}$.
\end{proof}

\begin{proposition}
\label{CaputoRelaxation Kilbas Saigo}
 For $\lambda<0$, $u_0=1$ and $\alpha>0$, the series \eqref{solution to Caputo linear nonconstant equation}, which solves $\DCp\beta u =\lambda x^{\alpha-\beta}u $, has the following bounds for all $x\ge0$,
\[
\frac{1}{1+|\lambda|\hspace{0.5pt}\Gamma(1\!-\!\beta)x^\alpha}\le u(x)\le \frac{1}{1+|\lambda|\hspace{0.5pt}\Gamma(1\!+\!\alpha\!-\!\beta)\Gamma(1\!+\!\alpha)^{-1}x^\alpha}.
\]
\end{proposition}
Note that these two bounds are identical to those in \cite[Proposition 4.12]{BSV19} (To see this, rewrite $u(x)$ as $E_{\beta,\,\alpha/\!\beta,\,\alpha/\!\beta\hspace{-0.5pt}-\hspace{-0.8pt}1}(\lambda x^\alpha)$ following their notation $E_{\,\cdot\,,\,\cdot\,,\,\cdot\,}(\,\cdot\,)$ which is the Kilbas–Saigo function). The proof is parallel to that of Proposition \ref{CaputoRelaxation}, so we omit it. The only difference is that we now need the criterion $\DCp\beta v\ge \lambda x^{\alpha-\beta}v$ to find $v$ of the form $(1+cx^\alpha)^{-1}$ as the upper bound (the lower bound is analogous).
 
\subsection{Bounds of censored relaxation solutions}
\label{maximum principle proofs}
$ $\newline
\begin{proof} \emph{[of Lemma \ref{lower bound for solutions to linear equation}]} 
According to Remark \ref{rmk:oncfd}-(i), we only need to consider a particular $\lambda$. Take $\lambda=-d^{-\beta}$, and our job is to prove that the solution is bounded from below by $w(x)=(1+x)^{-1}(1+x^\beta)^{-1}$, which is convex on $[0,\infty)$. Similar to the proof of Lemma \ref{upper bound for solutions to linear equation}, it remains to prove $\DCs\beta w\le\lambda w$, i.e., for all $x>0$,
\begin{equation*}
\int_0^x\big(w(x)-w(x-r)\big)\frac{r^{-1-\beta}}{\big|\Gamma(-\beta)\big|}{\rm d}r\le\lambda w(x),
\end{equation*}
or equivalently, for all $x>0$,
\begin{equation}
\label{lower bound of linear equation, equivalent condition for d}
\big|\lambda\Gamma(-\beta)\big|\le\int_0^x\frac{{\rm d}r}{r^{1+\beta}}\Big(\frac{w(x-r)}{w(x)}-1\Big).
\end{equation}

For $x>2$, the right-hand side of \eqref{lower bound of linear equation, equivalent condition for d} can be bounded from below by
\begin{equation*}
\begin{aligned}
\int_{x-1}^x\frac{{\rm d}r}{r^{1+\beta}}\Big(\frac{w(x\!-\!r)}{w(x)}-1\Big)
\ge\frac1{x^{1+\beta}}\Big(\frac{w(1)}{w(x)}-1\Big)
=\frac1{x^{1+\beta}}\bigg(\frac{1+x^\beta+x+x^{1+\beta}}4-1\bigg)
\ge\frac14.
\end{aligned}
\end{equation*}
For $0<x<2$, since $w$ is convex, $w(x-r)$ is bounded from below by $w(x)-w'(x)r$, so
\begin{equation*}
\frac{w(x-r)}{w(x)}-1\ge r\bigg(\frac1{1+x}+\frac{\beta x^{\beta-1}}{1+x^\beta}\bigg),
\end{equation*}
therefore the right-hand side of \eqref{lower bound of linear equation, equivalent condition for d} can be bounded from below by
\begin{equation*}
\int_0^x\frac{{\rm d}r}{r^\beta}\bigg(\frac1{1+x}+\frac{\beta x^{\beta-1}}{1+x^\beta}\bigg)
\ge\frac{x^{1-\beta}}{1-\beta}\frac{\beta x^{\beta-1}}{1+x^\beta}
\ge\frac{\beta}{(1-\beta)(1+2^\beta)}.
\end{equation*}

Now we only need verify $\lambda=-d^{-\beta}$ satisfies $\big\vert\lambda\Gamma(-\beta)\big\vert\le\min\big\{1/4,\,\beta(1-\beta)^{-1}/(1+2^\beta)\big\}$, which is obvious by the definition of $d$.

\end{proof}

\begin{proof} \emph{[of Lemma \ref{upper bound for solutions to linear nonconstant equation}]}
The positivity of $u$ can be proven similarly to that in Lemma \ref{upper bound for solutions to linear equation}. As for its upper bound, analogously we let $v(x)=1/(x^{1+\alpha}+c)$. Now, given $\beta\in(0,1),\alpha>0$ and $\lambda<0$, it remains to find a $c>0$, such that $\DCs\beta v\ge\lambda x^{\alpha-\beta}v$, i.e., for all $x>0$,
\begin{equation*}
\int_0^x\big(v(x)-v(x-r)\big)\frac{ r^{-1-\beta}}{\big|\Gamma(-\beta)\big|}\,{\rm d}r\ge\lambda x^{\alpha-\beta}v(x),
\end{equation*}
or equivalently, for all $x>0$,
\begin{equation}
\label{equivalent inequality nonconstant}
\begin{aligned}
\big\vert\lambda\Gamma(-\beta)\big\vert\ge&\, x^{\beta-\alpha}\int_0^x\frac{x^{1+\alpha}-(x-r)^{1+\alpha}}{c+(x-r)^{1+\alpha}}\cdot\frac{{\rm d}r}{r^{1+\beta}}\\ =&\, x^{-\alpha}\int_0^1\frac{1-s^{1+\alpha}}{c/x^{1+\alpha}+s^{1+\alpha}}\cdot\frac{{\rm d}s}{(1-s)^{1+\beta}},
\end{aligned}
\end{equation}
where $s=1-r/x$. To obtain an upper bound of the right-hand side of \eqref{equivalent inequality nonconstant}, we split the interval $[0,1]$ into two halves. For the first half interval, we have
\begin{equation*}
\begin{aligned}
\int_0^{1/2}\frac{1-s^{1+\alpha}}{c/x^{1+\alpha}+s^{1+\alpha}}\cdot\frac{{\rm d}s}{(1-s)^{1+\beta}}&\le2^{1+\beta}\int_0^{1/2}\frac{{\rm d}s}{c/x^{1+\alpha}+s^{1+\alpha}}\\
&\le2^{1+\beta+\alpha}\int_0^{1/2}\frac{{\rm d}s}{(c^{(1+\alpha)^{-1}}\!/x+s)^{1+\alpha}}\\
&\le\frac{2^{1+\beta+\alpha}}{\alpha}\frac{x^\alpha}{c^{\alpha/(1+\alpha)}},
\end{aligned}
\end{equation*}
where the second inequality is due to Jensen's inequality $t^{1+\alpha}+s^{1+\alpha}\ge2^{-\alpha}(t+s)^{1+\alpha}$ for $t,s,\alpha>0$. For the second half interval, we have
\begin{equation*}
\begin{aligned}
\int_{1/2}^1\frac{1-s^{1+\alpha}}{c/x^{1+\alpha}+s^{1+\alpha}}\cdot\frac{{\rm d}s}{(1-s)^{1+\beta}}&\le\frac{\alpha^{\alpha/(1+\alpha)}}{1+\alpha}\frac{2x^\alpha}{c^{\alpha/(1+\alpha)}}\int_{1/2}^1\frac{1-s^{1+\alpha}}{(1-s)^{1+\beta}}\,{\rm d}s\\
&\le\frac{\alpha^{\alpha/(1+\alpha)}}{1+\alpha}\frac{2x^\alpha}{c^{\alpha/(1+\alpha)}}\int_{1/2}^1\frac{(1+\alpha)(1-s)}{(1-s)^{1+\beta}}\,{\rm d}s\\
&=\frac{\alpha^{\alpha/(1+\alpha)}}{1-\beta}\frac{2^\beta x^\alpha}{c^{\alpha/(1+\alpha)}},
\end{aligned}
\end{equation*}
where the first inequality is due to Young's inequality for products: $a^{\gamma}b^{1-\gamma}\le \gamma a+(1-\gamma)b$ with $\gamma=\alpha/(1+\alpha),\;a=c(1+\alpha^{-1})/x^{1+\alpha}$ and $b=(1+\alpha)s^{1+\alpha}$; the second inequality is due to the concavity of $1-s^{1+\alpha}$ whose tangent at $s=1$ is $(1+\alpha)(1-s)$.

Therefore the right-hand side of \eqref{equivalent inequality nonconstant} can be bounded from above by
\begin{equation*}
\frac{2^\beta}{c^{\alpha/(1+\alpha)}}\Big(\frac{2^{1+\alpha}}{\alpha}+\frac{\alpha^{\alpha/(1+\alpha)}}{1-\beta}\Big).
\end{equation*}
Let
\begin{equation*}
c\ge\bigg(\frac{2^\beta}{\big\vert\lambda\Gamma(-\beta)\big\vert}\Big(\frac{2^{1+\alpha}}{\alpha}+\frac{\alpha^{\alpha/(1+\alpha)}}{1-\beta}\Big)\bigg)^{1+1/\alpha},
\end{equation*}
then \eqref{equivalent inequality nonconstant} will be satisfied and we are done.

\end{proof}


\section{Derivation of \texorpdfstring{$\mathbb E_x\big[(\tau_j)^n\big]$}{E tau j n} in Remark \ref{rmk:exptau}-(i)}\label{app:C_j_n}
By \eqref{eq:com} and the multinomial formula we have
$$
\mathbb E_x\big[(\tau_j)^n\big]= \mathbb E_x\!\Bigg[\bigg(\sum_{i=1}^j  E_i(S^c_{\tau_{i-1}})\bigg)^{\hspace{-3pt}n}\Bigg] = \hspace{-5pt}\sum_{n_1+...+n_j=n} \frac{n!}{\prod_{i=1}^jn_i!} \,\mathbb E_x\!\Bigg[\prod_{i=1}^jE_i(S^c_{\tau_{i-1}})^{n_i}\Bigg],\;n\in\mathbb N,
$$
where the right-hand side can be calculated using Lemma \ref{lem:iteratedexp}.

\begin{lemma}\label{lem:iteratedexp}
For $\{n_i\}_{i=1}^j\subseteq \mathbb N\cup\{0\}$   and $n=n_1+n_2+...+n_j\;(j\in\mathbb N) $, it holds that
\begin{align}\label{expectation of product of E_i to some powers}
\mathbb E_x\!\Bigg[\prod_{i=1}^jE_i(S^c_{\tau_{i\hspace{-0.5pt}-\!1}})^{n_i}\Bigg] 
&= \frac{x^{\beta n}\prod_{i=1}^{j}\Gamma(n_{i}+1) }{ \Gamma( \beta n_j+1) \Gamma(1-\beta)^{j-1} } \prod_{i=1}^{j-1}\frac{\Gamma\big(1-\beta+\beta\sum_{l=0}^{i-1}n_{j-l}\big)}{\Gamma\big(1+\beta\sum_{l=0}^{i}n_{j-l}\big)}  .
\end{align}
\end{lemma}

\begin{proof}
By Lemma \ref{lem:powerofE1} (with $m=0$), the left-hand side of \eqref{expectation of product of E_i to some powers} equals
\begin{align}\label{left hand side of {expectation of product of E_i to some powers}}
\mathbb E_x\!\bigg[E_j(S^c_{\tau_{j\hspace{-0.5pt}-\!1}})^{n_j}\prod_{i=1}^{j-1}E_i(S^c_{\tau_{i\hspace{-0.5pt}-\!1}})^{n_i}\bigg]
=\mathbb E_x\!\bigg[(S^c_{\tau_{j\hspace{-0.5pt}-\!1}})^{\beta n_j}\prod_{i=1}^{j-1}E_i(S^c_{\tau_{i\hspace{-0.5pt}-\!1}})^{n_i}\bigg]\frac{\Gamma(n_j+1)}{\Gamma(\beta n_j+1)}.
\end{align}
We now prove that the right-hand sides of \eqref{expectation of product of E_i to some powers} and \eqref{left hand side of {expectation of product of E_i to some powers}} are equal, by induction on $j$. When $j=1$, it is obvious; when $j\ge2$, suppose it is true for $j-1$, then we have
\begin{align*}
&\hspace{11pt}\text{The right-hand side of \eqref{left hand side of {expectation of product of E_i to some powers}}}\\
&\hspace{-6pt}=\hspace{5.2pt}\mathbb E_x\!\bigg[\Big( S^c_{\tau_{j\hspace{-0.5pt}-\hspace{-0.5pt}2}}+S^{j-1}_{E_{j\hspace{-0.5pt}-\!1}(S^c_{\tau_{j\hspace{-0.5pt}-\hspace{-0.5pt}2}})-} \Big)^{\!\beta n_j}E_{j\hspace{-0.5pt}-\!1}(S^c_{\tau_{j\hspace{-0.5pt}-\hspace{-0.5pt}2}})^{n_{j\hspace{-0.5pt}-\!1}}\prod_{i=1}^{j-2}E_i(S^c_{\tau_{i\hspace{-0.5pt}-\!1}})^{n_i}\bigg]\frac{\Gamma(n_j+1)}{\Gamma(\beta n_j+1)}\\
=&\,\mathbb E_x\!\Bigg[\mathbb E\!\bigg[\Big( S^c_{\tau_{j\hspace{-0.5pt}-\hspace{-0.5pt}2}}+S^{j-1}_{E_{j\hspace{-0.5pt}-\!1}(S^c_{\tau_{j\hspace{-0.5pt}-\hspace{-0.5pt}2}})-} \Big)^{\!\beta n_j}E_{j\hspace{-0.5pt}-\!1}(S^c_{\tau_{j\hspace{-0.5pt}-\hspace{-0.5pt}2}})^{n_{j\hspace{-0.5pt}-\!1}}\prod_{i=1}^{j-2}E_i(S^c_{\tau_{i\hspace{-0.5pt}-\!1}})^{n_i}\;\bigg\vert\;S^c_{\tau_{j\hspace{-0.5pt}-\hspace{-0.5pt}2}}\bigg]\Bigg]\frac{\Gamma(n_j+1)}{\Gamma(\beta n_j+1)}\\
=&\,\mathbb E_x\!\Bigg[\frac{\Gamma(n_{j\hspace{-0.5pt}-\!1}+1)}{ \Gamma(1-\beta)} \frac{ \Gamma(\beta n_j-\beta+1)}{\Gamma(\beta n_{j\hspace{-0.5pt}-\!1}+\beta n_j+ 1)}  (S^c_{\tau_{j\hspace{-0.5pt}-\hspace{-0.5pt}2}})^{\beta n_{j\hspace{-0.5pt}-\!1}+\beta n_j}\prod_{i=1}^{j-2}E_i(S^c_{\tau_{i\hspace{-0.5pt}-\!1}})^{n_i}\Bigg]\frac{\Gamma(n_j+1)}{\Gamma(\beta n_j+1)}\\
=&\frac{\Gamma(n_{j\hspace{-0.5pt}-\!1}\!+\!1)}{ \Gamma(1-\beta)} \frac{ \Gamma(\beta n_j-\beta+1)}{\Gamma(n_{j\hspace{-0.5pt}-\!1}\!+\!n_j\!+\! 1)}\frac{\Gamma(n_j+1)}{\Gamma(\beta n_j\!+\!1)}\mathbb E_x\!\bigg[\!(S^c_{\tau_{j\hspace{-0.5pt}-\hspace{-0.5pt}2}})^{\beta n_{j\hspace{-0.5pt}-\!1}+\beta n_j}\hspace{-3pt}\prod_{i=1}^{j-2}\hspace{-3pt}E_i(S^c_{\tau_{i\hspace{-0.5pt}-\!1}})^{n_i}\!\bigg]\!\frac{\Gamma(n_{j\hspace{-0.5pt}-\!1}+n_j+1)}{\Gamma(\beta n_{j\hspace{-0.5pt}-\!1}\!+\!\beta n_j\!+\!1)}\\
=&\,\text{The right-hand side of \eqref{expectation of product of E_i to some powers}},
\end{align*}
where the third equality is due to Lemma \ref{lem:powerofE1}, the equality in law between $S^1$ and $S^{j-1}$, as well as the independence of $S^{j-1}$ and $\{S^c_t\}_{t\in[0,\hspace{0.5pt}\tau_{j\hspace{-0.5pt}-\hspace{-0.5pt}2}]}$; the last equality is due to the induction hypothesis.

\end{proof}

\begin{lemma}\label{lem:powerofE1}
For all $n\in\mathbb N\cup \{0\}$ and $\alpha\ge 0$, we have
\begin{equation}
\mathbb E\big[E_1(x)^n (x+S^1_{E_1-})^\alpha\big] = \frac{\Gamma(n+1)}{ \Gamma(1-\beta)} \frac{ \Gamma(\alpha-\beta+1)}{\Gamma(\alpha+ n\beta+1)}  x^{\beta n+\alpha}  .
\label{eq:powerofE1X_1}
\end{equation}
 
\end{lemma}
\begin{proof} \emph{[suggested by Andreas Kyprianou and Jean Bertoin]}
Let $p_s$ be the density of the $\beta$-stable subordinator $-S^1$ at time $s>0$, and recall its Mellin transform for each $n\in \mathbb N$,
\begin{equation}
\int_0^\infty \frac{s^{n-1}}{\Gamma(n)} p_s(x)\,\dd s =\frac{x^{n\beta-1}}{\Gamma(\beta n)},\;x>0.
\label{eq:ident_lap}
\end{equation}  
Applying the compensation formula as in the proof of \cite[Proposition III.2-(i)]{bertoin} with $f(x)=x^n$ and $g(x)=x^\alpha$, we obtain
\begin{align*}
\mathbb E\Big[f\big(E_1(x)\big)g\big(x+S^1_{E_1(x)-}\big)\Big]&=\mathbb E\bigg[\raisebox{3pt}{$\sum\limits_{t\ge 0}f(t)g(x+S^1_{t-})\mathbf 1_{\{-S^1_{t-}<x\}}\mathbf 1_{\{-\Delta S_t>x+S^1_{t-}\}}$}\bigg]\\ 
&=\int_0^\infty\mathbb E\Big[f(t)g(x+S^1_{t-})\mathbf 1_{\{-S^1_{t-}<x\}}\int_{x+S^1_{t-}}^\infty \frac{z^{-1-\beta}}{\big|\Gamma(-\beta)\big|} \dd {z}\Big]\dd{t}\\
&=\int_0^\infty \int_0^xf(t)g(x-y)p_t(y)\int_{x-y}^\infty  \frac{z^{-1-\beta}}{\big|\Gamma(-\beta)\big|} \dd z\,\dd y\,\dd t \\
&=\int_0^x(x-y)^\alpha\frac{(x-y)^{-\beta}}{\Gamma(1-\beta)}\int_0^\infty t^n p_t(y) \dd{t}\dd{y}\\
&=\int_0^x\frac{(x-y)^{\alpha-\beta}}{\Gamma(1-\beta)}\Gamma(n+1)\frac{y^{n\beta+\beta-1}}{\Gamma(n\beta+\beta)}\dd{y}\\
&=x^{n\beta+\alpha }\frac{ \Gamma(\alpha-\beta+1)}{\Gamma(\alpha+\beta n+1)} \frac{\Gamma(n+1)}{ \Gamma(1-\beta)}.
\end{align*}

\end{proof}


\let\oldbibliography\thebibliography
\renewcommand{\thebibliography}[1]{\oldbibliography{#1}
\setlength{\itemsep}{0pt}}

\end{document}